%% file: hyperbolic-arxiv.tex
\documentclass[10pt]{amsart}

\usepackage[latin1]{inputenc}
\usepackage{amsmath, amsfonts, amssymb, amsthm, mathrsfs}
\usepackage{graphicx}
\usepackage[all]{xy}
 \usepackage{enumerate}

\usepackage{pgf,tikz}
\usepackage{mathrsfs}
\usetikzlibrary{arrows}
\usetikzlibrary{decorations.markings}
\usepackage{morefloats}

\newtheorem{prop}{Proposition}
\newtheorem{thm}{Theorem}
\newtheorem*{mainthm}{Main Theorem}
\newtheorem{lemma}{Lemma}
\newtheorem{cor}{Corollary}

\theoremstyle{remark}
\newtheorem{remark}{Remark}
\theoremstyle{definition}
\newtheorem{definition}{Definition}

\newcommand{\C}{{\mathbb{C}}}
\newcommand{\R}{{\mathbb{R}}}
\newcommand{\Z}{{\mathbb{Z}}}
\newcommand{\hzero}{\operatorname{H}_0^\mathfrak{Fin}(\Gamma; \thinspace R_\C)}
\newcommand{\hone}{\operatorname{H}_1^\mathfrak{Fin}(\Gamma; \thinspace R_\C)}
\newcommand{\hn}{\operatorname{H}_n^\mathfrak{Fin}(\Gamma; \thinspace R_\C)}
\newcommand{\hp}{\operatorname{H}_p^\mathfrak{Fin}(\Gamma; \thinspace R_\C)}
\newcommand{\prerank}{\operatorname{pre-rank}(\partial_1)}
\newcommand{\rank}{\operatorname{rank}_{\mathbb{Z}}}

\title{Equivariant K-homology for hyperbolic reflection groups}
\author[Lafont \and Ortiz \and Rahm \and S\'anchez-Garc\'ia]{Jean-Fran\c{c}ois Lafont \and Ivonne J.~Ortiz \and Alexander D. Rahm \and Rub\'en J.~S\'anchez-Garc\'ia}
\date{\today}

\begin{document}

\begin{abstract}
We compute the equivariant $K$-homology of the classifying space for proper actions, for cocompact 3-dimensional hyperbolic reflection groups. 
This coincides with the topological $K$-theory of the reduced $C^\ast$-algebra associated to the group, via the Baum-Connes conjecture. 
We show that, for any such reflection group, the associated $K$-theory groups are torsion-free. As a result we can promote previous 
rational computations to integral computations. Our proof relies on a new efficient algebraic criterion for checking torsion-freeness of 
$K$-theory groups, which could be applied to many other classes of groups. 
\end{abstract}

\maketitle


\section{Introduction}

For a discrete group $\Gamma$, a general problem is to compute $K_*(C^*_{r}\Gamma)$, the topological $K$-theory of the reduced
$C^*$-algebra of $\Gamma$. The Baum-Connes Conjecture predicts that this functor can be determined, in a homological manner, from
the complex representation rings of the finite subgroups of $\Gamma$. This viewpoint led to general recipes for computing the rational
topological $K$-theory $K_*(C^*_{r}\Gamma)\otimes \mathbb Q$ of groups, through the use of Chern characters (see for instance
L\"uck and Oliver~\cite{LO01} and L\"uck~\cite{L02},~\cite{L07}, as well as related earlier work of Adem~\cite{Adem93}). 
When $\Gamma$ has small homological dimension, 
one can sometimes even give completely explicit formulas for the rational topological $K$-theory, see for instance 
Lafont, Ortiz, and S\'anchez-Garc\'ia~\cite{LOS12} for the case where $\Gamma$ is a $3$-orbifold group.

On the other hand, performing integral calculations for these $K$-theory groups is much harder. For $2$-dimensional crystallographic
groups, such calculations have been done in M. Yang's thesis~\cite{Y97}. This was subsequently extended to the class of cocompact planar groups by
L\"uck and Stamm~\cite{LS00}, and to certain higher dimensional dimensional crystallographic groups by Davis and L\"uck~\cite{DL13}
(see also Langer and L\"uck~\cite{LL12}).
For $3$-dimensional groups, L\"uck~\cite{L05} completed this calculation for the semi-direct product $\text{Hei}_3(\mathbb Z) \rtimes
\mathbb Z_4$ of the $3$-dimensional integral Heisenberg group with a specific action of the cyclic group $\mathbb Z_4$. Some further computations were completed 
by Isely~\cite{I11} for groups of the form $\mathbb Z^2 \rtimes \mathbb Z$; 
by Rahm~\cite{Rahm16} for the class of Bianchi groups;
by Pooya and Valette~\cite{PV16} for solvable Baumslag-Solitar groups;
and by Flores, Pooya and Valette~\cite{FPV16} for lamplighter groups of finite groups. 

Our present paper has two main goals. Our first goal is to add to the list of examples, by providing a formula for the 
integral $K$-theory groups of cocompact 3-dimensional hyperbolic reflection groups.

\begin{mainthm}
Let $\Gamma$ be a cocompact $3$-dimensional hyperbolic reflection group, generated by reflections in the side of a hyperbolic
polyhedron $\mathcal{P} \subset \mathbb{H}^3$. Then 
$$ K_0(C_r^*(\Gamma)) \cong \mathbb Z^{cf(\Gamma)} \ \text{ and }\  K_1(C_r^*(\Gamma)) \cong \mathbb Z^{cf(\Gamma) - \chi(\mathcal C)},$$
where the integers $cf(\Gamma), \chi(\mathcal C)$ can be explicitly computed from the combinatorics of the polyhedron $\mathcal P$.
\end{mainthm}

\vskip 10pt

Here, $cf(\Gamma)$ denotes the number of conjugacy classes of elements of finite order in $\Gamma$, and $\chi(\mathcal C)$ denotes the Euler characteristic of the Bredon chain complex. By a celebrated result of Andre'ev~\cite{Andreev}, there is a simple algorithm that inputs a Coxeter group $\Gamma$, 
and decides whether or not there exists a hyperbolic polyhedron $P_{\Gamma}\subset \mathbb H^3$ which generates $\Gamma$. 
In particular, given an arbitrary Coxeter group, one can easily verify if it satisfies the hypotheses of our {Main Theorem}.

\vskip 10pt

Note that the lack of torsion in the $K$-theory is not a property shared by all discrete groups acting on hyperbolic $3$-space. For example, $2$-torsion occurs in $K_0(C_r^*(\Gamma))$ whenever $\Gamma$ is a Bianchi group 
containing a $2$-dihedral subgroup $C_2 \times C_2$ (see ~\cite{Rahm16}). In fact, the key difficulty in the proof of our Main Theorem lies 
in showing that these $K$-theory groups are torsion-free. Some previous integral computations yielded $K$-theory groups that are 
torsion-free, though in those papers the torsion-freeness was a consequence of ad-hoc computations. 
Our second goal is to give a general criterion which explains the lack of torsion, 
and can be efficiently checked. This allows a systematic, algorithmic approach to the question of whether a $K$-theory group
is torsion-free. 

\vskip 10pt

Let us briefly describe the contents of the paper. In Section~\ref{background}, we provide background material on hyperbolic
reflection groups, topological $K$-theory, and the Baum-Connes Conjecture. We also introduce our main tool, the Atiyah-Hirzebruch
type spectral sequence. In Section~\ref{Bredon-analysis}, 
we use the spectral sequence to show that the $K$-theory groups we are interested in coincide with the 
Bredon homology groups $\hzero$ and $\hone$ respectively. We also explain, using the $\Gamma$-action on $\mathbb H^3$,
why the homology group $\hone$ is torsion-free. In contrast, showing that $\hzero$
is torsion-free is much more difficult. 
In Section~\ref{sec:H0-geom}, we give a geometric proof for this fact in a restricted setting.
In Section~\ref{sec:H0-algebra}, we give a linear algebraic proof in the general case, 
inspired by the ``representation ring splitting'' technique of~\cite{Rahm16}. In particular,
we establish a novel criterion (Theorem~\ref{thm:basetransf}) for verifying that $\hzero$ is torsion-free, 
for any collection of groups $\Gamma$ with prescribed types of finite subgroups. 
In Sections ~\ref{Heisenberg} and ~\ref{crystallographic}, we further illustrate this criterion by applying it to the 
Heisenberg semidirect product group of ~\cite{L05} and to the crystallographic groups of~\cite{DL13} respectively.
As the rank of the $K$-theory groups can be easily computed, this gives an alternate proof of the integral $K$-theoretic 
computations in \cite{L05} and \cite{DL13}.
Finally, in Section~\ref{sec:geom}, we return to our Coxeter groups, and provide an explicit formula for the rank of the 
Bredon homology groups (and hence for the $K$-groups we are interested in), in terms of the combinatorics of the 
polyhedron $\mathcal P$. 

\vskip 5pt

Our paper concludes with two fairly long appendices, which contain all the character tables and induction 
homomorphisms used in our proofs, making our results explicit and self-contained. 

\newpage

\centerline {\bf Acknowledgments}

\vskip 5pt

Portions of this work were carried out during multiple collaborative visits at Ohio State University, Miami University, and the University of Southampton. The authors would like to thank these institutions for their hospitality. The authors would also like to thank the anonymous
referee for several helpful suggestions.

Lafont was partly supported by the NSF, under grant DMS-1510640. Ortiz was partly supported by the NSF, under grant DMS-1207712.
Rahm was supported by Gabor Wiese's University of Luxembourg grant AMFOR.


\section{Background Material}\label{background}


\subsection{$K$-theory and the Baum-Connes Conjecture}

Associated to a discrete group $\Gamma$, one has $C^*_{r}\Gamma$, the {\it reduced $C^*$-algebra} of $\Gamma$. This algebra
is defined to be the closure, with respect to the operator norm, of the linear span of the image of the regular representation
$\lambda: \Gamma \rightarrow B(l^2(\Gamma))$ of $\Gamma$ on the Hilbert space $l^2(\Gamma)$ of square-summable
complex valued functions on $\Gamma$.  This algebra encodes various analytic properties of the group $\Gamma$~\cite{MV03}.

For a $C^*$-algebra $A$, one can define the {\it topological $K$-theory}
groups $K_*(A):= \pi_{*-1}(GL(A))$, which are the homotopy groups of the space $GL(A)$ of invertible matrices with entries in $A$. 
Due to Bott periodicity, there are canonical isomorphisms $K_*(A) \cong K_{*+2}(A)$, and thus it is sufficient to consider $K_0(A)$ and
$K_1(A)$. 

In the special case where $A=C^*_{r}\Gamma$, the {\it Baum-Connes Conjecture} predicts that there is a canonical isomorphism
$K_n^\Gamma (X)\rightarrow K_n(C_r^*(\Gamma))$, where $X$ is a model for $\underline{E}\Gamma$
(the classifying space for $\Gamma$-actions with isotropy in the family of finite subgroups), and $K_*^\Gamma(-)$
is the equivariant $K$-homology functor. The Baum-Connes conjecture has been verified for many classes of groups. We refer
the interested reader to the monograph by Mislin and Valette~\cite{MV03}, or the survey article by L\"uck and Reich~\cite{LR05} for more 
information on these topics.


\subsection{Hyperbolic reflection groups}

We will assume some familiarity with the geometry and topology of Coxeter groups, which the reader can obtain from Davis' 
book~\cite{Davis}. By a $d$-dimensional \emph{hyperbolic polyhedron}, we mean a bounded region of hyperbolic $d$-space $\mathbb{H}^d$ 
delimited by a given finite number of (geodesic) hyperplanes, that is, the intersection of a collection of half-spaces associated to the hyperplanes.
Let $\mathcal{P} \subset \mathbb{H}^d$ be a polyhedron such that all the interior angles between intersecting faces are of the form 
$\pi/m_{ij}$, where the $m_{ij} \ge 2$ are integer (although some pairs of faces may be disjoint).
Let $\Gamma = \Gamma_\mathcal{P}$ be the associated Coxeter group, generated by reflections in the hyperplanes containing the faces 
of $\mathcal{P}$. 

The $\Gamma$-space $\mathbb{H}^d$ is then a model for $\underline{E}\Gamma$, with fundamental domain $\mathcal{P}$. 
This is a strict fundamental domain -- no further points of $\mathcal{P}$ are identified under the group action -- and hence $\mathcal P = \Gamma\backslash{}\mathbb H^d$. Recall that $\Gamma$ admits the following Coxeter presentation:
\begin{equation}\label{CoxeterPresentation}
	\Gamma = \langle s_1, \ldots, s_n \;|\; (s_i s_j)^{m_{ij}} \rangle ,
\end{equation}
where $n$ is the number of distinct hyperplanes enclosing $\mathcal P$, $s_i$ denotes the reflection on the $i^{\text{th}}$ face, and $m_{ij} \ge 2$ are integers such that: $m_{ii}=1$ for all $i$, and, if $i \neq j$, the corresponding faces meet with interior angle $\pi/m_{ij}$. We will write $m_{ij}=\infty$ if the corresponding faces do not intersect. 
For the rest of this article, $d=3$, and $X$ is $\mathbb{H}^3$ with the $\Gamma$-action described above, with fundamental domain $\mathcal P$.


\subsection{Cell structure of the orbit space}
Fix an ordering of the faces of $\mathcal P$ with indexing set $J=\{1, \ldots, n\}$. We will write $\langle S \rangle$ for the subgroup generated by a subset $S \subset \Gamma$. At a vertex of $\mathcal P$, the concurrent faces 
(a minimum of three) must generate a reflection group acting on the 2-sphere, hence it must be a spherical triangle group. This forces the number of incident faces to be exactly three. 
The only finite Coxeter groups acting by reflections on $S^2$ are the triangle groups $\Delta(2,2,m)$ for some $m \ge 2$, $\Delta(2,3,3)$, 
$\Delta(2,3,4)$ and $\Delta(2,3,5)$, where we use the notation
\begin{equation}\label{eqn:TriangleGroups}
	\Delta(p,q,r) = \left\langle s_1, s_2, s_3 \; | \; s_1^2,s_2^2,s_3^2,(s_1s_2)^p,(s_1s_3)^q,(s_2s_3)^r \right\rangle .
\end{equation}
From our compact polyhedron $\mathcal P$, we obtain an induced $\Gamma$-CW-structure on $X=\mathbb H^3$ with: 
\begin{itemize}
\item one orbit of 3-cells, with trivial stabilizer; 
\item $n$ orbits of 2-cells (faces) with stabilizers $\langle s_i \; | \; s_i^2 \rangle \cong \mathbb Z/2$ ($i=1,\ldots,n$); 
\item one orbit of 1-cells (edges) for each unordered pair $i,j \in J$ with $m_{ij} \neq \infty$, with stabilizer a dihedral group $D_{m_{ij}}$ 
--- this group structure can be read off straight from the Coxeter presentation $\langle s_i, s_j \; | \; s_i^2, s_j^2, (s_is_j)^{m_{ij}}\rangle$;  
\item one orbit of 0-cells (vertices) per unordered triple $i,j,k \in J$ with $\langle s_i, s_j, s_k \rangle$ finite, with stabilizer the triangle group $\langle s_i, s_j, s_k \rangle \cong \Delta(m_{ij}, m_{ik}, m_{jk})$.
\end{itemize}
We introduce the following notation for the simplices of $\mathcal P$: 
\begin{eqnarray}\label{eqn:Notation}
f_i & \ \text{ (faces),}\hspace{2.25cm}\nonumber\\
e_{ij}&=f_i \cap f_j \ \text{ (edges)}, \hspace{3.36cm} \\
v_{ijk}&=f_i\cap f_j \cap f_k = e_{ij}\cap e_{ik} \cap e_{jk} \ \text{ (vertices)},\nonumber 
\end{eqnarray}
whenever the intersections are non-empty.


\subsection{A spectral sequence}
We ultimately want to compute the $K$-theory groups of the reduced $C^*$-algebra of $\Gamma$ via the Baum-Connes conjecture. Note that the conjecture holds for these groups: Coxeter groups have the Haagerup property~\cite{BJS88} and hence satisfy Baum-Connes~\cite{HK97}. Therefore, it suffices to compute the equivariant $K$-homology groups $K^\Gamma_*(X)$, since $X$ is a model of $\underline{E}\Gamma$. 
In turn, these groups can be obtained from the Bredon homology of $X$, calculated via an equivariant Atiyah-Hirzebruch spectral sequence coming from the skeletal filtration of the 
$\Gamma$-CW-complex~$X$ (cf.~\cite{Mislin03}). The second page of this spectral sequence is given by the Bredon homology groups 
\begin{equation}\label{SpectralSequence}
	E^2_{p,q} = \begin{cases} \hp & q \text{ even}, \\
								0 & q \text{ odd}.\end{cases}
\end{equation}
The coefficients $R_\C$ of the Bredon homology groups 
are given by the complex representation ring of the cell stabilizers, which are finite subgroups. 
In order to simplify notation, we will often write $H_p$ to denote $\hp$. In our case $\dim(X)=3$, so the Atiyah-Hirzebruch spectral sequence 
is particularly easy to analyze, and gives the following: 

\begin{prop}\label{prop:ses}
There are short exact sequences
\[
	\xymatrix{0 \ar[r] & \textup{coker}(d^3_{3,0}) \ar[r]\ & K_0^\Gamma(X) \ar[r] & H_2 \ar[r] & 0}
\]
and
\[
	\xymatrix{0 \ar[r] & H_1 \ar[r]\ & K_1^\Gamma(X) \ar[r] & \textup{ker}(d^3_{3,0}) \ar[r] & 0},
\]
where $d^3_{3,0} \colon E^3_{3,0}=H_3 \longrightarrow E^3_{0,2}=H_2$ is the differential on the $E^3$-page of the
Atiyah-Hirzebruch spectral sequence.
\end{prop}

\begin{proof}
This follows at once from a result of Mislin \cite[Theorem 5.29]{Mislin03}. 
\end{proof}


\subsection{Bredon Homology}
To lighten the notation, we write $\Gamma_e$ for the stabilizer in $\Gamma$ of the cell $e$. The Bredon homology groups in Equation (\ref{SpectralSequence}) 
are defined to be the homology groups of the following chain complex:
\begin{equation}\label{eqn:BredonChainComplex}
	\xymatrix{\ldots \ar[r] & \bigoplus_{e \in I_{d}} R_{\mathbb{C}} \left( \Gamma_e \right) \ar[r]^-{\partial_d} & \bigoplus_{e \in I_{d-1}} R_{\mathbb{C}} \left( \Gamma_e \right) \ar[r] &  \ldots},
\end{equation}	
where $I_d$ is a set of orbit representatives of $d$-cells ($d \ge 0$) in $X$. The differentials
$\partial_d$ are defined via the geometric boundary map and induction between representation rings. More precisely, 
if $ge'$ is in the boundary of $e$ ($e \in I_d$, $e' \in I_{d-1}$, $g \in \Gamma$), then $\partial$ restricted to 
$R_{\mathbb{C}} \left( \Gamma_e \right) \rightarrow R_{\mathbb{C}} \left( \Gamma_{e'} \right)$ is given by the composition
$$
	\xymatrix{
		R_{\mathbb{C}} \left( \Gamma_e \right) \ar[r]^-{\text{ind}} & 
		R_{\mathbb{C}} \left( \Gamma_{g e'} \right) \ar[r]^-{\cong} & 
		R_{\mathbb{C}} \left( \Gamma_{e'} \right),
		}
$$	
where the first map is the induction homomorphism of representation rings associated to the subgroup inclusion $\Gamma_{e} \subset \Gamma_{g e'}$, and the second is the isomorphism induced by conjugation $\Gamma_{g e'} = g \Gamma_{e'} g^{-1}$. Finally, we add a sign depending on a chosen (and thereafter fixed) orientation on the faces of ${\mathcal{P}}$. The value $\partial_d (e)$ equals the sum of these maps over all boundary cells of $e$.

Since ${\mathcal{P}}$ is a strict fundamental domain, we can choose the faces of ${\mathcal{P}}$ as orbit representatives. With this choice
of orbit representatives, the $g$ in the previous paragraph is always the identity. We will implicitly make this assumption from now on.


\section{Analyzing the Bredon chain complex for $\Gamma$} \label{Bredon-analysis}

Let $S = \{ s_i \, : \,  1\le i \le n\}$ be the set of Coxeter generators and $J = \{ 1, \ldots, n\}$. Since $X$ is 3-dimensional, the Bredon
chain complex associated to $X$ reduces to
\begin{equation}\label{BredonChainComplex}
	\xymatrix{0 \ar[r] & \mathcal{C}_3 \ar[r]^-{\partial_3} & \mathcal{C}_2 \ar[r]^-{\partial_2} & \mathcal{C}_1 \ar[r]^-{\partial_1} & \mathcal{C}_0 \ar[r] & 0}.
\end{equation}
We now have to analyze the differentials in the above chain complex. 
Recall that for a finite group $G$,
the complex representation ring $R_\mathbb{C}(G)$ is defined as the free abelian group with basis the set of irreducible representations of $G$ 
(the ring structure is not relevant in this setting). Hence $R_\mathbb{C}(G) \cong \mathbb{Z}^{c(G)}$, where we write $c(G)$ for the set of conjugacy classes in $G$.


\subsection{Analysis of $\partial_3$}

Let $G$ be a finite group with irreducible representations $\rho_1, \ldots, \rho_m$ of degree $n_1, \ldots, n_m$, and $\tau$ the only representation of the trivial subgroup $\{1_G\} \le G$. Then $\tau$ induces the regular representation of $G$:
\begin{equation}\label{eqn:RegularRep}
	\text{Ind}_{\{1_G\}}^G(\tau) = n_1\rho_1 + \ldots + n_m\rho_m \,.
\end{equation}
\begin{lemma}
Let $X$ be a $\Gamma$-CW-complex with finite stabilizers, and $k \in \mathbb N$. If there is a unique orbit of $k$-cells and this orbit has trivial stabilizer, then $H_k=0$, provided that $\partial_k \neq 0$.
\end{lemma}
\begin{proof}
The Bredon module $\mathcal{C}_k$ equals $R_\mathbb{C}(\langle 1 \rangle) \cong \mathbb Z$ with generator $\tau$, the trivial representation. Then $\partial_k(\tau) \neq 0$ implies $\ker(\partial_k)=0$;
and therefore the corresponding homology group vanishes. 
\end{proof}

From the lemma, we can easily see that $H_3=0$ if $\partial_3 \neq 0$. Indeed, for $\partial_3 = 0$ to occur, one would need all boundary faces of $\mathcal P$ to be pairwise identified. This cannot happen since $\mathcal P$ is a strict fundamental domain -- the group acts by reflections on the faces. The Lemma then forces $H_3=0$, and Proposition~\ref{prop:ses} gives us:
\begin{cor}\label{prop:analysis3}
We have $K_1^\Gamma(X) \cong H_1$, and there is a short exact sequence
\[
	\xymatrix{0 \ar[r] & H_0 \ar[r]\ & K_0^\Gamma(X) \ar[r] & H_2 \ar[r] & 0}.
\]
\end{cor}


\subsection{Analysis of $\partial_2$}\label{subsect:partial2}

Let $f$ be a face of ${\mathcal{P}}$ and $e \in \partial f$ an edge. 
Suppose, using the notation of Equation~(\ref{eqn:Notation}), that $f=f_i$ and $e=e_{ij}$. 
Then we have a map $R_{\mathbb{C}} \left( \langle s_i \rangle \right) \to R_{\mathbb{C}} \left( \langle s_i, s_j \rangle \right)$ induced by inclusion. 
Recall that $\langle s_i \rangle \cong C_2$ and $\langle s_i, s_j \rangle \cong D_{m_{ij}}$. 
Denote the characters of these two finite groups as specified in Tables~\ref{CharTableC2}
and~\ref{CharTableDm}; 
and denote by a character name with the suffix ``$\uparrow$'' 
the character induced in the ambient larger group.
\begin{table}[!ht]
\[
	\begin{array}{c|rr}
		 C_2 & e & s_i\\ 
		 \hline 
		 \rho_1 & 1 & 1\\
		 \rho_2 & 1 & -1
	\end{array}
\] 
\medskip\caption{Character table of $\langle s_i \rangle \cong C_2$.}\label{CharTableC2}
\end{table}
\begin{table}[!ht]
\[
	\begin{array}{c|cc}
		 D_m & (s_is_j)^r & s_j(s_is_j)^r\\ 
		 \hline 
		 \chi_1 & 1 &  1 \\
		 \chi_2 & 1 & -1 \\
		 \widehat{\chi_3} & (-1)^r &  (-1)^r \\
		 \widehat{\chi_4} & (-1)^r &  (-1)^{r+1} \\
		 \phi_p & 2 \cos\left( \frac{2 \pi p r}{m} \right) &  0 \\
	\end{array}
\]
\medskip\caption{Character table of $\langle_i,s_j\rangle \cong D_m$, where $i<j$ and 
$0 \le r \le m-1$, while $1\leq p \leq m/2-1$ if $m$ is even, and $1\leq p\leq (m-1)/2$ if $m$ is odd, 
and where the hat $\ \widehat{ }\ $ denotes a character which appears only when $m$ is even.}
\label{CharTableDm}
\end{table} 

A straightforward analysis (provided in Section~\ref{Rank2inducedMaps} 
of Appendix~\ref{AppendixB}) shows that
\begin{eqnarray*}
	\rho_1 \uparrow & = & \chi_1 + \widehat{\chi_4} + \sum \phi_p\,,\\
	\rho_2 \uparrow & = & \chi_2 + \widehat{\chi_3} + \sum \phi_p\,,
\end{eqnarray*}
if $i < j$, or
\begin{eqnarray*}
	\rho_1 \uparrow & = & \chi_1 + \widehat{\chi_3} + \sum \phi_p\,,\\
	\rho_2 \uparrow & = & \chi_2 + \widehat{\chi_4} + \sum \phi_p\,,
\end{eqnarray*}
if $j < i$. 
Thus the induction map on the representation rings is the morphism of free abelian groups $ \mathbb{Z}^2 \to \mathbb{Z}^{c( D_{m_{ij}} )}$ given by
\begin{eqnarray*}
	(a,b) &\mapsto & \pm(a, b, \widehat{b}, \widehat{a}, a+b, \ldots, a+b) \quad \text{or}\\
	(a,b) &\mapsto & \pm(a, b, \widehat{a}, \widehat{b}, a+b, \ldots, a+b),
\end{eqnarray*}
where again the hat $\ \widehat{ }\ $ denotes an entry which appears only when ${m_{ij}}$ is even.
Using the analysis above, we can now show the following.
\begin{thm} \label{H2 is zero}
If $\mathcal{P}$ is compact, then $H_2=0$.
\end{thm}
\noindent From this theorem and Corollary~\ref{prop:analysis3}, we immediately obtain:
\begin{cor}\label{cor:KequalsH}
$K_0^\Gamma(X) = H_0$ and $K_1^\Gamma(X)=H_1$.
\end{cor}
\begin{proof}[Proof of Theorem \ref{H2 is zero}]
Fix an orientation on the polyhedron $\mathcal P$, and consider the induced orientations on the faces.
At an edge we have two incident faces $f_i$ and $f_j$ with opposite orientations. 
So without loss of generality we have, as a map of free abelian groups,
\begin{eqnarray}\label{InductionDihedral}
	R_{\mathbb{C}} \left( \langle s_i \rangle \right) \oplus  
		R_{\mathbb{C}} \left( \langle s_j \rangle \right)
		& \rightarrow &  
		R_{\mathbb{C}} \left( \langle s_i, s_j \rangle \right),\\
	(a_i,b_i\,|\,a_j,b_j) & \mapsto & (a_i-a_j, b_i-b_j, \widehat{a_i-b_j}, \widehat{b_i-a_j}, S,\ldots, S), \nonumber
\end{eqnarray}
where $S = a_i+b_i-a_j-b_j$, and the elements with a hat $\widehat{\ \ }$ appear only when $m_{ij}$ is even. Note that we use vertical bars `$|$' for clarity, to separate elements coming from different representation rings.

By the preceding analysis, $\partial_2(x)=0$ implies that, for each $i,j \in J$ such that the corresponding faces $f_i$ and $f_j$ meet, we have
	\begin{enumerate}
		\item $a_i = a_j$ and $b_i = b_j$, and 
		\item $a_i = a_j = b_i = b_j$, if $m_{ij}$ is even.
	\end{enumerate}
\noindent Suppose that $f_1, \ldots, f_n$ are the faces of ${\mathcal{P}}$. Let $x = (a_1, b_1| \ldots |a_n, b_n) \in \mathcal C_2$ be an element in $\text{Ker}(\partial_2)$. Note that $\partial \mathcal{P}$ is connected (since $P$ is homeomorphic to $\mathbb D^3$), so by (1) and (2) above, we have that $a_1= \ldots =a_n$ and $b_1=\ldots =b_n$. Since the stabilizer of a vertex is a spherical triangle group, there is an even $m_{ij}$, which also forces $a=b$. 
Therefore, we have $x=(a,a |\ldots |a,a)$, so $x=\partial_3(a)$ (note that the choice of orientation above forces all signs to be positive).
This yields $\textup{ker}(\partial_2) \subseteq \textup{im}(\partial_3)$, which gives the vanishing of the second homology group.
\end{proof}


\subsection{Analysis of $\partial_1$}

Let $e=e_{ij}$ be an edge and $v = v_{ijk} \in \partial e$ a vertex, using the notation of Equation~\eqref{eqn:Notation}. 
We study all possible induction homomorphisms 
$R_{\mathbb{C}} \left( \langle s_i, s_j \rangle \right) \to R_{\mathbb{C}} \left( \langle s_i, s_j, s_k \rangle \right)$ 
in Appendix~\ref{AppendixB}, and in this section use these computations to verify that $H_1$ is torsion-free.

\begin{thm}\label{thm:H1}
There is no torsion in $H_1$.
\end{thm}

\begin{proof}
Consider the Bredon chain complex 
\[
\xymatrix{\mathcal C_2 \ar[r]^-{\partial_2} & \mathcal C_1 \ar^-{\partial_1}[r] & \mathcal C_0.} 
\]
To prove that $H_1=\textup{ker}(\partial_1)/\textup{im}(\partial_2)$ is torsion-free, it suffices to prove that $\mathcal C_1/\textup{im}(\partial_2)$ is torsion-free. 
Let $\alpha \in \mathcal C_1$ and $0 \neq k\in \mathbb Z$ such that $k\alpha \in \textup{im}(\partial_2)$. We shall prove that $\alpha \in \textup{im}(\partial_2)$. 

Since $k\alpha \in \textup{im}(\partial_2)$, we can find $\beta \in \mathcal C_2$ with $\partial_2(\beta)=k\alpha$. Suppose that $\mathcal{P}$ has $n$ faces, and write $\beta=(a_1,b_1 | \ldots | a_n,b_n) \in \mathcal C_2$, 
using vertical bars `$|$' to separate elements coming from different representation rings. We shall see that one can find a $1$-chain $\beta '$, homologous
to $\beta$, and with every entry of $\beta '$ a multiple of $k$.

At an edge $e_{ij}$, the differential $\partial_2$ is described in Equation (\ref{InductionDihedral}). 
Since every entry of $\partial_2(\beta)$ is a multiple of $k$, we conclude from Equation (\ref{InductionDihedral}) that for every pair 
of intersecting faces $f_i$ and $f_j$,
\[
	a_i \equiv a_j \pmod k \quad \text{and} \quad b_i \equiv b_j  \pmod k.
\]

Equation (\ref{InductionDihedral}) also shows that $\mathbf{1}_{\partial P} = (1,1| \ldots | 1,1)$, the formal sum over all generators of representation rings of face stabilizers of ${\partial P} \in \mathcal C_2$, is in the kernel of $\partial_2$. In particular, setting $\beta' = \beta - a_1\mathbf 1_{\partial P}$,
we see that $\partial_2(\beta ')=\partial_2(\beta) = \alpha$ and we can assume without loss of generality that $\beta'$ satisfies $a_1' \equiv 0 \pmod k$. 

Let us consider the coefficients for the $1$-chain $\beta '$. For every face $f_j$ intersecting $f_1$, we have $a_1' - a_j' \equiv 0 \pmod k$, which implies $a_j' \equiv 0 \pmod k$. Since $\partial \mathcal P$ is connected, 
repeating this argument we arrive at $a_i' \equiv 0 \pmod k$ for all $i$.
In addition, there are even $m_{ij}$ (the stabilizer of a vertex is a spherical triangle group), 
and hence (\ref{InductionDihedral}) also yields $a_i' - b_j' \equiv 0 \pmod k$, which implies $b _j' \equiv 0 \pmod k$. 
Exactly the same argument as above then ensures that $b_i' \equiv 0 \pmod k$ for all $i$. 

Since all coefficients of $\beta'$ are divisible by $k$, we conclude that $\alpha = \partial_2(\beta ' /k) \in \textup{im}(\partial_2)$, as desired. 
\end{proof}

We note that a similar method of proof can be used, in many cases, to show that $H_0$ is torsion-free -- though the argument becomes
much more complicated. This approach 
is carried out in Section~\ref{sec:H0-geom}. 

\begin{cor} \label{cor:H1}
Let $cf(\Gamma)$ be the number of conjugacy classes of elements of finite order in $\Gamma$, and $\chi(\mathcal C)$ the Euler characteristic of the Bredon chain complex (\ref{BredonChainComplex}). Then we have
\[
H_1 \cong \mathbb Z^{cf(\Gamma) - \chi(\mathcal C)}.
\]
\end{cor}
\begin{proof}
The Euler characteristic of a chain complex coincides with the alternating sum of the ranks of the homology groups, giving us
\[
	\chi(\mathcal C) = \textup{rank}(H_0) - \textup{rank}(H_1) + \textup{rank}(H_2) - \textup{rank}(H_3).
\]
Since $H_3=H_2=0$, we have $\textup{rank}(H_1) = \textup{rank}(H_0) - \chi(\mathcal C)$, and $\textup{rank}(H_0)=cf(\Gamma)$~\cite{Mislin03}.  Since $H_1$ is torsion-free (Theorem~\ref{thm:H1}), the result follows.
\end{proof}

Note that both $cf(\Gamma)$ and $\chi(\mathcal C)$ can be obtained directly from the geometry of the polyhedron $\mathcal P$ or, equivalently, from the Coxeter integers $m_{ij}$. We discuss this further, and give explicit formulas, in Section~\ref{sec:geom}.

\begin{remark}
A previous article by three of the authors~\cite{LOS12} gave an algorithm to compute the rank of the Bredon homology for groups 
$\Gamma$ with a cocompact, 3-manifold model $X$ for the classifying space $\underline{E}\Gamma$. The interested reader can easily
check that the computations in the present paper agree with the calculations in~\cite{LOS12}. 
\end{remark}

To complete the computation of the Bredon homology, and hence of the equivariant $K$-homology, all that remains is to compute the torsion subgroup of $H_0$. 
We will show that in fact $H_0$ is also torsion-free. 

\begin{thm}\label{thm:H0}
There is no torsion in $H_0$.
\end{thm}

We postpone the proof to Section~\ref{sec:H0-algebra} below. We note the following immediate consequence of Theorem~\ref{thm:H0}.

\begin{cor}\label{cor:K0}
$K_0^\Gamma(X)$ is torsion-free of rank $cf(\Gamma)$. 
\end{cor}

Our Main Theorem now follows immediately by combining Corollaries~\ref{cor:KequalsH}, \ref{cor:H1}, and~\ref{cor:K0}. 
Moreover, in Section~\ref{sec:geom}, we will give a formula for $cf(\Gamma)$ and $\chi(\mathcal C)$ from the combinatorics 
of the polyhedron.


\section{No torsion in $H_0$ -- the geometric approach}\label{sec:H0-geom}

We present a geometric proof for a restricted version of Theorem~\ref{thm:H0}. The method of proof 
is similar to the proof of Theorem~\ref{thm:H1}, but with further technical difficulties. We will show:

\begin{thm}\label{thm:geometricproof}
Assume the compact polyhedron $\mathcal P$ is such that all vertex stabilizers are of the form $D_{m}\times \mathbb Z_2$, where $m\geq 3$ can vary from vertex to vertex. Then there is no torsion in the $0$-dimensional Bredon homology group $\hzero$.
\end{thm}

First, we discuss some terminology and the overall strategy of the proof. Fix $k \ge 2$ an integer. Our overall objective is to rule out $k$-torsion in $H_0$. 
Let $\beta \in C_1$ (in the Bredon chain complex of $\Gamma$) such that $\partial_1(\beta)$ is the zero vector modulo~$k$. 
Note that an element $\alpha \in C_0$ has order $k$ in $H_0 = C_0 / \textup{im}(\partial_1)$ if and only if $k\alpha=\partial_1(\beta)$. 
Recall that
\[
	C_1 = \bigoplus_{e_{ij} \text{ edge}} R_{\mathbb C}\left( \langle s_i, s_j \rangle \right),
\]
is the direct sum of the representation rings (as abelian groups) of the edge stabilizers. 
The coefficients of $x$ supported along a particular edge $e_{ij}$ is by definition the projection of $x$ to 
$\mathbb{Z}^{n_{ij}} \cong R_{\mathbb C}\left( \langle s_i, s_j \rangle \right)$, 
where $n_{ij}$ is the dimension of the representation ring of the edge stabilizer. 
We say an element $x$ is $k$-divisible along an edge $e_{ij}$ provided the coefficients of $x$ supported
along $e_{ij}$ are congruent to zero mod $k$.

By \emph{establishing $k$-divisibility of $\beta$ along an edge $e_{ij}$}, we mean: 
substituting $\beta$ by a homologous element $\beta' \in C_1$ (homologous means that $\partial_1(\beta)=\partial_1(\beta')$), 
such that $\beta '$ is $k$-divisible along $e_{ij}$. When the $1$-chain $\beta$ is clear from the context, we will abuse terminology
and simply say that the edge $e$ is $k$-divisible. We sometimes refer to an edge with stabilizer $D_n$ as an \emph{$n$-edge}, 
or \emph{edge of type $n$}.

Our goal is to replace $\beta$ with a homologous chain $\beta '$ for which all the edges are $k$-divisible, that is,
\begin{enumerate}
\item $\partial_1(\beta ') =\partial_1(\beta)$, and
\item every coefficient in the $1$-chain $\beta '$ is divisible by $k$.
\end{enumerate}
If we can do this, it follows that $\alpha=\partial_1(\beta'/k)$, and hence that $\alpha$ is zero in $H_0$. 

The construction of $\beta'$ is elementary, but somewhat involved. It proceeds via a series of steps, which will be described in the 
following subsections. Sections \ref{color-skeleton} to \ref{fix-remaining-blue-edges} contain the conceptual, geometric arguments 
needed for the proof. At several steps in the proof, we require some technical algebraic lemmas. For the sake of exposition, we 
defer these lemmas and their proofs to the very last Section \ref{formerAppendixC}. 


\subsection{Coloring the $1$-skeleton}\label{color-skeleton}

Recall that $\beta\in \mathcal C_1 = \bigoplus_{e\in \mathcal P^{(1)}} R_{\mathbb C}(\Gamma_e)$, so we can view
$\beta$ as a formal sum of complex representations of the stabilizers of the various edges in the $1$-skeleton of $\mathcal P$. The edge stabilizers are dihedral groups. Let us $2$-color the edges of the polyhedron, blue if the stabilizer is $D_2$, 
and red if the stabilizer is $D_m$, where $m\geq 3$. From our constraints on the vertex groups, we see that every vertex has exactly two incident blue edges. 
Of course, any graph with all vertices of degree $2$ decomposes as a disjoint union of cycles. 

The collection of blue edges thus forms a graph, consisting of pairwise disjoint loops, separating the boundary of the polyhedron $\mathcal P$ (topologically a 2-sphere) into a finite collection of regions, at least two of which must be contractible. The red edges appear in the interior of these individual regions, joining pairs of vertices on the boundary of the region.  
Fixing one such contractible region $R_\infty$, the complement will be planar. We will henceforth fix a planar embedding of this complement. This allows
us to view all the remaining regions as lying in the plane $\mathbb R^2$.

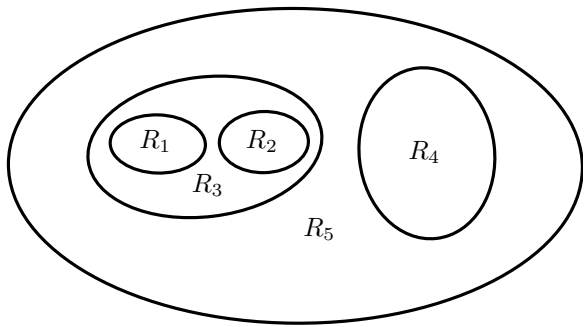
\begin{figure}[!ht]
\definecolor{cqcqcq}{rgb}{0.7529411764705882,0.7529411764705882,0.7529411764705882}
\begin{tikzpicture}[line cap=round,line join=round,>=triangle 45,x=1cm,y=0.65cm,scale=0.45]\clip(-4.761814599999997,-7.724787599999994) rectangle (12.783185400000026,7.448612399999995);\draw [rotate around={-0.4602371359713602:(4.288377000000009,-0.11693970000000042)},line width=1.2pt] (4.288377000000009,-0.11693970000000042) ellipse (8.479257342233103cm and 4.65289957470057cm);\draw [rotate around={6.970995807229445:(1.6148854000000112,0.7452123999999991)},line width=1.2pt] (1.6148854000000112,0.7452123999999991) ellipse (3.4720957401785997cm and 2.0646492629418574cm);\draw [rotate around={-83.60861155019066:(8.17308540000002,0.45481239999999834)},line width=1.2pt] (8.17308540000002,0.45481239999999834) ellipse (2.533467966552616cm and 2.0002241531041753cm);\draw [rotate around={-1.8676788394341683:(0.2233854000000095,0.8783123999999987)},line width=1.2pt] (0.2233854000000095,0.8783123999999987) ellipse (1.4068843266002835cm and 0.8595298589540278cm);\draw [rotate around={1.4502162909334195:(3.3572854000000136,0.9630123999999988)},line width=1.2pt] (3.3572854000000136,0.9630123999999988) ellipse (1.3157690716957189cm and 0.9038350513401275cm);
\begin{normalsize}
\draw[color=black] (0.15,1.0) node {$R_{1}$};
\draw[color=black] (3.3,1.0) node {$R_2$};
\draw[color=black] (8.1,0.5) node {$R_{4}$};
\draw[color=black] (1.7,-1.0) node {$R_{3}$};
\draw[color=black] (5.0,-3.0) node {$R_{5}$};
\end{normalsize}
\end{tikzpicture}
\caption{Example of enumeration of regions (see Section~\ref{enumerate-regions}). }\label{fig:enumeration}
\end{figure}


\subsection{Enumerating the regions.}\label{enumerate-regions}

Our strategy for modifying $\beta$ is as follows. We will work region by region. 
At each stage, we will modify $\beta$ by only changing it on edges contained in the closure of a region. In order to do this, we need to enumerate the regions.

We have already identified the (contractible) region $R_\infty$ -- this will be the last region dealt with. 
In order to decide the order in which we will deal with the remaining regions, we define a partial ordering on the set of regions. 
For distinct regions $R, R'$, we write $R<R'$ if and only if $R$ is contained in a bounded component of $\mathbb R^2 \setminus R'$. 
This defines a partial ordering on the finite set of regions. 
For example, any region which is minimal with respect to this ordering must be simply connected (hence contractible). 
We can thus enumerate the regions $R_1, R_2, \ldots$ so that, for any $i<j$, we have $R_j \not < R_i$ (Figure~\ref{fig:enumeration}).  
We will now deal with the regions in the order they are enumerated. Concretely, the choice of enumeration means that by the time we get to the $i^{th}$ region, 
we have already dealt with all the regions which are ``interior'' to $R_i$ (i.e.~in the bounded components of $\mathbb R^2 \setminus R_i$).

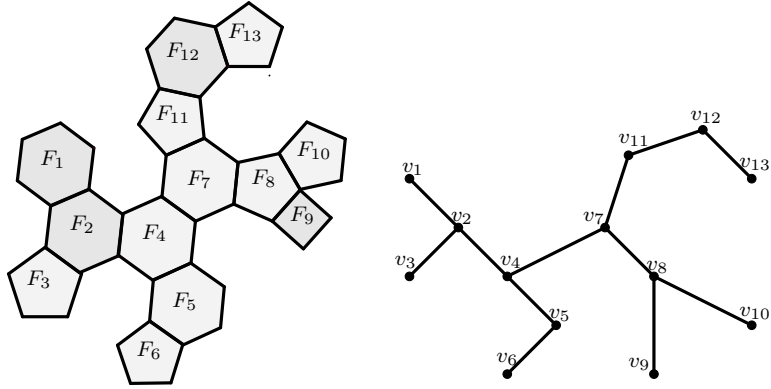
\begin{figure}[!ht]
\input{face_decomposition.tikz}
\hspace{-1cm}
\input{tree_decomposition.tikz}
\caption{Contractible region, its dual tree $T$ with an enumeration of its vertices, and the corresponding enumeration of the faces in the region
(see Section~\ref{enumerate-faces}).}\label{fig:decomp}
\end{figure}


\subsection{Enumerating faces within a region.} \label{enumerate-faces}

We now want to establish $k$-divisibility of the red edges inside a fixed region $R$. To do this, we first need to order the $2$-faces inside $R$.
Consider the graph $G$ dual to the decomposition of $R$ into $2$-faces. This graph has one vertex for each $2$-face in $R$, and 
an edge joining a pair of vertices if the corresponding $2$-faces share a (necessarily red) edge. 
Notice that every vertex of $R$ lies on a pair of blue edges. 
It follows that, if we remove the blue edges from the region~$R$, the result deformation retracts to $G$. 

If $s$ denotes the number of regions $R_i<R$ with $\partial R \cap \partial R_i$ non-empty, 
then the first homology $H_1(G) = H_1(R)$ will be free abelian of 
rank $s$. We now choose any spanning tree $T$ for the graph $G$, and note that $G\setminus T$ consists of precisely
$s$ edges, each of which is dual to a red edge inside the region $R$. Since $T$ is a tree, we can enumerate its vertices 
$v_1, v_2, \ldots, v_n$ so that 
\begin{enumerate}
 \item  $v_1$ is a leaf of the tree (vertex of degree one), and 
 \item the subgraph induced by any $\{v_1, \ldots , v_k\}$, $1 \le k \le n$, is connected.
\end{enumerate}
Since vertices of the dual tree correspond to $2$-faces in $R$, this gives us 
an enumeration of the $2$-faces $F_1, F_2, \ldots, F_n$ inside the region $R$. Figure~\ref{fig:decomp} illustrates this process
for a contractible region, while Figure~\ref{fig:noncontratible} gives an illustration for a non-contractible region.

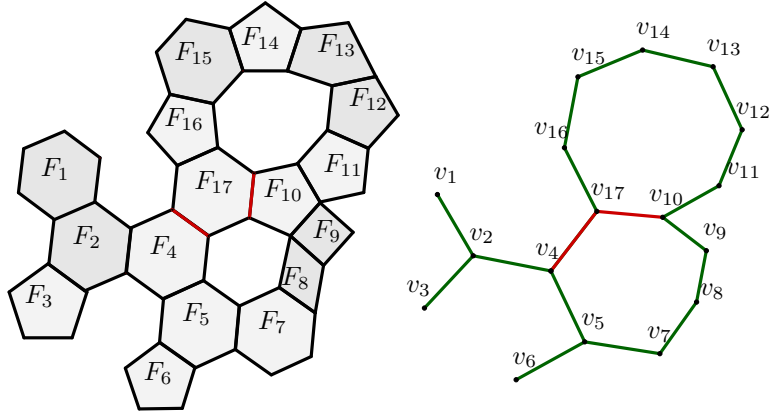
\begin{figure}[!ht]
\input{noncontractible.tikz}
\caption{Non-contractible region and its dual graph $G$. A spanning tree $T$ for $G$ consisting of all edges except $(v_4, v_{17})$ and $(v_{17}, v_{10})$.
The enumeration of the vertices, 
and corresponding enumeration of the faces is done according to Section~\ref{enumerate-faces}.}\label{fig:noncontratible}
\end{figure}


\subsection{Establishing $k$-divisibility of red edges dual to a spanning tree.}\label{fix-edges-dual-to-T}

Continuing to work within a fixed region $R$, we now explain how to establish $k$-divisibility of all the red edges which are dual to the edges in the spanning tree $T$. 
As explained in Section~\ref{enumerate-faces}, we have an enumeration $F_1, F_2, \ldots$ of the $2$-faces contained inside the region $R$. 
With our choice of enumeration, we have guaranteed that each $F_{k+1}$ shares precisely one distinguished red edge with $\bigcup_{i=1}^k F_i$,
distinguished in the sense that this red edge is dual to the unique edge in the tree $T$ connecting the vertex $v_{k+1}$ to the subtree spanned by $v_1, \ldots , v_k$.

We orient the (blue) edges along the boundary loops of $R$ clockwise, and the red edges cutting 
through $R$ in an arbitrary manner. For each $2$-face $F_i$, we want to choose corresponding elements $\eta_i$ in the representation 
ring $R_{\mathbb C}(\Gamma_{F_i}) \cong \mathbb Z^2$, where $\Gamma_{F_i} \cong C_2$ is the stabilizer of $F_i$. These $\eta_i$
shall be chosen such that all red edges dual to $T$ are $k$-divisible for the $1$-chain $\beta + \partial_2\big(\sum \eta_i)$, which is clearly homologous to $\beta$.

Pick an $\eta_1$ arbitrarily. We now assume that $\eta_1, \ldots, \eta_k$ are given, and explain how to choose $\eta_{k+1}$. By our choice of enumeration of vertices, the vertex $v_{k+1}$ is adjacent to some $v_j$ where $j\leq k$. Dual to the two vertices $v_j, v_{k+1}$ we have a pair $F_j, F_{k+1}$ of $2$-faces inside the region $R$. Dual to the edge that joins $v_j$ to $v_{k+1}$ is the (red) edge $F_j \cap F_{k+1}$. We see that $\partial _2( \eta_j + \eta_{k+1})$ is the only term which can change the portion of $\beta$ supported on the edge $F_j \cap F_{k+1}$. Since $\eta_j$ is already given, we want to choose $\eta_{k+1}$ in order to ensure that the resulting $1$-chain {\it is $k$-divisible on the edge $F_j \cap F_{k+1}$}. This will arrange key property (2) for the (red) edge $F_j \cap F_{k+1}$. 

That such an $\eta_{k+1}$ can be chosen is the content of Lemma~\ref{choosing-the-etas} in Section~\ref{formerAppendixC}. 
Iterating this process, we find that the $1$-chain $\beta + \partial_2\big(\sum \eta_i)$ is homologous to $\beta$, and that all red edges dual to edges in $T$ are $k$-divisible for $\beta + \partial_2\big(\sum \eta_i)$. 


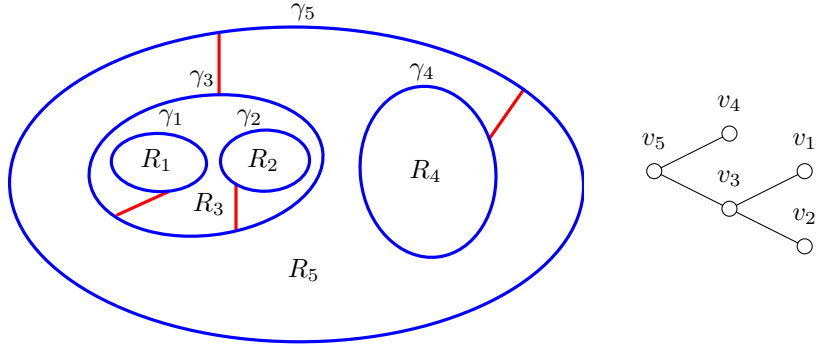
\begin{figure}[!ht]
\definecolor{cqcqcq}{rgb}{0.7529411764705882,0.7529411764705882,0.7529411764705882}
\begin{tikzpicture}[line cap=round,line join=round,>=triangle 45,x=1cm,y=0.65cm,scale=0.45]
\clip(-4.76,-7.72) rectangle (12.78,8.5);
\begin{scope}[decoration={
    markings,
    mark=at position 0.8 with {\arrow{>}}}
    ] 
\draw[color=red,line width=1.2pt] (11.0,4.2) -- (10.0,2.0);
\draw[color=red,line width=1.2pt] (2.0,6.7) -- (2.0,4.0);
\draw[color=red,line width=1.2pt] (2.5,-2.2) -- (2.5,-0.1);
\draw[color=red,line width=1.2pt] (-1.0,-1.5) -- (0.5,-0.45);
\end{scope}
\draw [rotate around={-0.4602371359713602:(4.288377000000009,-0.11693970000000042)},line width=1.2pt,blue] (4.288377000000009,-0.11693970000000042) ellipse (8.479257342233103cm and 4.65289957470057cm);
\draw [rotate around={6.970995807229445:(1.6148854000000112,0.7452123999999991)},line width=1.2pt,blue] (1.6148854000000112,0.7452123999999991) ellipse (3.4720957401785997cm and 2.0646492629418574cm);
\draw [rotate around={-83.60861155019066:(8.17308540000002,0.45481239999999834)},line width=1.2pt,blue] (8.17308540000002,0.45481239999999834) ellipse (2.533467966552616cm and 2.0002241531041753cm);
\draw [rotate around={-1.8676788394341683:(0.2233854000000095,0.8783123999999987)},line width=1.2pt,blue] (0.2233854000000095,0.8783123999999987) ellipse (1.4068843266002835cm and 0.8595298589540278cm);
\draw [rotate around={1.4502162909334195:(3.3572854000000136,0.9630123999999988)},line width=1.2pt,blue] (3.3572854000000136,0.9630123999999988) ellipse (1.3157690716957189cm and 0.9038350513401275cm);
\begin{normalsize}
\draw[color=black] (0.15,1.0) node {$R_{1}$};
\draw[color=black] (3.3,1.0) node {$R_2$};
\draw[color=black] (1.7,-1.0) node {$R_{3}$};
\draw[color=black] (8.1,0.5) node {$R_{4}$};
\draw[color=black] (4.5,-4.0) node {$R_{5}$};
\draw[color=black] (0.6,2.9) node {$\gamma_1$};
\draw[color=black] (2.9,2.9) node {$\gamma_2$};
\draw[color=black] (1.5,4.7) node {$\gamma_3$};
\draw[color=black] (8.0,5.0) node {$\gamma_4$};
\draw[color=black] (4.5,7.75) node {$\gamma_5$};
\end{normalsize}
\end{tikzpicture}
\raisebox{1.2cm}{
\begin{tikzpicture}
\clip (-0.7,-1.2) rectangle (2.2,1.2);
\draw (0,0) -- (1,0.5);
\draw (0,0) -- (1,-0.5);
\draw (1,-0.5) -- (2, -1);
\draw (1,-0.5) -- (2,0);
\node at (0, 0.4) {$v_5$};
\node at (1, 0.9) {$v_4$};
\node at (1, -0.1) {$v_3$};
\node at (2, -0.6) {$v_2$};
\node at (2,0.4) {$v_1$};
\draw[fill=white] (0, 0) circle [radius=.1];
\draw[fill=white] (1, 0.5) circle [radius=.1];
\draw[fill=white] (1, -0.5) circle [radius=.1];
\draw[fill=white] (2, 0) circle [radius=.1];
\draw[fill=white] (2, -1) circle [radius=.1];
\end{tikzpicture}}
\caption{(Left) Non-contratible region $R$ with blue loops and red edges. (Right) Graph $B$ associated to $R$. (See Section~\ref{the-graph-B}).}\label{fig:regionr}
\end{figure}

\subsection{Forming the graph $B$.}\label{the-graph-B}
Performing the process in Section~\ref{fix-edges-dual-to-T}
for each region $R$ (including the region $R_\infty$), we finally obtain a $1$-chain homologous to our original $\beta$ 
(which by abuse of notation we will still denote $\beta$) which is $k$-divisible except possibly along:
\begin{itemize}
\item $s_i$ red edges inside each region $R_i$, where $s_i$ denotes the number of regions entirely enclosed by 
the region $R_i$ who share a boundary with $R_i$;
\item all the blue edges inside the $1$-skeleton, which we recall decompose into a finite collection of blue loops.
\end{itemize}

We now use these to form a graph $B$, which captures all the remaining potentially bad edges for the chain $\beta$, i.e. edges that are still not $k$-divisible. 
$B$ is formed with one vertex for each blue loop. Note that each remaining red edge that is potentially not $k$-divisible
joins a pair of vertices which lie on some blue loops $\gamma_1, \gamma_2$ (as all vertices lie on blue loops). 
For each such red edge, we define an edge in the graph $B$ joining the two vertices $v_i$ corresponding to the blue loops $\gamma_i$ (Figure~\ref{fig:regionr}).

\begin{lemma}\label{B-tree}
The graph $B$ is a tree.
\end{lemma}

\begin{proof}
It follows immediately from the discussion in Section~\ref{enumerate-faces} that $B$ is connected. Thus it suffices to show $B$ has 
no embedded cycles. Observe that, since blue loops separate the plane into two connected components, the corresponding vertex of $B$ likewise
partitions the graph $B$ into two connected components (corresponding to the ``interior'' component and the ``exterior'' component determined by
the blue loop). Thus if there is an embedded cycle, then for each vertex $v$ of $B$, it must remain within one connected component of $B\setminus \{v\}$. 
This means that any embedded cycle has the property that all the vertices it passes through correspond to blue components lying in the closure
of a single region $R$ (and all red edges lie inside that region $R$).

Now by way of contradiction, assume that $e_1, \ldots ,e_k$ forms a cycle in $B$, cyclically joining vertices $\gamma_1, \ldots ,\gamma_k$,
where the $\gamma_i$ are blue loops inside the closure of the region $R$. We can concatenate the corresponding (red) edges $e_i$ along with paths on the blue loops $\gamma_i$ to obtain an embedded edge loop $\eta$ contained in the closure of the 
region $R$. Now pick a pair of $2$-faces $F_1, F_2$ inside the region 
$R$, where $F_1$ is contained inside the loop $\eta$, while $F_2$ is contained outside of the loop $\eta$, e.g.~pick the
$2$-faces on either side of the red edge $e_1$. Note that, by construction, the closed loop $\eta$ separates these regions from each other.

These two regions correspond to vertices $v_1, v_2$ in the graph $G$ associated to the region $R$. Since $T$ was a spanning 
tree for the graph $G$, it follows 
that we can find a sequence of edges in
the tree $T$ connecting $v_1$ to $v_2$. This gives rise to a sequence of $2$-faces connecting $F_1$ to $F_2$, where each consecutive
face share {\it an edge distinct from any of the red edges $e_i$}. Thus we obtain a continuous path joining $F_1$ to $F_2$ which is completely
disjoint from $\eta$, a contradiction. We conclude that $B$ cannot contain any cycles, and hence is a tree.
\end{proof}

Notice that each vertex in $B$ corresponds to a blue loop, which lies in the closure of precisely two regions. There will thus be a unique
such region which lies in the bounded component of the complement of the loop. This establishes a bijection between the regions and
the vertices of $B$. From the enumeration of regions
in Section \ref{enumerate-regions}, we can use the bijection to enumerate vertices of $B$. For example, the vertex with smallest labelling will always correspond
to the boundary of a contractable region. We refer the reader to Figure \ref{fig:B-labelling} for an illustration of this labelling.

\begin{figure}
\begin{center}
\begin{tikzpicture}

\draw (0,0) -- (1,0.5);
\draw (0,0) -- (1,-0.5);
\draw (1,-0.5) -- (2, -1);
\draw (1,-0.5) -- (2,0);

\node at (0, 0.4) {$v_5$};
\node at (1, 0.9) {$v_4$};
\node at (1, -0.1) {$v_3$};
\node at (2, -0.6) {$v_2$};
\node at (2,0.4) {$v_1$};

\draw[fill=white] (0, 0) circle [radius=.1];
\draw[fill=white] (1, 0.5) circle [radius=.1];
\draw[fill=white] (1, -0.5) circle [radius=.1];
\draw[fill=white] (2, 0) circle [radius=.1];
\draw[fill=white] (2, -1) circle [radius=.1];

\end{tikzpicture}
\caption{Enumeration of vertices in $B$ corresponding to enumeration of regions in Figure \ref{fig:enumeration} (see Section \ref{the-graph-B}).}\label{fig:B-labelling}
\end{center}
\end{figure}
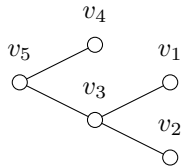


\subsection{Coefficients along the blue loops.}
Our next goal is to modify $\beta$ in order to make all the remaining red edges (i.e.~edges in the graph $B$) $k$-divisible. Note that 
the $1$-chain $\beta$ might not be an integral $1$-cycle, but it is a $1$-cycle mod $k$. We will now exploit this property to 
analyze the behavior of the $1$-chain $\beta$ along the blue loops.

\begin{prop}\label{coefficients-along-B}
Let $\gamma$ be any blue loop, oriented clockwise. Then the coefficients on the blue edges are all congruent to 
each other modulo $k$, that is, if $(a,b,c,d)$ and $(a', b', c', d')$ are the coefficients along any two blue edges in a given
blue loop, then $(a, b, c, d) \equiv (a', b', c', d')$ mod $k$.
Moreover, along any red edge in the graph $B$, the coefficient is congruent to zero mod $k$, except possibly if the edge
has an even label, in which case the coefficient is congruent to $(0, 0, \hat z, -\hat z, 0, \ldots , 0)$ for some $z$ (which
may vary from edge to edge).
\end{prop}

\begin{proof}
To see this, we argue inductively according to the ordering of the blue 
connected components (see Section~\ref{the-graph-B}). 

\vskip 10pt

\noindent \underline{Base Case.}
For the initial case, consider
the blue loop corresponding to vertex $v_1$ in the tree $B$. By hypothesis, this blue loop $\gamma$ has a single 
(red) edge incident to it which is potentially not $k$-divisible, corresponding to the unique edge in the tree $B$ incident 
to $v_1$. Let $w$ denote the single vertex on $\gamma$ where that red edge is incident, allowing us to view $\gamma$ as a path starting and
terminating at $w$.  Since all the remaining red edges incident to $\gamma$ are $k$-divisible, applying Lemma~\ref{bad-red-edge-propagating} 
in Section~\ref{formerAppendixC}
(with all the $\hat z \equiv 0$) for each incident $k$-divisible red edge shows that all the coefficients along the path $\gamma$ are 
congruent to each other mod $k$. Note that, in this base case, we are always in the 
cases $i<j<k$ or $j<k<i$ of Lemma~\ref{bad-red-edge-propagating}, according to whether the $k$-divisible red edge lies in the unbounded
or bounded region determined by the blue loop $\gamma$. This establishes the first statement of the Proposition. To get the 
second statement, we apply Lemma~\ref{red-edge-starting-in-middle-of-blue-path} in Section~\ref{formerAppendixC} at the vertex $w$, and we are done.

\vskip 10pt

\noindent \underline{Inductive Step.}
Now inductively, let us assume that we are focusing on the blue loop $\gamma_i$ corresponding to some vertex $v_i$ in $B$. We assume
that all the blue loops $\gamma_j$ corresponding to vertices $v_j$ with $j<i$ already satisfy the desired property. 
We also assume that all red edges in the graph $B$ connected to vertices $v_j$ with $j<i$ have coefficients of the form 
described in the Proposition. 

From the directed structure
of the graph $B$, the vertex $v_i$ has a {\it unique} edge $e$ connecting to a vertex $v_j$ with index $j>i$, and all the remaining edges in $B$ 
connect to a $v_j$ for some $j<i$. 

By the inductive hypothesis, this tells us that all but one of these red edges have coefficients congruent to 
$(0, 0, \hat z_j, -\hat z_j, 0, \ldots , 0)$ for some $z_j$ (which might depend on the edge). Again, viewing $\gamma_i$ as a path starting
and terminating at the same vertex $w$ (where $e$ is incident to $\gamma$), we may 
apply Lemma~\ref{bad-red-edge-propagating} and conclude that $\gamma_i$ has coefficients along all edges that are congruent to each other.
Applying Lemma~\ref{red-edge-starting-in-middle-of-blue-path} at the vertex $w$, shows that the coefficients along the edge $e$ must also be
of the form $(0, 0, \hat z, -\hat z, 0, \ldots ,0)$ for some $z$. This completes the inductive step and the proof of the Proposition.
\end{proof}


\subsection{Equivalence classes of red edges.}\label{equiv-classes}

Now consider a red edge which is potentially not $k$-divisible, corresponding to an edge in the graph $B$ joining vertices $v_i$ to $v_j$. The 
red edge thus joins the blue loop $\gamma_i$ to the blue loop $\gamma_j$. From Proposition~\ref{coefficients-along-B}, 
we see that the coefficient along the red edge must be congruent to $(0, 0, \hat z, -\hat z, 0, \ldots , 0)$ for some $z$. In particular, 
there is a single residue class that determines the coefficients along the red edge (modulo $k$).

Next let us momentarily focus on a blue loop $\gamma$, and assume the coefficients along the edges of $\gamma$ are
all congruent to $(a, b, c, d)$ modulo $k$, as ensured by Proposition~\ref{coefficients-along-B} (Figure~\ref{fig:badedges}). We define an equivalence relation 
on {\it all} the red edges with {\it even} label incident to $\gamma$, by defining the two equivalence classes: 
\begin{enumerate}[(a)]
\item the incident red edges that lie in the bounded region corresponding
to $\gamma$, and 
\item those that lie in the unbounded region. 
\end{enumerate}
It follows from Lemma~\ref{bad-red-edge-propagating} (Section~\ref{formerAppendixC}), that all edges in the equivalence class (a) have corresponding coefficients congruent to 
$(0, 0, \hat z_j, -\hat z_j, 0, \ldots , 0)$ where each $z_j\equiv b-a$, while all edges in equivalence class (b) have corresponding
coefficients congruent to $(0, 0, \hat z_j, -\hat z_j, 0, \ldots , 0)$ where each $z_j\equiv c-a$ (Figure~\ref{fig:badedges}).

\begin{figure}[!ht]
\input{bad_edges.tikz}
\caption{Bad red edges incident to blue loop $\gamma$. All \emph{interior} edges have coefficients satisfying one congruence, while all \emph{exterior} edges have coefficients satisfying a different congruence (Section~\ref{equiv-classes}).}
\label{fig:badedges}
\end{figure}
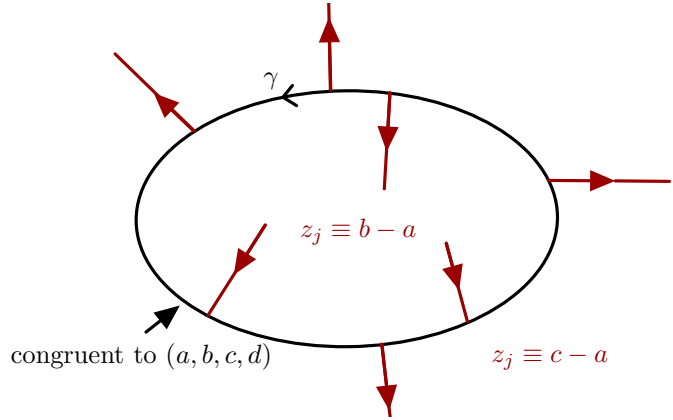

This equivalence relation is defined locally, and can be extended over all blue loops in the $1$-skeleton, resulting in an equivalence
relation on the collection of all red edges with even label. Observe that, by construction, each equivalence class has the property
that there is a single corresponding residue class $z$ mod $k$, with the property that {\it all} the edges within that equivalence
class have coefficient congruent to $(0, 0, \hat z, -\hat z, 0, \ldots , 0)$, i.e.~the $z$ is the same for the entire equivalence class. 

\begin{cor}
The edges in the graph $B$ that are not $k$-divisible are a finite union of equivalence classes for this relation.
\end{cor}

\begin{proof} 
Let $e$ be an edge in $B$, and assume that $e$ is equivalent to an edge $e'$ which is {\it not} an edge in $B$. Since all
edges that are not in $B$ are $k$-divisible, it follows that the coefficient on $e'$ is congruent to zero mod $k$. Thus the value of $z$ 
for the equivalence class $\mathcal E$ containing $e'$ is $z=0$. Since $e\in \mathcal E$, this forces $e$ to be $k$-divisible, a contradiction.
\end{proof}


\subsection{Establishing $k$-divisibility of the remaining red edges.}\label{fix-remaining-red-edges}

Observe that the edges in each equivalence class form a connected subgraph, and hence a subtree (see Lemma~\ref{B-tree}), 
of the graph $B$. This collection
of subtrees partitions the graph $B$. Any vertex of $B$ is incident to at most 
two such subtrees -- the incident red edges lying ``inside" and ``outside" the corresponding blue loop. 

We now proceed to establish $k$-divisibility of the remaining red edges for our chain. Fix an equivalence class $\mathcal E$ of red edges, and associate to
it a $1$-chain $\alpha_{\mathcal E}$ whose coefficients are given as follows:
\begin{enumerate}
\item if a blue loop $\gamma$ has an incident red edge $e\in \mathcal E$, and $e$ lies in the {\it bounded} region of $\gamma$,
then assign $(0,1,0,0)$ to each blue edge on $\gamma$;
\item if a blue loop $\gamma$ has an incident red edge $e\in \mathcal E$, and $e$ lies in the {\it unbounded} region of $\gamma$, 
then assign $(0,0,1,0)$ to each blue edge on $\gamma$;
\item along the red edges in the equivalence class (recall that all these edges have even labels), assign $\pm (0, 0, \hat 1, -\hat 1, 0, \ldots , 0)$, with sign chosen to ensure that the local $1$-cycle condition holds at both endpoints (see Lemma~\ref{bad-red-edge-propagating}).
\end{enumerate}
Notice that, one can choose the signs in (3) coherently because the equivalence class defines a subtree of the tree $B$ -- and thus 
there are no cycles (these could have potentially forced the sign along an edge to be both positive and negative).
Another key feature of the $1$-chains $\alpha_{\mathcal E}$ is that they are linearly independent. More precisely, two distinct 
equivalence classes $\mathcal E, \mathcal E'$ have associated $1$-chains $\alpha_{\mathcal E}$ and $\alpha_{\mathcal E'}$
whose supports are disjoint, except possibly along a single blue loop $\gamma$. In the case where the supports overlap along
$\gamma$, adding multiples of $\alpha_{\mathcal E}$ does not affect the $z$-value along the class $\mathcal E '$ (and vice
versa). 

 It is now immediate from the equality case of Lemma~\ref{bad-red-edge-propagating} in Section~\ref{formerAppendixC} that the $1$-chain $\alpha_{\mathcal E}$ is
in fact an integral $1$-cycle. 
Subtracting multiples of $\alpha_{\mathcal E}$ from our given chain $\beta$, we may thus obtain a homologous $1$-chain for
which all the red edges in $\mathcal E$ are now $k$-divisible. Repeating this for each of the equivalence classes, we have now obtained
a homologous $1$-chain (still denoted $\beta$) for which {\it all the red edges are $k$-divisible.}


\subsection{Establishing $k$-divisibility of the remaining blue edges.}\label{fix-remaining-blue-edges}

We now have obtained a $1$-chain with prescribed differential, whose coefficients along all red edges are $k$-divisible. It remains
to establish $k$-divisibility of the blue loops for the $1$-chain. 

If $\gamma$ is one of the blue loops, then since all incident red edges are $k$-divisible, we see that all the edges on $\gamma$ have 
coefficients which are congruent to either $(a, b, c, d)$, $(a, b, a, b)$, $(a, a, c, c)$, or $(a, a, a, a)$, according to the 
equivalence classes that are incident to $\gamma$ (see also Proposition \ref{coefficients-along-B}). 

Let us discuss, as an example, the case $(a, b, a, b)$. Note that this case occurs if the
only incident red edges to $\gamma$ with even label lie in the unbounded region determined by $\gamma$. Consider the pair
of integral $1$-cycles $\alpha_{13}, \alpha_{24}$ supported on $\gamma$, obtained by assigning to each edge on $\gamma$ the 
coefficient $(1, 0, 1, 0)$ and $(0, 1, 0, 1)$ respectively. From the equality case of Lemma~\ref{bad-red-edge-propagating} 
we see that $\alpha_{13}, \alpha_{24}$ are in fact $1$-cycles. By adding multiples of $\alpha_{13}, \alpha_{24}$, we can
now arrange for the coefficients along the blue loop $\gamma$ to all be $k$-divisible. The three other cases can be dealt with similarly; 
we leave the details to the reader.


\subsection{Completing the proof.}

Performing this process described in Section~\ref{fix-remaining-blue-edges}
for all the blue components, we finally obtain the
desired $1$-chain $\beta ^\prime$. Since $\beta ^\prime$ now satisfies properties (1) and (2) mentioned at the beginning of the
proof, we conclude that the given hypothetical torsion class $\alpha \in H_0$ was in fact the zero class. This completes the
proof of Theorem~\ref{thm:geometricproof}.

\begin{remark}\label{rem:222}
It is not obvious how to adapt the strategy in the geometric proof above to the case when other vertex stabilizer types are 
allowed. In the case of vertex stabilizers of the type $\Delta(2,3,3)$, $\Delta(2,3,4)$, and $\Delta(2,3,5)$, most of the 
arguments can be adapted. The main difficulty lies in the arguments of Section~\ref{fix-edges-dual-to-T}, which rely
heavily on Lemma~\ref{bad-red-edge-propagating} in Section~\ref{formerAppendixC}. Unfortunately, the analogue of that Lemma does not seem to hold when one
allows these other types of vertices as endpoints of the edge.
For vertices with stabilizer $\Delta(2,2,2)$, additionally difficulties arise, notably in Sections~\ref{equiv-classes}, \ref{fix-remaining-red-edges}, and~\ref{fix-remaining-blue-edges}.
\end{remark}


\subsection{Local Analysis}\label{formerAppendixC}
The geometric proof that $\hzero$ is torsion free, Theorem~\ref{thm:geometricproof}, relies on a detailed local analysis of the 
induction homomorphism at the vertices of the polyhedron $\mathcal P$.  We state and prove the results needed here. Although 
rather technical, they are all, unless an explicit proof is given, straightforward consequences of the induction homomorphisms 
in Appendix~\ref{AppendixB}. Let us introduce some notation. Throughout this section, $\alpha$ will denote an integral $1$-chain
that is also a $1$-cycle (mod $k$), i.e. $\partial_1(\alpha)\equiv 0$. Our goal is to understand how this condition constrains the coefficients of $\alpha$.

At every vertex $v=v_{ijk}$, there are three incident edges $e_1=e_{ij}$, $e_2=e_{ik}$ and $e_3=e_{jk}$, and let $x_1$, $x_2$ and $x_3$ be $\alpha$ projected along those edges, written as column vectors. Write $A_1$ for the integer matrix representing the induction homomorphism from $e_1$ to $v$, that is, 
$
	R_{\mathbb C}(\langle s_i, s_j \rangle) \longrightarrow  R_{\mathbb C}(\langle s_i, s_j, s_k \rangle)
$
(as always, we will implicitly identify representation rings with free abelian groups, via the bases explicitly described in Appendix~\ref{AppendixA}), and define $A_2$ and $A_3$ analogously for $e_2$ and $e_3$. So each of these matrices is a submatrix of $\partial_1$ in matrix form. Then the value of $\partial_1(x)$ at the vertex $v=v_{ijk}$ (that is, projected to $R_{\mathbb C}(\langle s_i, s_j, s_k \rangle)$) is given by the matrix product
$
	\left( \pm A_1 | \pm A_2 | \pm A_3 \right) \cdot \begin{pmatrix} x_1\\ x_2 \\ x_3 \end{pmatrix},
$
with signs depending on edge orientations. 
The product above is zero modulo $k$, by the hypothesis on $\alpha$ being a 1-cycle mod $k$.  We can reduce modulo $k$ all the entries and, abusing notation, still call the resulting matrices and vectors $A_1$, $A_2$, $A_3$, $x_1$, $x_2$ and $x_3$. Furthermore, for simplicity, let us redefine $A_1$ as $-A_1$ etc as needed to take account of the chosen orientations. Then we have that the column vector representing $\alpha$ locally at $v$ (i.e.~along the incident edges) is in the kernel of the matrix representing $\partial_1$ locally at $v$, that is,
$
	\begin{pmatrix} x_1\\ x_2 \\ x_3 \end{pmatrix} \in \ker \left( A_1 | A_2 | A_3 \right).
$
One important consequence is that we can perform row operations on the matrix $A = \left( A_1 | A_2 | A_3 \right)$ without changing its kernel and, 
in particular, we may row reduce $A$, for instance into its Hermite normal form, to simplify calculations. 
(Obviously, row reduction must be performed modulo $k$, that is, in $\mathbb{Z}/k\Z$.) 
Another consequence is that, to study the consequences of establishing $k$-divisibility of an edge, we only need to remove the corresponding matrix block and vector. 
For example, if $k$-divisibility of $e_1$ has been established for $\alpha$, that is, if $x_1$ is zero modulo $k$, then the equation above is equivalent to
$
	\begin{pmatrix} x_2 \\ x_3 \end{pmatrix} \in \ker \left( A_2 | A_3 \right), 
$
and we can now row reduce this matrix to help us calculate its kernel, if needed. Recall that, throughout Section~\ref{sec:H0-geom}, we are only interested
in the case where all vertices have stabilizers of the form $\Delta(2,2,m)$, with $m>2$ (see statement of Theorem \ref{thm:geometricproof}). So let us now focus on that case.

For a $\Delta(2,2,m)$, $m>2$, vertex, recall that we assume $j<k$ and we have, from Figures~\ref{InductionDelta22mCase1}, \ref{InductionDelta22mCase2} and~\ref{InductionDelta22mCase3}, induction matrices $M_{2,m}^{(1)<}$ (if $i<j$) or $M_{2,m}^{(1)>}$ (if $i>j$), $M_{2,m}^{(2)<}$ (if $i<k$) or $M_{2,m}^{>}$ (if $i>k$), and $M_{m,m}$ where
\[ 
	M_{2,m}^{(1)<} =
	\footnotesize
	\left( \begin{array}{cccc}
		1&0&0&0\\
		0&1&0&0\\
		\widehat{0}&\widehat{1}&\widehat{0}&\widehat{0}\\
		\widehat{1}&\widehat{0}&\widehat{0}&\widehat{0}\\ 
		1&1&0&0\\
		\vdots&\vdots&\vdots&\vdots\\
		1&1&0&0\\
		0&0&1&0\\
		0&0&0&1\\
		\widehat{0}&\widehat{0}&\widehat{0}&\widehat{1}\\
		\widehat{0}&\widehat{0}&\widehat{1}&\widehat{0}\\ 
		0&0&1&1\\
		\vdots&\vdots&\vdots&\vdots\\
		0&0&1&1
	\end{array} \right)
	{\normalsize \text{, }
	M_{2,m}^{(2)<} =}
	\footnotesize
	\left( \begin{array}{cccc}
		1&0&0&0\\
		0&1&0&0\\
		\widehat{1}&\widehat{0}&\widehat{0}&\widehat{0}\\
		\widehat{0}&\widehat{1}&\widehat{0}&\widehat{0}\\ 
		1&1&0&0\\
		\vdots&\vdots&\vdots&\vdots\\
		1&1&0&0\\
		0&0&1&0\\
		0&0&0&1\\
		\widehat{0}&\widehat{0}&\widehat{1}&\widehat{0}\\
		\widehat{0}&\widehat{0}&\widehat{0}&\widehat{1}\\ 
		0&0&1&1\\
		\vdots&\vdots&\vdots&\vdots\\
		0&0&1&1
	\end{array} \right) 	
	{\normalsize \text{, }
	M_{m,m} =}
 \left( \begin {array}{c} {\normalsize \text{Id}_{c(D_m)}}\\[0.5ex] \hline \\[-1ex]  {\normalsize \text{Id}_{c(D_m)}}
\end {array} \right),
\]
and $M_{2,m}^{(1)>}$, respectively $M_{2,m}^{(2)>}$, equals $M_{2,m}^{(1)<}$, respectively $M_{2,m}^{(2)<}$, with the 2nd and 3rd columns interchanged. 

In the following lemmas, recall that $\alpha$ is an integral $1$-chain which is also a $1$-cycle mod $k$. Moreover, in Lemmas~\ref{bad-red-edge-propagating} 
and~\ref{red-edge-starting-in-middle-of-blue-path}, let $\widehat{\ }$ denote coefficients that only appear when $m$ is even, and recall the standard labelling 
of faces: the $m$-edge lies between the faces labelled $F_j$ and $F_k$, and the labelling always satisfies (without loss of generality) that $j<k$. Finally, recall 
that we refer to an edge with stabiliser $D_m$ as an \emph{$m$-edge}, or \emph{edge of type $m$}. 

\begin{lemma}\label{choosing-the-etas}
Let $\alpha$ be an integral $1$-chain, which we assume is also a $1$-cycle mod $k$. Let $F_1$, $F_2$ be a pair of adjacent $2$-faces, 
sharing a common edge $e$, with endpoint $v$ whose stabilizer is of the form $\Delta(2,2,m)$, with $m>2$. 
Assume that we are given an $\eta_1 = (n_1, m_1)$ in the representation ring $R_{\mathbb C}(C_2)$ associated 
to the stabilizer of the $2$-face $F_1$. Then there exists a choice of $\eta_2=(n_2, m_2)$ in the representation ring $R_{\mathbb C}(C_2)$ associated to the stabilizer of the $2$-face $F_2$, with the property that $\alpha + \partial_2(\eta_1+\eta_2)$ has coefficient along $e$
congruent to zero mod $k$ (i.e.~the edge $e$ is now $k$-divisible).
\end{lemma}

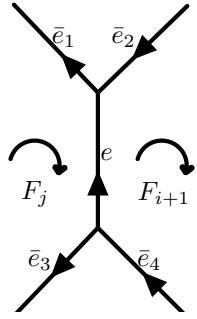
\begin{figure}[!ht]
\input{localpicture.tikz}
\caption{Local picture near the edge $e$ (Lemma~\ref{choosing-the-etas}).}\label{fig:localpicture}
\end{figure}

\begin{proof}
The edge $e$ has stabilizer $D_m$, with $m>2$. We will assume the orientations along the edges and faces are
as given in Figure~\ref{fig:localpicture}.
Assume the coefficients of $\alpha$ supported on the edge $e$ are given by $(a, b, \hat c, \hat d, r_1, \ldots ,r_k)$. 
We choose $\eta_2:=(a+n_1, r_1+m_1-a)$. 
A straightforward computation using the induction formulas shows that, with this choice of $\eta_2$, the coefficient of 
$\alpha':=\alpha + \partial_2(\eta_1+\eta_2)$ along the edge $e$ is of the form $ z:=(0, b', \hat c', \hat d', 0, r_2', \ldots , r_k')$. That is
to say, we chose $\eta_2$ in order to force $a'=r_1'=0$. 

We are left with checking that the remaining coefficients of $\alpha'$ 
are all congruent to zero mod $k$. To see this, we use the fact that $\alpha'$ is also a $1$-cycle mod $k$. From the labelings of
the faces around the vertex $v$, and the order in which we label faces (one region at a time), we see that we are in one of the
two cases $i<j<k$ or $j<k<i$. Let us consider the case $i<j<k$, and assume that the coefficients along the $2$-edges incident to $v$
are given by $ x := (x_1, x_2, x_3, x_4)$ and $ y:=(y_1, y_2, y_3, y_4)$. As $\alpha'$ is a $1$-cycle mod $k$, we have
\[
	\begin{pmatrix} M_{2,m}^{(1)<}\,\Big|\,  -M_{2,m}^{(2)<}\,\Big|\, - M_{m,m} \end{pmatrix} \begin{pmatrix}  x\\  y\\ \  z \end{pmatrix} \equiv \mathbf{0},
\]
Note that the last $k$ rows of the matrix are identical, giving rise to $k$ identical relations $x_2 + x_4 - y_2 - y_4 +r_i' \equiv 0$ (for
$1\leq i \leq k$). Since $r_1' = 0$, these equations immediately imply that all the remaining $r_i' \equiv 0$. 

Let us now assume that $m$ is odd. The first
and third row of the matrix give rise to equations $x_1 - y_1 \equiv 0$ and $x_1 + x_3 - y_1 -y_3 \equiv 0$, forcing $x_3 - y_3 \equiv 0$.
Using the second row, we get $x_3 - y_3 + b' \equiv 0$, which immediately gives $b' \equiv 0$. This completes the proof when 
$i<j<k$ and $m$ is odd. The case where $m$ is even is analogous -- one just uses the equations obtained from the first five rows of the matrix to 
conclude that $a', b', \hat c'$, and $\hat d'$ are all congruent to zero mod $k$. 

Finally, if $j<k<i$, then one proceeds in a completely similar manner, but using the block matrix $\left( M_{2,m}^{(1)>}\, |\,  -M_{2,m}^{(2)>}\,|\, - M_{m,m} \right)$ instead. It is again straightforward to work through the equations -- we leave the details to the reader.
\end{proof}

\begin{lemma}\label{bad-red-edge-propagating}
Consider a vertex of type $\Delta(2,2,m)$, $m>2$, with the incident $2$-edges oriented compatibly. If the coefficients of $\alpha$ along the $m$-edge are congruent to $(0,0, \hat z, -\hat z, 0, \ldots , 0)$ for some $z$, then the coefficients $(a, b, c, d)$ and $(a', b', c', d')$ along the pair of $2$-edges satisfy the following congruences:
\begin{enumerate}[(i)]
\item if $i<j<k$, then $(a, b, c, d) \equiv (a', b' , c', d')$, and $\hat z \equiv b-a \equiv d-c$;
\item if $j<i<k$, then $(a, b, c, d) \equiv (a', c', b', d')$, and $\hat z \equiv c-a \equiv d-b$;
\item if $j<k<i$, then $(a, b, c, d) \equiv (a', b' , c', d')$, and $\hat z \equiv c-a \equiv d-b$;
\end{enumerate}
and we have oriented the $m$-edge so that the vertex is its source. (With the opposite orientation, simply replace $\hat z$ by $-\hat z$.)
Moreover, the same statement holds if one changes all congruences to equalities.
\end{lemma}
\begin{proof}
Let $x_1=(a, b, c, d)$, $x_2=(a', b', c', d')$ and $x_3=(0,0, \hat z, -\hat z, 0, \ldots , 0)$ be the coefficients of $\alpha$ along the edges incident to the vertex. Consider the case $i<j<k$ first. Since $\alpha$ is a 1-cycle mod $k$,
\[
	\begin{pmatrix} M_{2,m}^{(1)<}\,\Big|\,  -M_{2,m}^{(2)<}\,\Big|\, - M_{m,m} \end{pmatrix} \begin{pmatrix} x_1\\ x_2\\ \ x_3 \end{pmatrix} \equiv \mathbf{0},
\]
which gives $a - a' \equiv 0$, $b - b' \equiv 0$, $b - a' - \hat{z} \equiv 0$, $a - b' + \hat z \equiv 0$, $a + b - a' - b' \equiv 0$, 
$c - c' \equiv 0$, $d - d' \equiv 0$, $d - c' - \hat{z} \equiv 0$, $c - d' + \hat z \equiv 0$, and $c + d - c' - d' \equiv 0$, from which the result follows. The other two cases, $j<i<k$ and $j<k<i$, are analogous but for the block matrix $\left( M_{2,m}^{(1)>}\, |\,  -M_{2,m}^{(2)<}\,|\, - M_{m,m} \right)$, respectively $\left( M_{2,m}^{(1)>}\, |\,  -M_{2,m}^{(2)>}\,|\, - M_{m,m} \right)$. The former gives the same congruences but with $b$ and $c$ interchanged, and the latter with $b$ and $c$, and $b'$ and $c'$, interchanged. For the opposite orientation of the $m$-edge, replace $- M_{m,m} $ by $M_{m,m}$ in the calculation above.  
\end{proof}


\begin{lemma}\label{red-edge-starting-in-middle-of-blue-path}
Consider a vertex of type $\Delta(2,2,m)$, $m>2$, with the incident $2$-edges oriented compatibly. Assume the coefficients along the $2$-edges are both congruent to $(a, b, c, d)$, that the $m$-edge is oriented compatibly with the first 2-edge, and that the faces are labelled so that
$i<j<k$ or $j<k<i$ (so we are excluding the case $j<i<k$). Then 
the $m$-edge coefficients are congruent to $(0, 0, \hat z, - \hat z, 0, \ldots , 0)$ where 
\begin{enumerate}[(i)]
\item if $i<j<k$, then $\hat z \equiv a - b$;
\item if $j<k<i$, then $\hat z \equiv a - c$.
\end{enumerate}
(If we reverse the orientation on the $m$-edge, the congruencies above hold with $\hat z$ replaced by $-\hat z$.) In particular, if $m$ is odd, the $m$-edge is automatically 
$k$-divisible.
\end{lemma}

\begin{proof}
We are assuming $x_1\equiv x_2 \equiv (a, b, c, d)$, and that $x_3=(x,y, \hat z, \hat t, r_1, \ldots , r_k)$ are the coefficients of $\alpha$ along the edges incident to the vertex. Consider the case $i<j<k$ first. Since $\alpha$ is a 1-cycle mod $k$,
\[
\begin{pmatrix} M_{2,m}^{(1)<}\,\Big|\,  -M_{2,m}^{(2)<}\,\Big|\, - M_{m,m} \end{pmatrix} \begin{pmatrix} x_1\\ x_2\\ x_3 \end{pmatrix} 
\equiv \mathbf{0},
\]
which gives $x \equiv a - a \equiv 0$, $y\equiv b - b \equiv 0$, $\hat c \equiv a - b$, $\hat d \equiv b - a$, while all the remaining equations
are of the form $r_i \equiv (a+ b) - (a+ b) \equiv 0$. The claim follows. The case $j<k<i$ is completely analogous, but uses instead
the block matrix $\left( M_{2,m}^{(1)>}\, |\,  -M_{2,m}^{(2)>}\,|\, - M_{m,m} \right)$. The details are left to the reader.
\end{proof}

It is perhaps worth noting that the analogue of Lemma~\ref{red-edge-starting-in-middle-of-blue-path} is {\it false} if the faces are enumerated
to satisfy $j<i<k$. In particular, the corresponding block matrix $\left( M_{2,m}^{(1)>}\, |\,  -M_{2,m}^{(2)<}\,|\, - M_{m,m} \right)$ leads
to, for example, $y \equiv b-c$, which is not necessarily zero. 


\section{No torsion in $H_0$ -- the linear algebra approach}\label{sec:H0-algebra}
In this section, we give a proof of Theorem~\ref{thm:H0} inspired by the representation ring splitting technique of~\cite{Rahm16}. 
We do this by establishing a criterion for $\hzero$ to be torsion-free. Our criterion is efficient to check, and only requires elementary linear algebra.
Furthermore, we will see it is satisfied for $\Gamma$ a $3$-dimensional hyperbolic Coxeter group. 

The verification of our criterion relies on simultaneous base transformations of the representation rings, 
bringing the induction homomorphisms into the desired form. For the $3$-dimensional hyperbolic Coxeter groups,
these transformations are carried out in Appendices~\ref{AppendixA} and~\ref{AppendixB}.  In Section \ref{sec:further-examples} which comes next,
we will also see that this condition is satisfied for several additional classes of groups that had previously been considered by other authors.

\begin{definition}\label{def:vertexblock}
 The \emph{vertex block} of a given vertex $v$ in a Bredon chain complex differential matrix $\partial_1$ consists of all the blocks of $\partial_1$ 
 that are representing maps induced (on complex representation rings from $\Gamma_e \to \Gamma_v$) by edges $e$ incident to $v$.
\end{definition}

We represent elements 
in the Bredon chain complex as column vectors. So the matrix $D$ for the differential $\partial _1$ is a $\rank \mathcal{C}_0 \times \rank \mathcal{C}_1$
matrix, acting by left multiplication on a column vector in $\mathcal C_1$.
For a vertex $v$, denote by $n_0$ the rank of $R_\mathbb{C} \left(  \Gamma_{v} \right)$, and by  $n_1$, $n_2$, $n_3$, the
ranks of the representation rings corresponding to the three edges $e_1$, $e_2$ and $e_3$ incident to $v$.  Then the vertex
block for $v$ is a submatrix of $D$ of size $n_0 \times (n_1+n_2+n_3)$.
Since vertex blocks have been constructed to contain all entries from incident edges, 
we note that the rest of the entries in their rows are zero. 

\begin{thm}\label{thm:basetransf}
 If there exists a base transformation such that all minors in all vertex blocks are in the set $\{-1, 0, 1\}$,
 then $\hzero$ is torsion-free.
\end{thm}

\begin{proof}[Proof of Theorem~\ref{thm:basetransf}]
We start by recalling a general result on Smith Normal Forms, already observed by Smith~\cite{SmithNF}.
Denote by $d_i(A)$ the $i$-th determinant divisor $d_{i}(A)$, defined to be 
the greatest common divisor of all $i\times i$ minors of a matrix $A$ when $i\geq 1$, and to be $d_{0}(A):=1$ when $i=0$. 
Then the elementary divisors of the matrix~$A$, up to multiplication by a unit, coincide with the ratios
$\alpha_i = \frac{d_i(A)}{d_{i-1}(A)}$.

Let us use the notation
$$ \prerank := \rank \mathcal{C}_1 -\rank \ker \partial_1,$$
where $\mathcal{C}_1$ is the module of $1$-chains in the Bredon chain complex (Equation~(\ref{BredonChainComplex})).
Observe that, if $A$ is {\it any} $i \times i$ submatrix of $D$ of non-zero determinant, and $i<\prerank$, then $A$ can be expanded to some $(i+1)\times
(i+1)$ submatrix of non-zero determinant.

Now $\hzero$ is torsion-free if and only if $\alpha_i= \pm1$ for all $1 \le i \le \prerank$. From the discussion above,
it is sufficient to find, for each $1 \le i \le \prerank$, an $i\times i$ minor in the Bredon chain complex differential 
matrix $\partial_1$ with determinant $\pm 1$. We produce such a minor by induction on $i$. 

\vskip 10pt

\noindent \underline{Base Case.}
For $i = 1$, we observe there are vertices with adjacent edges, hence there are non-zero vertex blocks.
As by assumption all the entries in the vertex blocks are in the set $\{-1, 0, 1\}$, there exists an entry of value $\pm 1$.

\vskip 5pt

\noindent \underline{Inductive step.}
Let $2 \leq i \leq \prerank$, and assume we already have an $(i-1) \times (i-1)$ minor of $\partial_1$ of value $\pm 1$, 
corresponding to a submatrix $B$. We want to find an 
$i \times i$ minor of $\partial_1$ of value $\pm 1$. 

Given any vertex block $V$, choose a maximal square submatrix $B^\circ$ of $B$ which is
disjoint from the rows and columns of $V$. At the two extremes, this submatrix could be empty (if $B$ is contained in $V$) or could coincide with $B$ 
(if $B$ is completely disjoint from $V$). 
Note that, after possibly permuting rows, we get a square sub-block $M$ of $V$ such that the submatrix $B$ takes the form
 $$B = \det \begin{pmatrix}                    M & 0 \\
                                                 * & B^\circ \\
                                                \end{pmatrix}.$$
Then in particular $\pm 1 = \det(B) = \det(M) \cdot \det(B^\circ)$, which forces $\det(B^\circ) = \pm 1$.
One can then consider extending $B^\circ$ to an $i\times i$ block $B'$ 
by picking a submatrix $M'$ inside $V$ of size $i - \text{size}(B^\circ)$. Such an extension might 
not be possible, but when it is,
the resulting block $B'$ takes the form (possibly after permuting rows)
 $$B' = \det \begin{pmatrix}                    M' & 0 \\
                                                 * & B^\circ \\
                                                \end{pmatrix}.$$
Consider the collection of all $i\times i$ blocks obtained in this manner, 
and note that for any such block, we have $\det(B') = \det(M') \cdot \det(B^\circ) = \pm 1$. 

Since $i \leq \prerank$, there exists a vertex block $V$ for which this construction yields an $i \times i$-block with $\det(B')\neq 0$.
We have that $M'$ is a minor in the vertex block $V$, so by hypothesis $\det(M') \in \{-1, 0, 1\}$. Since $\det(B')\neq 0$, we conclude 
that $\det(M') = \pm 1$. And as we already noted above, the submatrix $B^\circ$ of $B'$ satisfies $\det(B^\circ) \in \{-1, 1\}$. This implies 
our submatrix $B'$ satisfies $\det(B')=\det(M')\cdot \det(B^\circ)=\pm1$, which completes the inductive step and hence the proof of the theorem. 

\end{proof}

Our proof of Theorem~\ref{thm:H0} now reduces to verifying the hypotheses of Theorem~\ref{thm:basetransf}, when $\Gamma$
is a $3$-dimensional hyperbolic reflection group. We will rely on the simultaneous base transformations that can be found 
in Appendix~\ref{AppendixA}. 

\begin{prop}
 For a system of finite subgroups of types $A_5 \times C_2$, $S_4$, $S_4 \times C_2$, $\Delta(2,2,2) = (C_2)^3$ and
 $\Delta(2,2,m) = C_2 \times D_m$ for $m \geq 3$ as vertex stabilizers,
 with their three 2-generator Coxeter subgroups as adjacent edge stabilizers,
 there is a simultaneous base transformation such that all vertex blocks have all their minors contained in the set $\{-1, 0, 1\}$.
\end{prop}
\begin{proof}
We apply the base transformation specified in Appendix~\ref{AppendixA}. 
Then we have that all of the induced maps have all of their entries in the set $\{-1, 0, 1\}$ (see Appendix~\ref{AppendixB}, 
all tables referenced in this proof can be found there). Next, for each vertex stabilizer type, 
we assemble the vertex blocks from the three vertex-edge-adjacency induced 
maps. 

Let us provide full details for the case of vertex stabilizer $\Delta(2,2,m) = C_2 \times D_m$ for $m \geq 3$.
By Tables~\ref{MapDmtoDmxC2} and~\ref{MapD2toDmxC2}, the vertex block of a stabilizer of type $C_2 \times D_m$ for $m \geq 3$ odd consists of
\begin{center}
two blocks \footnotesize $\pm \begin{pmatrix}
 1 &  0 &  0 & 0\\
 0 &  0 &  0 & 1\\
 0 &  0 &  0 & 0\\
\vdots&\vdots&\vdots&\vdots\\
 0 &  0 &  0 & 0 \\
 0 &  1 &  0 & 0\\
 0 &  0 &  1 &0\\
 0 &  0 &  0 &0\\
 \vdots&\vdots&\vdots&\vdots\\
 0 &  0 &  0 &0\\  
            \end{pmatrix}$\normalsize \ 
\ and one block \footnotesize$\pm \begin{pmatrix}
\text{identity matrix of size } \frac{m+3}{2}\\
 0 \\  
            \end{pmatrix}$.\normalsize
\end{center}
Note that all the columns in this matrix have a very special form: all but one of the entries are zero, and the single non-zero entry
is $\pm 1$. An easy induction shows that, when checking whether the minors all take value in the set $\{0, \pm 1\}$, such columns
can always be discarded (and likewise for rows). This fact is very useful for reducing the size of the matrices to check. 
For the matrix above, this fact immediately lets us conclude that all minors are in $\{0, \pm 1\}$.

For $m \geq 6$ even, but not a power of $2$, Tables~\ref{MapDmtoDmxC2 m even} and~\ref{MapD2toDmxC2 m even} yield the following vertex block, where each matrix block is specified up to orientation sign (we make this assumption from now on),
\footnotesize
\[
		\begin{array}{c|ccccccc|cccc|cccc}
D_m \times C_2 & \multicolumn{7}{c}{D_m \hookrightarrow D_m \times C_2} & \multicolumn{4}{c}{D_2 \hookrightarrow D_m \times C_2} & \multicolumn{4}{c}{D_2 \hookrightarrow D_m \times C_2}\\
		 \hline 
\rho_1 \otimes 	 \chi_1 		 \downarrow			&  1    &  0   & 0 &0 &0 & 0     &0  &  1 & 0 & 0 & 0   &  1 & 0 & 0 & 0\\
\rho_1 \otimes 	( \chi_2 -\chi_1)        \downarrow			&  0    &  1   & 0 &1 &0 & 0     &0  &  0 & 0 & 0 & 1   &  0 & 0 & 0 & 1\\
\rho_1 \otimes 	( \chi_3 -\chi_2)        \downarrow			&  0    &  0   & 1 &-1&0 & 0     &0  &  0 & 0 & 0 & 0   &  0 & 0 & 0 & 0\\
\rho_1 \otimes 	( \chi_4 -\chi_1)        \downarrow			&  0    &  0   & 0 &1 &0 &0      &0  &  0 & 0 & 0 & 0   &  0 & 0 & 0 & 0\\
\rho_1 \otimes 	(\phi_1 -\chi_3 -\chi_1) \downarrow			&  0    &  0   & 0 &1 &1 &0      &0  &  0 & 0 & 0 & 0   &  0 & 0 & 0 & 0\\ \vdots \quad
\rho_1 \otimes 	(\phi_p -\phi_{p-1} )\downarrow
		    \quad \vdots 	 				&  0    &  0   & 0 &0 &0 &\ddots &0  &\vdots&\vdots&\vdots&\vdots&\vdots&\vdots&\vdots&\vdots\\ 
\rho_1 \otimes (\phi_{\frac{m}{2}-1} -\phi_{\frac{m}{2}-2})\downarrow   &  0    &  0   & 0 &0 &0 &0      &1  &  0 & 0 & 0 & 0   &  0 & 0 & 0 & 0\\
\hline
(\rho_2 -\rho_1) \otimes 	 \chi_1\downarrow			&       &      &   &  &  &       &   &  0 & 1 & 0 & 0   &  0 & 1 & 0 & 0\\
(\rho_2 -\rho_1) \otimes ( \chi_2 -\chi_1)\downarrow		  	&       &      &   &  &  &       &   &  0 & 0 & 1 & 0   &  0 & 0 & 1 & 0\\
(\rho_2 -\rho_1) \otimes ( \chi_3 -\chi_2)\downarrow		 	&       &      &   &  &  &       &   &  0 & 0 & 0 & 0   &  0 & 0 & 0 & 0\\
(\rho_2 -\rho_1) \otimes (\chi_4 +\chi_3 -\chi_2 -\chi_1)\downarrow 	&       &      &   &  \mathbf{0}& &       &   &  0 & 0 & 0 & 0   &  0 & 0 & 0 & 0\\
(\rho_2 -\rho_1) \otimes (\phi_1 -\chi_2 -\chi_1)\downarrow	  	&       &      &   &  &  &       &   &  0 & 0 & 0 & 0   &  0 & 0 & 0 & 0\\ 
\vdots \quad
\rho_1 \otimes 	(\phi_p -\phi_{p-1} )\downarrow
		    \quad \vdots 	 				&       &      &   &  &  &       &   &\vdots&\vdots&\vdots&\vdots&\vdots&\vdots&\vdots&\vdots\\ 
\rho_1 \otimes (\phi_{\frac{m}{2}-1} -\phi_{\frac{m}{2}-2})\downarrow   &       &      &   &  &  &       &   &  0 & 0 & 0 & 0   &  0 & 0 & 0 & 0\\
	\end{array} 
\]
\normalsize
Again, we can discard all the rows and columns which have at most one entry $\pm 1$ (and all other entries zero).
This reduces the above vertex block to the much smaller matrix 
$$\left(\begin{array}{ccccccccccccccccc}
   1   & 0 &1 &0 &\pm 1&\pm 1\\
   0   & 1 &-1&0 & 0   &  0 \\
   0   & 0 &1 &0 & 0   &  0 \\
   0   & 0 &1 &1 & 0   &  0 \\ 
\end{array}\right),$$\normalsize
for which we can easily check that all minors lie in $\{0, \pm 1\}$.

\vskip 10pt

For $m \geq 4$ a power of $2$, Tables~\ref{MapDmtoDmxC2 m power of 2} and~\ref{MapD2toDmxC2 m power of 2} yield the following vertex block,
\footnotesize
\[
		\begin{array}{c|ccccccc|cccc|cccc}
D_m \times C_2 & \multicolumn{7}{c}{D_m \hookrightarrow D_m \times C_2} & \multicolumn{4}{c}{D_2 \hookrightarrow D_m \times C_2} & \multicolumn{4}{c}{D_2 \hookrightarrow D_m \times C_2}\\
		 \hline 
\rho_1 \otimes 	 \chi_1 		 \downarrow&  1    &  0   & 0           &0     &0       & 0     &0   &  1 & 0 & 0 & 0   &  1 & 0 & 0 & 0\\
\rho_1 \otimes 	( \chi_2 -\chi_1)        \downarrow&  0    &  1   & 0           &1     &0       & 0     &0   &  0 & 0 & 0 & 1   &  0 & 0 & 0 & 1\\
\rho_1 \otimes 	( \chi_3 -\chi_1)        \downarrow&  0    &  1   & 1           &0     &0       & 0     &0   &  0 & 0 & 0 & 0   &  0 & 0 & 0 & 0\\
\rho_1 \otimes 	( \chi_4 -\chi_2)        \downarrow&  0    &  -1  & 0           &0     &0       &0      &0   &  0 & 0 & 0 & 0   &  0 & 0 & 0 & 0\\
\rho_1 \otimes 	(\phi_1 -\chi_2 -\chi_1) \downarrow&  0    &  0   & 1           &0     &1       &0      &0   &  0 & 0 & 0 & 0   &  0 & 0 & 0 & 0\\ \vdots \quad
\rho_1 \otimes 	(\phi_p -\phi_{p-1} )\downarrow
		    \quad \vdots 	 	   &  0    &  0   & 0           &0     &0       &\ddots &0   &\vdots&\vdots&\vdots&\vdots&\vdots&\vdots&\vdots&\vdots\\ 
(\phi_{\frac{m}{2}-1} -\phi_{\frac{m}{2}-2})\downarrow &  0&  0   & 0           &0     &0       &0      &1  &  0 & 0 & 0 & 0   &  0 & 0 & 0 & 0\\
\hline
(\rho_2 -\rho_1) \otimes 	 \chi_1\downarrow			&       &      &   &  &  &       &   &  0 & 1 & 0 & 0   &  0 & 1 & 0 & 0\\
(\rho_2 -\rho_1) \otimes ( \chi_2 -\chi_1)\downarrow		  	&       &      &   &  &  &       &   &  0 & 0 & 1 & 0   &  0 & 0 & 1 & 0\\
(\rho_2 -\rho_1) \otimes ( \chi_3 -\chi_1)\downarrow		 	&       &      &   &  &  &       &   &  0 & 0 & 0 & 0   &  0 & 0 & 0 & 0\\
(\rho_2 -\rho_1) \otimes (\chi_4 +\chi_3 -\chi_2 -\chi_1)\downarrow 	&       &      &   & \mathbf{0}  & &       &   &  0 & 0 & 0 & 0   &  0 & 0 & 0 & 0\\
(\rho_2 -\rho_1) \otimes (\phi_1 -\chi_2 -\chi_1)\downarrow	  	&       &      &   &  &  &       &   &  0 & 0 & 0 & 0   &  0 & 0 & 0 & 0\\ \vdots \quad
\rho_1 \otimes 	(\phi_p -\phi_{p-1} )\downarrow
		    \quad \vdots 	 				&       &      &   &  &  &       &   &\vdots&\vdots&\vdots&\vdots&\vdots&\vdots&\vdots&\vdots\\ 
\rho_1 \otimes (\phi_{\frac{m}{2}-1} -\phi_{\frac{m}{2}-2})\downarrow   &       &      &   &  &  &       &   &  0 & 0 & 0 & 0   &  0 & 0 & 0 & 0\\
	\end{array} 
\]
\normalsize
Again, we can discard the rows and columns which have at most one entry $\pm 1$ (and all other entries zero).
This reduces the above vertex block to the matrix
 $$\left(\begin{array}{ccccccccccccccccc}
  1    &  0   & 0           &0     &0       &\pm 1& 0   &\pm 1 &  0\\
  0    &  1   & 0           &1     &0       &  0  &\pm 1&  0   &\pm 1\\
  0    &  1   & 1           &0     &0       &  0  & 0   &  0   &  0\\
  0    &  -1  & 0           &0     &0       &  0  & 0   &  0   &  0\\
  0    &  0   & 1           &0     &1       &  0  & 0   &  0   &  0
\end{array}\right),$$
for which we can easily check that it has all its minors in $\{0, \pm 1\}$. This completes the verification of the vertex block condition in the
case of vertices with stabilizer $\Delta(2,2,m) = C_2 \times D_m$ for $m \geq 3$.

For the finitely many remaining stabilizer types, we can proceed case-by-case: we input each vertex block into a computer routine which 
computes all of its minors. Such a routine is straightforward to implement and takes approximately two seconds per vertex block on a 
standard computer. The authors' implementation is available at \textsf{http://math.uni.lu/\textasciitilde{}rahm/vertexBlocks/}. Note that for 
the groups under consideration, the matrix rank of the vertex block is at most $7$, so the $8 \times 8$-minors are all zero, and it is enough 
to compute the $n \times n$-minors for $n \leq 7$. This computer check verifies the minor condition for the vertex blocks associated to
all remaining vertex stabilizers, and completes the proof of the theorem.
\end{proof}

\begin{cor}
For any Coxeter group $\Gamma$ having a system of finite subgroups of types $\Delta(2,2,2) = (C_2)^3$,  $\Delta(2,2,m) = C_2 \times D_m$ for $m \geq 3$, $S_4$, $S_4 \times C_2$ or $A_5 \times C_2$ as vertex stabilizers,
 we have that the Bredon homology group $\hzero$ is torsion-free.
\end{cor}

\begin{remark} 
\begin{enumerate}[(a)]
\item When trying to extend the proof of Theorem~\ref{thm:basetransf} to $\hn$ for $n > 0$, 
one should take into account the natural map $\hn \to H_n(\underbar{B}\Gamma ; \Z)$ described by Mislin~\cite{MV03}, which is an isomorphism for $n > \dim \underbar{E}\Gamma^{\rm sing}+1$,
where $\underbar{E}\Gamma^{\rm sing}$ consists of the non-trivially stabilized points in $\underbar{E}\Gamma$.
Hence such an extension of the theorem can only be useful when $n \leq \dim \underbar{E}\Gamma^{\rm sing}+1$.

\item Note that the search for suitable base transformations for a given group $\Gamma$ (as described in Appendix~\ref{AppendixA} in our case), 
can be quite laborious. If the reader wants to apply Theorem~\ref{thm:basetransf} for a given group $\Gamma$, it is prudent to first construct the 
vertex blocks without any base transformation and compute their elementary divisors. 
If there exists a suitable simultaneous base transformation which satisfies the hypotheses of Theorem~\ref{thm:basetransf}, then those elementary 
divisors must be in the set $\{-1, 0, 1\}$. 
\end{enumerate}
\end{remark}


\section{Further examples with torsion-free $\hzero$}\label{sec:further-examples}

In this section we briefly steer away from Coxeter groups, and instead
give some further examples illustrating our criterion for the Bredon homology 
group $\hzero$ to be torsion-free.

\subsection{The Heisenberg semidirect product group}\label{Heisenberg}
Let us show that $\hzero$ is torsion-free for $\Gamma$ the Heisenberg semidirect product group of L\"uck's paper.
In Tables~\ref{CharTableC2Hei} and~\ref{CharTableC4}, we transform the character tables of all the non-trivial finite subgroups 
of the Heisenberg semidirect product group, as identified by L\"uck~\cite{L05}. 
\medskip 

\begin{table}[!h]
\[ \left(
	\begin{array}{c|rr}
		 C_2 & e & s\\ 
		 \hline 
		 \rho_1 & 1 & 1\\
		 \rho_2 & 1 & -1
	\end{array}
  \right)
  \mapsto
  \left(
	\begin{array}{c|rr}
		 C_2 & e & s\\ 
		 \hline 
		 \rho_1 +\rho_2 & 2 & 0\\
		 \rho_2 & 1 & -1
	\end{array}
  \right)
\] 
\caption{Character table of the cyclic group $C_2$ of order 2, with generator $s$.}\label{CharTableC2Hei}
\end{table}

\medskip

\begin{table}[!h]
\[ \left(
	\begin{array}{c|rrrr}
		 C_4 & e & s & s^2 & s^3\\ 
		 \hline 
		 \rho_1 & 1 & 1 & 1 & 1\\
		 \rho_2 & 1 & -1& 1 & -1\\
		 \rho_3 & 1 & i & -1& -i\\
		 \rho_4 & 1 &-i & -1& i\\
	\end{array}
	\right)
	\mapsto
	 \left(
	\begin{array}{c|rrrr}
		 C_4 & e & s & s^2 & s^3\\ 
		 \hline 
		 \rho_1 & 1 & 1 & 1 & 1\\
	 \rho_2-\rho_1  & 0 & -2& 0 & -2\\
	 \rho_3-\rho_1  & 0 & i-1 & -2& -i-1\\
      \rho_4-\rho_3	& 0 &-2i& 0& 2i\\
      \end{array}
	\right)
\] 
\caption{Character table of the cyclic group $C_4$ of order $4$, with generator $s$. We let $i^2 = -1$.}\label{CharTableC4}
\end{table}

\medskip

In Tables~\ref{C2toC4}, \ref{C2triviallytoC4} and~\ref{C1toC4}, 
we compute all possible induction homomorphisms $R_\mathbb{C} \left( H \right) \to R_\mathbb{C} \left( G \right)$ appearing in any possible Bredon chain complex.

\begin{table}[!h]
\[ 	\begin{array}{c|rr|cc}
C_2 \hookrightarrow C_4 		& e & s^2 & (\cdot | \rho_1 +\rho_2) & (\cdot | \rho_2)\\ 
		 \hline 
		 \rho_1\downarrow 	& 1 & 1   & 1				& 0\\
	 (\rho_2-\rho_1)\downarrow  	& 0 &  0  & 0				& 0\\
	 (\rho_3-\rho_1)\downarrow  	& 0 & -2  & 0				& 1\\
 (\rho_4+\rho_3-\rho_2-\rho_1)\downarrow& 0 & 0   & 0				& 0\\
	\end{array}
\] 
\caption{The only non-trivial inclusion $C_2 \hookrightarrow C_4$ of a cyclic group of order 2 into a cyclic group of order 4: $s \mapsto s^2.$}\label{C2toC4}
\end{table}

\begin{table}[!h]
\[ 	\begin{array}{c|rr|cc}
C_2 \hookrightarrow C_4 		& e & e & (\cdot | \rho_1 +\rho_2) & (\cdot | \rho_2)\\ 
		 \hline 
		 \rho_1\downarrow 	& 1 & 1   & 1				& 0\\
	 (\rho_2-\rho_1)\downarrow  	& 0 & 0   & 0				& 0\\
	 (\rho_3-\rho_1)\downarrow  	& 0 & 0   & 0				& 0\\
 (\rho_4+\rho_3-\rho_2-\rho_1)\downarrow& 0 & 0   & 0				& 0\\
	\end{array}
\] 
\caption{The trivial inclusion $C_2 \hookrightarrow C_4$ of a cyclic group of order~2 into a cyclic group of order 4: $s \mapsto e.$}\label{C2triviallytoC4}
\end{table}

\begin{table}[!h]
\[ 	\begin{array}{c|r|c}
C_2 \hookrightarrow C_4 		& e & (\cdot | \tau) \\ 
		 \hline 
		 \rho_1\downarrow 	& 1 & 1\\
	 (\rho_2-\rho_1)\downarrow  	& 0 & 0\\
	 (\rho_3-\rho_1)\downarrow  	& 0 & 0\\
 (\rho_4+\rho_3-\rho_2-\rho_1)\downarrow& 0 & 0\\
	\end{array}
\] 
\caption{The only inclusion $\{1\} \hookrightarrow C_4$ 
of the trivial group into a cyclic group of order 4.}\label{C1toC4}
\end{table}

Obviously, any concatenation of copies of the three matrices given in Tables~\ref{C2toC4}, \ref{C2triviallytoC4} and~\ref{C1toC4} yields a matrix with all of its minors contained in the set $\{-1, 0,1\}$.
For the inclusions into cyclic groups of order $2$, an analogous (and even simpler) procedure works.
\textit{Hence by Theorem~\ref{thm:basetransf}, $\hzero$ is torsion-free for $\Gamma$ the Heisenberg semidirect product group of L\"uck's article~\cite{L05}.}

\subsection{Crystallographic groups} \label{crystallographic}
Davis and L\"uck \cite{DL13} consider the semidirect product of $\Z^n$ with the cyclic $p$-group $\Z/p$,
where the action of $\Z/p$ on $\Z^n$ is given by an integral representation,
which is assumed to act freely on the complement of zero. 
The action of this semidirect product group $\Gamma$ on $\underbar{E}\Gamma \cong \R^n$ is crystallographic, with 
$\Z^n$ acting by lattice translations, and $\Z/p$ acting with a single fixed point.
In particular, all cell stabilizers are trivial except for one orbit of vertices of stabilizer type $\Z/p$. 
So all maps in the Bredon chain complex are induced by the trivial representation, and we can 
easily apply Theorem~\ref{thm:basetransf} to see that $\hzero$ is torsion-free for $\Gamma$.


\section{$cf(\Gamma)$ and $\chi(\mathcal C)$ from the geometry of $\mathcal P$}\label{sec:geom}

Let $\Gamma$ be the reflection group of the compact 3-dimensional hyperbolic polyhedron $\mathcal P$. In this 
final section, 
we compute the number of conjugacy classes of elements of finite order of $\Gamma$, $cf(\Gamma)$, and the Euler 
characteristic of the Bredon chain complex (\ref{eqn:BredonChainComplex}), $\chi(\mathcal C)$, from the geometry of the 
polyhedron $\mathcal P$. This gives us explicit combinatorial formulas for the 
Bredon homology and equivariant $K$-theory groups computed in our {Main Theorem}. 

\subsection{Conjugacy classes of elements of finite order}\label{sec:cf}
We now give an algorithm to calculate $cf(\Gamma)$, the number of conjugacy classes of elements of finite order in the 
Coxeter group $\Gamma$. We know that each element of finite order can be conjugated to one which 
stabilizes one of the $k$-dimensional faces of the polyhedron, for some $k \in \{0,1,2\}$. Of course,
the only element which stabilizes {\it all} faces is the identity element. Let us set that aside, and consider
the non-identity elements, to which we associate the integer $k$. We now count the elements according
to the integer $k$, in descending order.

\vskip 5pt

\noindent \underline{Case $k=2$:} These are the conjugacy classes represented by the canonical 
generators of the Coxeter group $\Gamma$. The number of these is given by the total number $|\mathcal P^{(2)}|$ 
of facets of the polyhedron $\mathcal P$.

\vskip 5pt

\noindent \underline{Case $k=1$:} These elements are edge stabilizers which are not conjugate
to the stabilizer of a face. We first note that there are some possible conjugacies between edge stabilizers. 
Geometrically, these occur when there is a geodesic $\gamma \subset \mathbb H^3$ whose projection onto
the fundamental domain $\mathcal P$ covers multiple edges inside the $1$-skeleton $\mathcal P ^{(1)}$. 
A detailed analysis of when this can happen is given in~\cite{LO09}. Following the description in that paper,
we decompose the $1$-skeleton into equivalence classes of edges, where two edges are equivalent if there
exists a geodesic whose projection passes through both edges. Denote by $[\mathcal P^{(1)}]$ the set of
equivalence classes of edges, and note that each equivalence class $[e]$ has a well defined group associated to
it, which is just the dihedral group $\Gamma_e$ stabilizing a representative edge. We can thus count the conjugacy 
classes in the corresponding dihedral group, and discard the three conjugacy classes already accounted for 
(the conjugacy class of the two canonical generators counted in case $k=2$, as well as the identity). Thus the 
contribution from finite elements of this type is given by
$$\sum_{[e]\in [\mathcal P^{(1)}]} (c(\Gamma_e) - 3).$$
(Recall that $c(D_m)$, the number of conjugacy classes in a dihedral group of order $2m$, is $m/2+3$ if $m$ even, and $(m-1)/2+2$ if $m$ is odd.)

\vskip 5pt

\noindent \underline{Case $k=0$:} Finally, we consider the contribution from the elements in the vertex stabilizers
{\it which have not already been counted}. That is to say, for each vertex $v\in \mathcal P^{(0)}$, we count the
conjugacy classes of elements in the corresponding $3$-generated spherical triangle group, which cannot be
conjugated into one of the canonical $2$-generated special subgroups. This number, $\bar c(\Gamma_v)$,
depends only on the isomorphism type of the spherical triangle group $\Gamma_v$, see Table~\ref{table:conjugacyspherical}. The contribution from these types of finite elements is thus $$\sum_{v\in \mathcal P^{(0)}} \bar c(\Gamma_v)\,.$$ 
\begin{table}[!ht]
\[
	\begin{array}{c|cc}
		 \Gamma_v & c(\Gamma_v) & \bar c(\Gamma _v)\\ 
		 \hline 
		 \Delta(2,2,m) & 2\,c(D_m) & c(D_m) -3 \\
		 \Delta(2,3,3)  & 5 & 1\\
		 \Delta(2,3,4)  & 10 & 3\\
		 \Delta(2,3,5)  & 10 & 5
	\end{array}
\] 
\caption{Number of conjugacy classes in spherical triangle groups. 
The left column is the total number (cf.~Appendix~\ref{AppendixA}), 
and the right column the number of those not conjugated into one of the three canonical 2-generated special subgroups.} \label{table:conjugacyspherical}
\end{table}

\vskip 5pt

Combining all these, we obtain the desired combinatorial formula for the number of conjugacy classes of elements
of finite order inside the group $\Gamma$:
$$cf(\Gamma)=1+|\mathcal P^{(2)}| + \sum_{[e]\in [\mathcal P^{(1)}]} (c(\Gamma_e) - 3) + \sum_{v\in \mathcal P^{(0)}} \bar c(\Gamma_v)\,.$$

\subsection{Euler characteristic.}
The Euler characteristic of the Bredon chain complex can be easily calculated from the
stabilizers of the various faces of the polyhedron $\mathcal P$, according to the formula:
$$\chi(\mathcal C) = \sum_{f \in \mathcal P} (-1)^{\dim(f)}\dim(R_{\mathbb{C}}(\Gamma_f))\,.$$
Depending on the dimension of the faces, we know exactly what the dimension of the complex representation ring is (the number of conjugacy classes in the stabilizer):
\begin{itemize}
\item for the $3$-dimensional face (the interior), the stabilizer is trivial, so there is a $1$-dimensional complex representation ring;
\item for the $2$-dimensional faces, the stabilizer are $\mathbb Z_2$, and there is a $2$-dimensional complex representation ring;
\item for the $1$-dimensional faces $e$, the stabilizers are dihedral groups, and there is a $c(\Gamma_e)$-dimensional complex representation ring;
\item for the $0$-dimensional faces $v$, the stabilizers are spherical triangle groups, and there is a $c(\Gamma_v)$-dimensional complex representation ring.
\end{itemize}
Putting these together, we obtain
$$\chi(\mathcal C) = -1+2|\mathcal P^{(2)}| - \sum_{e\in \mathcal P^{(1)}} c(\Gamma_e) + \sum_{v\in \mathcal P^{(0)}}  c(\Gamma_v)\,.$$
Finally, we obtain the desired explicit version of the {Main Theorem}, expressing the $K$-theory groups in terms of the geometry of the polyhedron $\mathcal P$. 
\begin{mainthm}[\bf explicit]
Let $\Gamma$ be a cocompact $3$-dimensional hyperbolic reflection group, generated by reflections in the side of a hyperbolic polyhedron $\mathcal{P} \subset \mathbb{H}^3$. Then 
$K_0(C_r^*(\Gamma))$ is a torsion-free abelian group of rank
$$cf(\Gamma)=1+|\mathcal P^{(2)}| + \sum_{[e]\in [\mathcal P^{(1)}]} (c(\Gamma_e) - 3) + \sum_{v\in \mathcal P^{(0)}} \bar c(\Gamma_v)\,,$$
and $K_1(C_r^*(\Gamma))$ is a torsion-free abelian group of rank
$$cf(\Gamma)-\chi(\mathcal C) = 2 - |\mathcal P^{(2)}| + \sum_{[e]\in [\mathcal P^{(1)}]} (c(\Gamma_e) - 3) +\sum_{e\in \mathcal P^{(1)}} c(\Gamma_e) - \sum_{v\in \mathcal P^{(0)}} (c(\Gamma_v) -\bar c(\Gamma_v))\,,$$
where the values for the $c(\Gamma_v)$ and $\bar{c}(\Gamma_v)$ are listed in Table~\ref{table:conjugacyspherical}. 
\end{mainthm}

\bibliographystyle{plain}

\appendix
\section{Character tables and base transformations}\label{AppendixA}
In this Appendix, we list the character tables of all the groups involved in the Bredon chain complex (\ref{BredonChainComplex}), that is, the finite Coxeter subgroups 
of $\Gamma$ up to rank three. 
These are based on the representation theory described in~\cite{JamesLiebeck},
where all these character tables are constructed. 
In the character tables below, rows correspond to irreducible representations, and columns to representatives of conjugacy classes, written in term of the Coxeter generators $s_1, \ldots, s_n$ in a fixed Coxeter presentation of $\Gamma$, as in (\ref{CoxeterPresentation}). 

In addition, for each character table, we apply elementary row operations to obtain the transformed tables needed for Appendix~\ref{AppendixB}, which are in turn used 
in our proof that $H_0$ is torsion-free (Section~\ref{sec:H0-algebra}). 
Although the rows of the transformed tables are not irreducible characters, it is easy to check that they still constitute bases for the complex representation rings.

Note that, for consistency across subgroups, we will pay particular attention to the order of the Coxeter generators within a subgroup. We write $e$ for the identity 
element in $\Gamma$. 

\subsection{Rank 0} This is the trivial group, with character table given below. 
\begin{table}[!ht]
\[
	\begin{array}{c|c}
		 & e\\ \hline \tau & 1
	\end{array}
\] 
\medskip\caption{Character table of the trivial group.}\label{CharTableTrivial}
\end{table}

\subsection{Rank 1} A rank 1 Coxeter group is a cyclic group of order 2. 
Write $s_i$ for its Coxeter generator, then its character table is Table~\ref{CharTableC2}.

\subsection{Rank 2} \label{sec:AppendixARank2}
A finite rank 2 Coxeter group with Coxeter generators $s_i$ and $s_j$ is a dihedral group of order $m=m_{ij} \ge 2$,
\begin{equation}\label{eqn:PresentationDihedral}
	D_{m} = \langle s_i, s_j \; | \; s_i^2=s_j^2=(s_is_j)^{m} \rangle\,.
\end{equation}
The character table of this group is given in Table~\ref{CharTableDm}.
In order to be consistent, we assume the Coxeter generators are ordered so that $i < j$. If $j < i$, then the character table is identical except that the third and fourth rows (the characters $\widehat{\chi}_3$ and $\widehat{\chi_4}$) are interchanged, since $(s_js_i)^r=(s_is_j)^{-r}$ and $s_i(s_js_i)^r=s_j(s_is_j)^{-r-1}$.

For the case $m=2$, that is $D_2 = C_2 \times C_2$, we will sometimes use the notation coming from the character table of $C_2$ (Table~\ref{CharTableC2}) instead. (Recall that the irreducible characters of a direct product $G \times H$ are obtained from the irreducible characters of $G$ and $H$ as $\rho_i \otimes \tau_j$, where $\left(\rho_i \otimes \tau_j\right) (g,h)=\rho_i(g) \cdot \tau_j(h)$.)
This gives the notation and characters in Table~\ref{CharTableC2xC2}, which are equivalent to Table~\ref{CharTableDm} with $\rho_1 \otimes \rho_1 = \chi_1$, $\rho_1 \otimes \rho_2 = \chi_4$, $\rho_2 \otimes \rho_1 = \chi_3$ and $\rho_2 \otimes \rho_2 = \chi_2$. As before, we assume $i<j$, or, if $j<i$, the second and third rows (characters) must be interchanged. 
\begin{table}[!ht]
\[
	\begin{array}{c|rrrr}
		 C_2 \times C_2 & e & s_i & s_j & s_is_j\\ 
		 \hline 
		 \rho_1 \otimes \rho_1 & 1 &  1 & 1 & 1\\
		 \rho_1 \otimes \rho_2 & 1 &  1 & -1 & -1\\
		 \rho_2 \otimes \rho_1 & 1 &  -1 & 1 & -1\\
		 \rho_2 \otimes \rho_2 & 1 &  -1 & -1 & 1
	\end{array}
\] 
\medskip\caption{Alternative character table of $\langle s_i, s_j \rangle \cong D_2 = C_2 \times C_2$, $i < j$.}\label{CharTableC2xC2}
\end{table}

We now give the base transformations of the character table of $D_m$ needed later, shown in Tables~\ref{CharTableD2}, \ref{CharTableDm m odd} and~\ref{CharTableDm m even}. 

\begin{table}[!htb]
\[
		\begin{array}{c|cccc}
		D_2 & e & s_i s_j & s_i & s_j\\ 
		 \hline 
		 \sum\chi_i& 4 & 0 & 0 & 0\\ 
		 \chi_2 +{\chi_3} & 2 & 0 & -2 & 0 \\
		 {\chi_3} & 1 & -1 & -1 & 1 \\
        {\chi_3}+{\chi_4} & 2 & -2 & 0 & 0 \\
	\end{array}
\]
\medskip\caption{Base transformation of the character table of $D_2$.} \label{CharTableD2}
\end{table} 

\begin{table}[!htb]
\[
\begin{array}{c|ccccc}
		 D_m        & e    & s_i  & s_i s_j                 & (s_i s_j)^r  &(s_i s_j)^{\frac{m-1}{2}}\\ 
		 \hline 
\chi_1+\chi_2+2\sum_{p = 1}^\frac{m-1}{2} \phi_p& 2m   &  0   & 0  & \hdots & 0 \\
      \chi_2+ \sum_{p = 1}^\frac{m-1}{2} \phi_p & m    & -1   & 0  & \hdots & 0\\
              \sum_{p = 1}^\frac{m-1}{2} \phi_p & m-1  &  0   & -1 & \hdots & -1\\\vdots\quad
\sum_{p = k}^\frac{m-1}{2}\phi_p \quad\vdots    &m-2k+1&\vdots& a_{k,1}  & a_{k,r} & a_{k,\frac{m-1}{2}}\\ 		 
	\phi_{\frac{m-1}{2}}  			& 2    &  0   & b_{1} & b_{r} & b_{\frac{m-1}{2}} \\
	\end{array}
\]
\medskip
\caption{Base transformation of the character table of $D_m$, $m \geq 3$ odd. Here, $a_{k,r} := \sum_{p = k}^\frac{m-1}{2} 2\cos(\frac{2\pi p r}{m})$, $b_{r} :=2\cos(\frac{\pi(m-1)r}{m}) $, and $1< k, r < \frac{m-1}{2}$.}\label{CharTableDm m odd}
\end{table}

\begin{table}[!htb]
\[
\begin{array}{c|cccccc}
		 D_m        				& e    & s_i  & s_i s_j                 &(s_i s_j)^r&(s_i s_j)^{\frac{m}{2}} & s_j s_i s_j\\ 
		 \hline 
\sum_{p=1}^{4}\chi_p +2\sum_{p=1}^{\frac{m}{2}-1}\phi_p & 2m   &  0   & 0  			& \hdots        	& 0 & 0\\
\chi_2+ \chi_3 +\sum_{p=1}^{\frac{m}{2}-1}\phi_p  & m    &  0   & 0  			& \hdots 	    	& 0 & -2\\
              \chi_3 +\sum_{p=1}^{\frac{m}{2}-1}\phi_p  & m-1  &  1   & -1		 	& \hdots 	 	& -1 & -1\\
      \chi_2+ \chi_4 +\sum_{p=1}^{\frac{m}{2}-1}\phi_p  & m    & -2   & 0  			&\hdots  		&  0 & 0\\              
\sum_{p=1}^{\frac{m}{2}-1}\phi_p			& m-2  &  0   & 0 			& -1-(-1)^r & -1-(-1)^{\frac{m}{2}} & 0\\\vdots\quad
\sum_{p = k}^{\frac{m}{2}-1}\phi_p           \quad\vdots&m-2k  &\vdots& a_{k,1}   & a_{k,r} &2\sum_{p = k}^{\frac{m}{2}-1}(-1)^p  & 0\\ 		 
	\phi_{\frac{m}{2}-1} 				& 2    &  0   & b_1 & b_r &2(-1)^{\frac{m}{2}-1}  & 0\\
	\end{array}
\]
\medskip
\caption{Base transformation of the character table of $D_m$, $m \geq 4$ even. Here, $a_{k,r} := \sum_{p = k}^{\frac{m}{2}-1} 2\cos(\frac{2\pi p r}{m})$, $b_{r} :=2\cos(\frac{\pi(m-2)r}{m}) $, $1 < k < \frac{m}{2}-1$ and $1 < r < \frac{m}{2}$.}\label{CharTableDm m even}
\end{table}

\subsection{Rank 3} A finite rank 3 Coxeter subgroup is one of the spherical triangle groups $\Delta(2,2,m)$, with $m \ge 2$, $\Delta(2,3,3)$, $\Delta(2,3,4)$, or $\Delta(2,3,5)$, 
using the notation of Equation~(\ref{eqn:TriangleGroups}), or, more compactly, the Coxeter diagrams in Figure~\ref{fig:CoxeterDiagrams}.

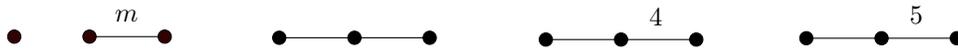
\begin{figure}[!ht]
\definecolor{ttqqqq}{rgb}{0.2,0,0}\begin{tikzpicture}[line cap=round,line join=round,>=triangle 45,x=1cm,y=1cm]\draw (-1,0)-- (0,0);\draw (1.5228505978878246,-0.01518665482800333)-- (2.5228505978878246,-0.01518665482800333);\draw (2.5228505978878246,-0.01518665482800333)-- (3.5228505978878246,-0.01518665482800333);\draw (5.06662787912817,-0.03876535251766991)-- (6.06662787912817,-0.03876535251766991);\draw (6.06662787912817,-0.03876535251766991)-- (7.066627879128174,-0.03876535251766991);\draw (8.529955421143631,-0.02118267329724504)-- (9.529955421143631,-0.02118267329724504);\draw (9.529955421143631,-0.02118267329724504)-- (10.529955421143635,-0.02118267329724504);
\begin{normalsize}\draw [fill=ttqqqq] (-2,0) circle (2.5pt);\draw [fill=ttqqqq] (-1,0) circle (2.5pt);\draw [fill=ttqqqq] (0,0) circle (2.5pt);\draw[color=black] (-0.5063520482015159,0.2802981955398016) node {$m$};\draw [fill=black] (1.5228505978878246,-0.01518665482800333) circle (2.5pt);\draw [fill=black] (2.5228505978878246,-0.01518665482800333) circle (2.5pt);\draw [fill=black] (3.5228505978878246,-0.01518665482800333) circle (2.5pt);\draw [fill=black] (5.06662787912817,-0.03876535251766991) circle (2.5pt);\draw [fill=black] (6.06662787912817,-0.03876535251766991) circle (2.5pt);\draw [fill=black] (7.066627879128174,-0.03876535251766991) circle (2.5pt);\draw[color=black] (6.532825395586867,0.2655696116315269) node {4};\draw [fill=black] (8.529955421143631,-0.02118267329724504) circle (2.5pt);\draw [fill=black] (9.529955421143631,-0.02118267329724504) circle (2.5pt);\draw [fill=black] (10.529955421143635,-0.02118267329724504) circle (2.5pt);\draw[color=black] (9.99404261403141,0.2950267794480762) node {5};
\end{normalsize}\end{tikzpicture}
\caption{Coxeter diagrams of rank 3 spherical Coxeter groups.}\label{fig:CoxeterDiagrams}
\end{figure}

In these diagrams, the vertices represent Coxeter generators and the edges are labelled by $m_{ij}$, with the conventions: no edge if $m_{ij}=2$, and no label if $m_{ij}=3$. 

\subsubsection{$\Delta(2,2,2)$}
This triangle group is isomorphic to $C_2 \times C_2 \times C_2$ and we have irreducible characters $\rho_{abc} := \rho_a \otimes \rho_b \otimes \rho_c$, $a,b,c \in \{1,2\}$, from Table~\ref{CharTableC2}, listed in Table~\ref{CharTableC2xC2xC2} below, where $\rho_a \otimes \rho_b \otimes \rho_c (x) = \rho_a(x_1) \cdot \rho_b(x_2) \cdot \rho_c(x_3)$ for all $x=(x_1,x_2,x_3) \in C_2 \times C_2 \times C_2$.
\begin{table}[!ht]
\[
	\begin{array}{c|rrrrrrrr}
		 C_2 \times C_2 \times C_2 & e & s_i & s_j & s_k & s_is_j & s_is_k & s_js_k & s_is_js_k \\ 
		 \hline 
		 \rho_{111} := \rho_1 \otimes \rho_1 \otimes \rho_1 & 1 &  1 & 1 & 1 & 1 &  1 & 1 & 1\\
		 \rho_{112} := \rho_1 \otimes \rho_1 \otimes \rho_2 & 1 &  1 & 1 & -1 & 1 &  -1 & -1 & -1\\
		 \rho_{121} := \rho_1 \otimes \rho_2 \otimes \rho_1 & 1 &  1 & -1 & 1 & -1 &  1 & -1 & -1\\
		 \rho_{122} := \rho_1 \otimes \rho_2 \otimes \rho_2 & 1 &  1 & -1 & -1 & -1 &  -1 & 1 & 1\\
		 \rho_{211} := \rho_2 \otimes \rho_1 \otimes \rho_1 & 1 &  -1 & 1 & 1 & -1 &  -1 & 1 & -1\\
		 \rho_{212} := \rho_2 \otimes \rho_1 \otimes \rho_2 & 1 &  -1 & 1 & -1 & -1 &  1 & -1 & 1\\
		 \rho_{221} := \rho_2 \otimes \rho_2 \otimes \rho_1 & 1 &  -1 & -1 & 1 & 1 &  -1 & -1 & 1\\
		 \rho_{222} := \rho_2 \otimes \rho_2 \otimes \rho_2 & 1 &  -1 & -1 & -1 & 1 &  1 & 1 & -1\\
	\end{array} 
\] 
\medskip\caption{Character table of $\langle s_i, s_j,s_k \rangle \cong \Delta(2,2,2) = C_2 \times C_2 \times C_2$, $i<j<k$, from Table~\ref{CharTableC2}.}\label{CharTableC2xC2xC2}
\end{table}
Here we assume that we have ordered the Coxeter generators $s_i, s_j, s_k$ so that $i<j<k$. Finally, the base transformation of the character table of $\Delta(2,2,2)$ needed later is shown in Table~\ref{BTCharTableC2xC2xC2}.

\begin{table}[!htb]
\[
	\begin{array}{c|rrrrrrrr}
		 C_2 \times C_2 \times C_2 & e & s_i & s_j & s_k & s_is_j & s_is_k & s_js_k & s_is_js_k \\ 
		 \hline 
					 \rho_{111} & 1 &  1 & 1  & 1  & 1 &  1  & 1  & 1\\
                             \rho_{112} -\rho_{111} & 0 &  0 & 0  & -2 & 0 &  -2 & -2 & -2\\
                             \rho_{121} -\rho_{111}& 0 &  0  & -2 & 0  & -2 &  0  & -2 & -2\\
                             \rho_{122} -\rho_{121}& 0 &  0  & 0  & -2 & 0  &  -2 & 2  & 2\\
                             \rho_{211} -\rho_{111}& 0 &  -2 & 0  & 0  & -2 &  -2 & 0  & -2\\
                             \rho_{212} -\rho_{211}& 0 &  0  & 0  & -2 & 0  &  2  & -2 & 2\\
                             \rho_{221} -\rho_{121}& 0 &  -2 & 0  & 0  & 2  &  -2 & 0  & 2 	\\
                             \rho_{222} -\rho_{221}& 0 &  0  &  0 & -2 & 0 &  2 & 2 & -2\\
	\end{array}
\] 
\medskip
\caption{Base transformation of the character table of \mbox{$\langle s_i, s_j,s_k \rangle \cong \Delta(2,2,2) = C_2 \times C_2 \times C_2$,} $i<j<k$.
}\label{BTCharTableC2xC2xC2}
\end{table}

\subsubsection{$\Delta(2,2,m)$ with $m > 2$} This group is isomorphic to $C_2 \times D_m$, and has Coxeter presentation
\begin{equation*}
		\Delta(2,2,m) = \left\langle s_i, s_j, s_k \; | \; s_i^2,s_j^2,s_k^2,(s_is_j)^2,(s_is_k)^2,(s_js_k)^m \right\rangle\,,
\end{equation*}
where we have sorted the Coxeter generators $s_j$ and $s_k$ such that $j < k$ (the generator $s_i$ is uniquely determined from the presentation). As a direct product of two groups, the character table of this group can be obtained from those of $C_2$ (Table~\ref{CharTableC2}) and $D_m$ (Table~\ref{CharTableDm}). This is shown on Table~\ref{CharTableDelta22n}, where $T_{D_m}$ is the matrix of entries of the character table of $D_m$ (Table~\ref{CharTableDm}). As explained before, if $k<j$ one needs to swap the characters $\chi_3$ and $\chi_4$, that is, swap $\rho_1 \otimes \chi_3$ and $\rho_1 \otimes \chi_4$, and $\rho_2 \otimes \chi_3$ and $\rho_2 \otimes \chi_4$. 
\begin{table}[!ht]
\[
	\begin{array}{c|cc|cc}
		 \Delta(2,2,m) & (s_js_k)^r & s_k(s_js_k)^r & s_i (s_js_k)^r & s_i s_k(s_js_k)^r\\ 
		 \hline 
		 \rho_1 \otimes \chi_1 & & & & \\
		 \rho_1 \otimes \chi_2 & & & &\\
		 \rho_1 \otimes \widehat{\chi_3} & \multicolumn{2}{c|}{T_{D_m}} & \multicolumn{2}{c}{T_{D_m}}\\
		 \rho_1 \otimes \widehat{\chi_4} & & & &\\
		 \rho_1 \otimes \phi_p & & & &\\
		 \hline 
		 \rho_2 \otimes \chi_1 & & & & \\
		 \rho_2 \otimes \chi_2 & & & &\\
		 \rho_2 \otimes \widehat{\chi_3} & \multicolumn{2}{c|}{T_{D_m}} & \multicolumn{2}{c}{-T_{D_m}}\\
		 \rho_2 \otimes \widehat{\chi_4} & & & &\\
		 \rho_2 \otimes \phi_p & & & &\\
	\end{array}
\]
\medskip\caption{Character table of $\Delta(2,2,m)$, $m=m_{jk}>2$, $j<k$.}\label{CharTableDelta22n}
\end{table}

The corresponding base transformations for $\Delta(2,2,m) \cong C_2 \times D_m$ are given in Table~\ref{CharTableDmxC2 m odd} ($m$ odd), Table~\ref{CharTableDmxC2 m even} ($m \ge 6$ even not a power of 2), and Table~\ref{CharTableDmxC2 m power of 2} ($m \ge 4$ a power of 2). 

\begin{table}[htb] \small
\[ 	
	\begin{array}{c|ccc|ccc}
D_m \times C_2 &  e   & s_i  & (s_i s_j)^r  &   \alpha    &\alpha s_i    & \alpha(s_i s_j)^r \\ 
		 \hline 
\rho_1 \otimes 	 \chi_1 				  	& 1 & 1 &  1 & 1 & 1 & 1 \\
\rho_1 \otimes 	( \chi_2 -\chi_1)                         & 0 & -2  & 0 & 0 &  -2  & 0 \\
\rho_1 \otimes 	(\phi_1 -\chi_2 -\chi_1)		   & 0 & 0 & b_r & 0 & 0 & b_r\\
\rho_1 \otimes 	(\phi_p -\phi_{p-1} ), 	  	& 0 & 0  & a_{p,r} &  0 & 0 & a_{p,r}  \\
\hline
(\rho_2 -\rho_1) \otimes 	 \chi_1			  &  0 & 0  & 1 &  -2 & -2 & -2 \\
(\rho_2 -\rho_1) \otimes ( \chi_2 -\chi_1)		  & 0  & 0 & 0 &   0 &  4 & 0  \\
(\rho_2 -\rho_1) \otimes (\phi_1 -\chi_2 -\chi_1)	  & 0 & 0 & 0 & 0 & -4  & -2b_r \\
(\rho_2 -\rho_1) \otimes (\phi_p -\phi_{p-1} ),		  & 0 & 0 & 0 &   0 &  0  & -2a_{p,r}  \\
\end{array}
\]
\medskip
\caption{Base transformation of the character table of $D_m \times C_2$ for $m\geq 3$ odd. 
Here, $a_{p,r} := 2\cos(\frac{2\pi pr}{m})-2\cos(\frac{2\pi (p-1)r}{m})$, 
$b_r := 2\cos(\frac{2\pi r}{m})-2$, $2 \leq p\leq \frac{m-1}{2}$ and $1 \le r \le m-1$.}\label{CharTableDmxC2 m odd}
\end{table}

\begin{table}[htb] \small
\[ 	
	\begin{array}{c|ccc|ccc}
		 D_m \times C_2 & s_i & (s_i s_j)^r & s_j s_i s_j  &\alpha s_i  & \alpha(s_i s_j)^r &\alpha s_j s_i s_j \\ 
\hline 
\rho_1 \otimes 	 \chi_1 				 & 1  & 1 & 1 		 &1 & 1 & 1  \\
\rho_1 \otimes 	( \chi_2 -\chi_1)                & -2  & 0 & -2 		 & -2  & 0 & -2 \\
\rho_1 \otimes 	( \chi_3 -\chi_2)                & 2  & c_r& 0 	 & 2 & c_r & 0\\
\rho_1 \otimes 	( \chi_4 -\chi_1)                & -2  & c_r & 0 & -2 & c_r & 0\\
\rho_1 \otimes 	(\phi_1 -\chi_3 -\chi_1)	 & -2 & b_r & 0 		 & -2 & b_r & 0\\
\vdots\quad \rho_1 \otimes (\phi_p -\phi_{p-1} ) \quad \vdots  
								& 0  &a_{p,r}& 0 	& 0 & a_{p,r} & 0 \\
\hline
(\rho_2 -\rho_1) \otimes 	 \chi_1		   &0&0&0		& -2 & -2 & -2 \\
(\rho_2 -\rho_1) \otimes ( \chi_2 -\chi_1)   &0&0&0	 	&  4 & 0  &4\\
(\rho_2 -\rho_1) \otimes ( \chi_3 -\chi_2)   &0&0&0		&-4&-2c_r& 0\\
(\rho_2 -\rho_1) \otimes ( \chi_4 +\chi_3 -\chi_2 -\chi_1)
								   &0&0&0		& 0 & -4c_r &0\\
(\rho_2 -\rho_1) \otimes (\phi_1 -\chi_2 -\chi_1)
								   &0&0&0		& 0 & -2(b_r+c_r) & 0\\
(\rho_2 -\rho_1) \otimes (\phi_p -\phi_{p-1} )		  
								 &0&0&0		& 0 & -2a_{p,r} & 0\\
\end{array} 
\]
\medskip\caption{Base transformation of the character table of $D_m \times C_2$ for $m \geq 6$ even, not a power of $2$. 
 Here, 
 \mbox{ $a_{p,r} := 2\cos(\frac{2\pi pr}{m})-2\cos(\frac{2\pi (p-1)r}{m})$}, 
\mbox{$b_r := 2\cos(\frac{2\pi r}{m})-(-1)^r-1$,}
\mbox{$c_r = (-1)^r -1$} and \mbox{$1 < p, r < \frac{m}{2}$}.}\label{CharTableDmxC2 m even}
\end{table}

\begin{table}[htb] \small
\[ 	
	\begin{array}{c|ccc|ccc}
		 D_m \times C_2 & s_i & (s_i s_j)^r & s_j s_i s_j  &\alpha s_i  & \alpha(s_i s_j)^r &\alpha s_j s_i s_j \\ 
\hline 
\rho_1 \otimes 	 \chi_1 				 & 1  & 1 & 1 		 &1 & 1 & 1  \\
\rho_1 \otimes 	( \chi_2 -\chi_1)                & -2  & 0 & -2 		 & -2  & 0 & -2 \\
\rho_1 \otimes 	( \chi_3 -\chi_2)                & 0  & c_r& -2 	 	& 0 & c_r & -2\\
\rho_1 \otimes 	( \chi_4 -\chi_1)                & 0  & c_r & 2 		& 0 & c_r & 2\\
\rho_1 \otimes 	(\phi_1 -\chi_3 -\chi_1)	 & 0 & b_r & 0 		 & 0 & b_r & 0\\
\vdots\quad \rho_1 \otimes (\phi_p -\phi_{p-1} ) \quad \vdots  
								& 0  &a_{p,r}& 0 	& 0 & a_{p,r} & 0 \\
\hline
(\rho_2 -\rho_1) \otimes 	 \chi_1		   &0&0&0		& -2 & -2 & -2 \\
(\rho_2 -\rho_1) \otimes ( \chi_2 -\chi_1)   &0&0&0	 	&  4 & 0  &4\\
(\rho_2 -\rho_1) \otimes ( \chi_3 -\chi_2)   &0&0&0		&0&-2c_r& 4\\
(\rho_2 -\rho_1) \otimes ( \chi_4 +\chi_3 -\chi_2 -\chi_1)
								   &0&0&0		& 0 & -4c_r &0\\
(\rho_2 -\rho_1) \otimes (\phi_1 -\chi_2 -\chi_1)
								   &0&0&0		& 0 & -2b_r & 0\\
(\rho_2 -\rho_1) \otimes (\phi_p -\phi_{p-1} )		  
								 &0&0&0		& 0 & -2a_{p,r} & 0\\
\end{array} 
\]
\medskip
\caption{Base transformation of the character table of \mbox{$D_m \times C_2$}
for $m \geq 4$ even and a power of $2$.
Here, 
\mbox{$a_{p,r} := 2\cos(\frac{2\pi pr}{m})-2\cos(\frac{2\pi (p-1)r}{m})$}, 
\mbox{$b_r := 2\cos(\frac{2\pi r}{m})-2$}, 
\mbox{$c_r = (-1)^r -1$} and $1 < p, r < \frac{m}{2}$.}\label{CharTableDmxC2 m power of 2}
\end{table}

\subsubsection{$\Delta(2,3,3)$} This triangle group has Coxeter presentation
\begin{equation}\label{eqn:PresentationDelta233}
	\Delta(2,3,3) = \left\langle s_i, s_j, s_k \; | \; s_i^2,s_j^2,s_k^2,(s_is_j)^3,(s_is_k)^2,(s_js_k)^3 \right\rangle\,.
\end{equation}
This group is, in standard notation, $A_3$, and it is isomorphic to the symmetric group $S_4$, with Coxeter generators $s_i = (1\ 2)$, $s_j = (2\ 3)$ and $s_k = (3\ 4)$, for instance. Then we have conjugacy classes (cycle types in $S_4$) represented, in terms of the Coxeter generators, by $s_i=(1\ 2) $, $s_i s_j = (1\ 3\ 2)$, $s_i s_j s_k = (1\ 4\ 3\ 2)$ and $s_i s_k = (1\ 2) (3\ 4)$. In particular, we have the character table for $\Delta(2,3,3)$ shown in Table~\ref{CharTableDelta233} below. In this case, the Coxeter generators are unique up to conjugation since $\text{Out}(S_4)$ is trivial, hence the choice of $s_i$ and $s_k$ does not affect the character table. 
\begin{table}[!ht]
\[
	\begin{array}{c|ccccc}
		 \Delta(2,3,3) & e & s_i & s_is_j & s_is_js_k & s_is_k\\ 
		 \hline 
		 \xi_1 & 1 &  1 & 1 & 1 & 1\\
		 \xi_2 & 1 &  -1 & 1 & -1 & 1\\
		 \xi_3 & 2 &  0 & -1 & 0 & 2\\
		 \xi_4 & 3 &  1 & 0 & -1 & -1\\
		 \xi_5 & 3 &  -1 & 0 & 1 & -1\\
	\end{array}
\]
\medskip\caption{Character table of $\langle s_i, s_j, s_k \rangle \cong \Delta(2,3,3) \cong S_4$.}\label{CharTableDelta233}
\end{table}

The required base transformation of the character table Table~\ref{CharTableDelta233} is given in Table~\ref{BTCharTableDelta233}.
\begin{table}[htb]
\[
	\begin{array}{c|ccccc}
		  S_4 & e & s_i & s_i s_j & s_i s_j s_k & s_i s_k\\ 
		 \hline 
		 \xi_1 & 1 &  1 & 1 & 1 & 1\\
		 \widetilde{\xi_2} := \xi_2 -\xi_1 & 0 & -2 & 0 & -2 & 0\\
		 \widetilde{\xi_3} := \xi_3 -\xi_2 -\xi_1 & 0 & 0 & -3 & 0 & 0\\
		 \widetilde{\xi_4} := \xi_4 -\xi_3 -\xi_1 & 0 & 0 & 0 & -2 & -4\\
		 \widetilde{\xi_5} := \xi_5 -\xi_4 -\xi_2 +\xi_1 & 0 & 0 & 0 & 4 & 0 \\
	\end{array}
\]
\medskip\caption{Base transformation of the character table of $\Delta(2,3,3) \cong S_4$.}\label{BTCharTableDelta233}
\end{table}

\subsubsection{$\Delta(2,3,4)$} This triangle group has Coxeter presentation
\begin{equation}\label{eqn:PresentationDelta234}
	\Delta(2,3,4) = \left\langle s_i, s_j, s_k \; | \; s_i^2,s_j^2,s_k^2,(s_is_j)^3,(s_is_k)^2,(s_js_k)^4 \right\rangle ,
\end{equation}
and the Coxeter generators are uniquely determined from the presentation. In standard notation, this is the finite Coxeter group $B_3$ (or $C_3$). This group is isomorphic to $S_4 \times C_2$ with Coxeter generators, for instance, $s_i = (1\ 2) \alpha$, $s_j = (1\ 3) \alpha$ and $s_k = (1\ 2) (3\ 4) \alpha$, where $\alpha$ is the generator of the $C_2$ factor. We can choose representatives of the conjugacy classes in terms of the Coxeter generators, for example, 
\begin{align*}
	e 			& \ \sim \ e			& \alpha 				&\ \sim \  (s_is_js_k)^3\\
	(1\ 2) 		&\ \sim \  s_i s_k 		& (1\ 2)\alpha			&\ \sim \    s_i\\
	(1\ 2 \ 3) 		&\ \sim \  s_i s_j 		& (1\ 2\ 3)\alpha 		&\ \sim \    s_is_js_k\\
	(1\ 2\ 3\ 4) 	&\ \sim \ s_j s_k 		& (1\ 2\ 3\ 4)\alpha 		&\ \sim \   s_i(s_js_k)^2\\
 	(1\ 2) (3\ 4) 	&\ \sim \  (s_j s_k)^2  	& (1\ 2) (3\ 4)\alpha  	&\ \sim \    s_k\\
\end{align*}
where `$\sim$' means `conjugate'. Using this representatives, and the fact that $\Delta(2,3,4)$ is a direct product, the character table of this group can be written as in Table~\ref{CharTableDelta234} below, where $T_{S_4}$ is the matrix of coefficients of the character table of $S_4$ (Table~\ref{CharTableDelta233}).
\begin{table}[!ht]
\[
	\begin{array}{c|ccccc|ccccc}
		 S_4\times C_2 & e & s_is_k & s_is_j & s_js_k & (s_js_k)^2 & (s_is_js_k)^3 & s_i & s_is_js_k & s_i(s_js_k)^2 & s_k\\ 
		 \hline 
		 \rho_1 \otimes \xi_1 &  &   &  &  & & & & &  \\
		 \rho_1 \otimes\xi_2 &  &   &  & & \\
		 \rho_1 \otimes\xi_3 & \multicolumn{5}{c|}{T_{S_4}} & \multicolumn{5}{c}{T_{S_4}}\\
		 \rho_1 \otimes\xi_4 &  &   &  &  & \\
		 \rho_1 \otimes\xi_5 &  &   &  &  & \\
		 \hline
		 \rho_2 \otimes \xi_1 &  &   &  &  & & & & &  \\
		 \rho_2 \otimes\xi_2 &  &   &  & & \\
		 \rho_2 \otimes\xi_3 & \multicolumn{5}{c|}{T_{S_4}} & \multicolumn{5}{c}{-T_{S_4}}\\
		 \rho_2 \otimes\xi_4 &  &   &  &  & \\
		 \rho_2 \otimes\xi_5 &  &   &  &  & 
	\end{array}
\]
\medskip\caption{Character table of $\langle s_i, s_j, s_k \rangle \cong \Delta(2,3,4) = S_4 \times C_2$.}\label{CharTableDelta234}
\end{table} 

The base transformation needed is given in Table~\ref{BTCharTableDelta234}, where $\{\widetilde{\xi_i}\}$ is the transformed basis of the character table of $S_4$, Table~\ref{BTCharTableDelta233}.

\begin{table}[!htb]
\footnotesize
\[ 
\begin{array}{c|ccccc|ccccc}
S_4\times C_2 & (1) & (12) & (123)& (1234) & (12)(34) & \alpha & \alpha(12) & \alpha(123)& \alpha(1234) & \alpha(12)(34)\\ 
\hline 
\alpha_1						& 1 & 1  & 1 & 1  & 1 & 1 & 1 & 1 & 1 & 1\\
\alpha_2			& 0 & -2 & 0 & -2 & 0 & 0 &-2 & 0 &-2 & 0 \\
\alpha_3			& 0 & 0  &-3 & 0  & 0 & 0 & 0 &-3 & 0 & 0\\
\alpha_4		& 0 &-2  & 0 & 4  & -4& 2 & 0 & 2 & 6 &-2\\
\alpha_5				& 0 & 0  & 0 & 4  & 0 & 0 & 0 & 0 & 4 & 0\\
\hline
\alpha_6   & 0 & 2  & 0 &  2  & 0 &-2 & 0 &-2 & 0 &-2 \\
\alpha_7 	& 0 &-4  & 0 & 0  & 0 & 0 & 0 & 0 & 4 & 0 \\
\alpha_8			& 0 & 0  & 0 & 0  & 0 & 0 & 0 & 6 & 0 & 0 \\
\alpha_9					& 0 & 0  & 0 & 6  &-4 & 4 & 0 & 4 &-6 &  8\\
\alpha_{10}			& 0 & 0  & 0 & 0  & 0 & 0 & 0 & 0 & -8& 0  
\end{array}
\]
where $\alpha_1 := \rho_1 \otimes \xi_1$, $\alpha_2 := \rho_1 \otimes\widetilde{\xi_2}$, $\alpha_3 := \rho_1 \otimes\widetilde{\xi_3}$, $\alpha_4 := \rho_1 \otimes (\widetilde{\xi_4}+2\widetilde{\xi_5}) - (\rho_2 \otimes{\xi_1} -\rho_1 \otimes{\xi_2})$, $\alpha_5 := \rho_1 \otimes\widetilde{\xi_5}$, 
$\alpha_6 := \rho_2 \otimes{\xi_1} -\rho_1 \otimes{\xi_2}$, 
$\alpha_7 := \rho_2 \otimes \xi_2 -\rho_2 \otimes\xi_1 +\rho_1 \otimes({\xi_5} -\xi_4)$, 
$\alpha_8 := (\rho_2 -\rho_1) \otimes\widetilde{\xi_3}$, 
$\alpha_9 := \rho_2 \otimes (\widetilde{\xi_4}+2\widetilde{\xi_5})+(\rho_1 -\rho_2) \otimes({\xi_1} +{\xi_2})$, 
$\alpha_{10} := (\rho_2 -\rho_1) \otimes\widetilde{\xi_5}$.
\medskip\caption{Base transformation of the character table of $\langle s_i, s_j, s_k \rangle \cong \Delta(2,3,4) = S_4 \times C_2$.}\label{BTCharTableDelta234}
\end{table}

\subsubsection{$\Delta(2,3,5)$} This group has Coxeter presentation
\begin{equation}\label{eqn:PresentationDelta235}
	\Delta(2,3,5) = \left\langle s_i, s_j, s_k \; | \; s_i^2,s_j^2,s_k^2,(s_is_j)^3,(s_is_k)^2,(s_js_k)^5 \right\rangle 
\end{equation}
and it is isomorphic to $A_5 \times C_2$ with Coxeter generators $s_i = (1\ 2) (3\ 5) \alpha$, $s_j = (1\ 2) (3\ 4) \alpha$, $s_k = (1\ 5) (2\ 3) \alpha$, for example. In standard notation, this is the exceptional finite Coxeter group $H_3$.
\begin{remark}\label{rmk:Coxgens}
These Coxeter generators are \emph{not} unique, since $\text{Out}(A_5 \times C_2) \cong C_2$. There are then two sets of Coxeter generators up to conjugation, the other one given by conjugation by a suitable $g \in S_5 \setminus A_5$, for instance, conjugating by $g=(2\ 4)$: $s_i' = (1\ 4) (3\ 5) \alpha$, $s_j' = (1\ 4) (2\ 3) \alpha$, $s_k' = (1\ 5) (3\ 4) \alpha$. In the first case $s_js_k = (1\ 3\ 4\ 2\ 5)$, conjugated to $(1\ 2\ 3\ 4\ 5)$, and in the second case $s_j's_k' = (1\ 3\ 2\ 4\ 5)$, conjugated to $(1\ 3\ 2\ 4\ 5)$, which represent different conjugacy classes in $A_5$. 
\end{remark}
A character table for the alternating group $A_5$ reads as follows.
\[
	\begin{array}{c|ccccc}
		 A_5 & e & (1\ 2\ 3) & (1\ 2) (3\ 4) & (1\ 2\ 3\ 4\ 5) & (1\ 3\ 4\ 5\ 2)\\ 
		 \hline 
		 \xi_1 & 1 &  1 & 1 & 1 & 1\\
		 \xi_2 & 4 &  1 & 0 & -1 & -1\\
		 \xi_3 & 5 &  -1 & 1 & 0 & 0\\
		 \xi_4 & 3 &  0 & -1 & \frac{1+\sqrt{5}}{2} & \frac{1-\sqrt{5}}{2}\\
		 \xi_5 & 3 &  0 & -1 & \frac{1-\sqrt{5}}{2} & \frac{1+\sqrt{5}}{2}\\
	\end{array}
\]
If we call $T_{A_5}$ the matrix of coefficients of this table then the character table of $\Delta(2,3,5)$ is given by Table~\ref{CharTableDelta235}, where we have used the following representatives of the conjugacy classes in terms of the Coxeter generators (`$\sim$' means `conjugate')
\begin{align*}
	e & \ \sim \ e	& \alpha &\ \sim \    (s_i s_j s_k)^5\\
	(1\ 2\ 3) &\ \sim \    s_i s_j 	& (1\ 2\ 3)\alpha &\ \sim \    s_i (s_j s_k)^2\\
	(1\ 2) (3\ 4) &\ \sim \    s_i s_k & (1\ 2) (3\ 4)\alpha &\ \sim \    s_i\\
	(1\ 2\ 3\ 4\ 5) &\ \sim \    s_j s_k & (1\ 2\ 3\ 4\ 5)\alpha  &\ \sim \   s_is_js_k\\
 	(1\ 2\ 3\ 5\ 4) &\ \sim \    (s_i s_j s_k)^4  & (1\ 2\ 3\ 5\ 4)\alpha  &\ \sim \    s_j s_i s_k\\
\end{align*}

\begin{table}[h!]
\small
\[
	\begin{array}{c|ccccc|ccccc}
		 A_5 \times C_2 & e & s_i s_j & s_i s_k & s_j s_k & (s_i s_j s_k)^4 & 
		 (s_i s_j s_k)^5 & s_i (s_j s_k)^2 & s_i & s_i s_j s_k & s_j s_i s_k\\
		 \hline 
		 \rho_1 \otimes \xi_1 &  &   &  &  & & & & &  \\
		 \rho_1 \otimes\xi_2 &  &   &  & & \\
		 \rho_1 \otimes\xi_3 & \multicolumn{5}{c|}{T_{A_5}} & \multicolumn{5}{c}{T_{A_5}}\\
		 \rho_1 \otimes\xi_4 &  &   &  &  & \\
		 \rho_1 \otimes\xi_5 &  &   &  &  & \\
		 \hline
		 \rho_2 \otimes \xi_1 &  &   &  &  & & & & &  \\
		 \rho_2 \otimes\xi_2 &  &   &  & & \\
		 \rho_2 \otimes\xi_3 & \multicolumn{5}{c|}{T_{A_5}} & \multicolumn{5}{c}{-T_{A_5}}\\
		 \rho_2 \otimes\xi_4 &  &   &  &  & \\
		 \rho_2 \otimes\xi_5 &  &   &  &  & 
	\end{array}
\]
\medskip\caption{Character table of $\langle s_i, s_j, s_k \rangle \cong \Delta(2,3,5) = A_5 \times C_2$.}\label{CharTableDelta235}
\end{table}

\begin{remark}
With the non-conjugate set of Coxeter generators $s_i', s_j', s_k'$ (see Remark~\ref{rmk:Coxgens}), we get the analogous set of representatives except swapping the obvious ones, that is, $s_j s_k \sim (s_i' s_j' s_k')^4$, $(s_i s_j s_k)^4 \sim s_j's_k'$, $( s_i s_j s_k)^5 s_j s_k \sim s_i's_j's_k'$ and $s_is_js_k \sim ( s_i' s_j' s_k')^5 s_j' s_k'$. On the character table this amounts to interchanging $\rho_i \otimes \xi_4$ with $\rho_i \otimes \xi_5$ for $i=1$ and $2$.
\end{remark}

For the base transformation required, first note that the character table for the alternating group $A_5$ above can be transformed as follows,
\[
	\begin{array}{c|ccccc}
		 A_5 & e & (1\ 2\ 3) & (1\ 2) (3\ 4) & (1\ 2\ 3\ 4\ 5) & (1\ 3\ 4\ 5\ 2)\\ 
		 \hline 
		          \xi_1                & 1 &  1 & 1 & 1 & 1\\
\widetilde{\xi_2}:=\xi_2 -4\xi_1               & 0 & -3 & -4& -5 & -5\\
\widetilde{\xi_3}:=\xi_3 -\xi_2 -\xi_1         & 0 &  -3 & 0 & 0 & 0\\
\widetilde{\xi_4}:= \xi_4 -\xi_2 +\xi_1        & 0 &  0 & 0 & \frac{5+\sqrt{5}}{2} & \frac{5-\sqrt{5}}{2}\\
\widetilde{\xi_5}:= \xi_5 +\xi_4 -\xi_1 -\xi_3 & 0 &  0 & -4 & 0 & 0\\
	\end{array}
\]
and then
\[  
	\begin{array}{c|ccccc}
		 A_5 & e & (1\ 2\ 3) & (1\ 2) (3\ 4) & (1\ 2\ 3\ 4\ 5) & (1\ 3\ 4\ 5\ 2)\\ 
		 \hline 
		          \xi_1                & 1 &  1 & 1 & 1 & 1\\
\widetilde{\widetilde{\xi_2}}:=\widetilde{\xi_2} -\widetilde{\xi_3} -\widetilde{\xi_5}               & 0 & 0 & 0& -5 & -5\\
\widetilde{\xi_3}        & 0 &  -3 & 0 & 0 & 0\\
\widetilde{\xi_4}       & 0 &  0 & 0 & \frac{5+\sqrt{5}}{2} & \frac{5-\sqrt{5}}{2}\\
\widetilde{\xi_5}& 0 &  0 & -4 & 0 & 0\\
	\end{array}
\]
With this notation, the required base transformation is given in Table~\ref{BTCharTableDelta235}.

\begin{table}[h!]
\footnotesize
\[ 
	\begin{array}{c|ccccc|ccccc}
		 A_5 \times C_2 & e & s_i s_j & s_i s_k & s_j s_k & (s_i s_j s_k)^4 & 
		 (s_i s_j s_k)^5 & s_i (s_j s_k)^2 & s_i & s_i s_j s_k & s_j s_i s_k\\
		 \hline 
\beta_1    & 1 &  1 & 1 & 1 & 1 & 1 &  1 & 1 & 1 & 1\\
\beta_2  & 0 & 0 & 0 & \sqrt{5} & -\sqrt{5}& 0 & 0  & 0& \sqrt{5} & -\sqrt{5}\\
\beta_3  & 0 &  -3 & 0 & 0 & 0 & 0 &  -3 & 0 & 0 & 0\\
\beta_4  & 0 &  0 & 0 & \frac{5+\sqrt{5}}{2} & \frac{5-\sqrt{5}}{2}& 0 &  0 & 0 & \frac{5+\sqrt{5}}{2} & \frac{5-\sqrt{5}}{2}\\
\beta_5  & 0 &  0 & -4 & 0 & 0& 4 &  4 & 0  & 4 & 4\\
\hline
\beta_6    & 0 &  0 & 0 & 0 & 0 & -2& -2 &-2 & -2& -2\\
\beta_7  & 0 & 0  & 0 & 0  & 0 & 0 & 0  & 0 & 10 & 10\\
\beta_8 & 0 &  0  & 0 & 0 & 0 & 0 &  6  & 0 & 0 & 0\\
\beta_9  & 0 &  0 & 0 & 0 & 0& 0 &  0 & 0 &-5-\sqrt{5} & -5+\sqrt{5}\\
\beta_{10} & 0 &  0 & 0 & 0 & 0& -8& -8 & 0 &-8 &-8\\
\end{array}
\]
where 
$\beta_1 :=\rho_1 \otimes \xi_1$,
$\beta_2 :=\rho_1 \otimes \widetilde{\widetilde{\xi_2}} +2\rho_1 \otimes \widetilde{\xi_4}$,
$\beta_3 :=\rho_1 \otimes \widetilde{\xi_3}$,
$\beta_4 :=\rho_1 \otimes \widetilde{\xi_4}$,
$\beta_5 :=\rho_1 \otimes \widetilde{\xi_5} -2 (\rho_2 -\rho_1) \otimes \xi_1$,
$\beta_6 :=(\rho_2 -\rho_1) \otimes \xi_1$,
$\beta_7 :=(\rho_2 -\rho_1) \otimes \widetilde{\widetilde{\xi_2}}$,
$\beta_8 :=(\rho_2 -\rho_1) \otimes \widetilde{\xi_3}$,
$\beta_9 :=(\rho_2 -\rho_1) \otimes \widetilde{\xi_4}$,
$\beta_{10}:=(\rho_2 -\rho_1) \otimes (\widetilde{\xi_5}+4 \xi_1)$.
\medskip\caption{Base transformation of the character table of $\langle s_i, s_j, s_k \rangle \cong \Delta(2,3,5) = A_5 \times C_2$.}\label{BTCharTableDelta235}
\end{table}

\section{Induction homomorphisms}\label{AppendixB}
In this Appendix, we compute all possible induction homomorphisms $R_\mathbb{C} \left( H \right) \to R_\mathbb{C} \left( G \right)$ appearing in the Bredon chain complex (\ref{BredonChainComplex}). That is, $G$ is a finite Coxeter subgroup of $\Gamma$ of rank $n$ generated by $n \le 3$ of the Coxeter generators (\ref{CoxeterPresentation}), and $H$ is a subgroup of $G$ generated by a subset of exactly $n-1$ of those Coxeter generators. 

We give explicit induction homomorphisms with respect to the standard character tables, and also with respect to the transformed bases in Appendix~\ref{AppendixA} for rank 3 subgroups, as this is needed for the simultaneous base transformation argument in the proof that $\hzero$ is torsion-free (Section~\ref{sec:H0-algebra}). 

We implicitly use the character tables and notation in Appendix~\ref{AppendixA}, and Frobenius reciprocity~\cite{Serre77}, throughout this Appendix. We also note that Frobenius reciprocity extends linearly:
\begin{lemma}
If $H$ is a subgroup of a finite group $G$ and $\phi$ and $\eta$, respectively $\tau$ and $\pi$, are representations of $G$, respectively $H$, then
$$
	(\phi \downarrow + \xi \downarrow | \tau +\pi) = (\phi + \xi  | \tau\uparrow +\pi\uparrow)\,.
$$
\end{lemma}
\begin{proof} 
\begin{align*}
(\phi \downarrow + \xi \downarrow | \tau +\pi)
 =& \frac{1}{|H|}\sum\limits_{h\in H} (\phi\downarrow + \xi\downarrow)(h) \cdot \overline{(\tau + \pi)(h)}\\
 =& \frac{1}{|H|}\sum\limits_{h\in H} (\phi\downarrow\cdot\overline{\tau} +\xi\downarrow\cdot\overline{\tau} +\phi\downarrow\cdot\overline{\pi} +\xi\downarrow\cdot \overline{\pi})(h) 
\end{align*}
which by Frobenius reciprocity on irreducible characters equals
\[
\frac{1}{|G|}\sum\limits_{g\in G} (\phi\cdot\overline{\tau}\uparrow +\xi\cdot\overline{\tau}\uparrow +\phi\cdot\overline{\pi}\uparrow +\xi \cdot \overline{\pi}\uparrow)(h) = (\phi + \xi  | \tau\uparrow +\pi\uparrow)\,. \qedhere
\]
\end{proof}

\subsection{Rank 1} The only induction homomorphism in this case is 
$$
	R_\mathbb{C} \left(\{ e \}\right) \to R_\mathbb{C} \left(\langle s_i \rangle\right)
$$
which must the regular representation $\tau \mapsto \rho_1 + \rho_2$ shown, in terms of free abelian groups, in Figure~\ref{InductionC2}.
\begin{figure}[h!]
\begin{eqnarray*}
	R_\mathbb{C}(\{ e \}) &\to& R_\mathbb{C}(C_2)\\
	\mathbb{Z} &\to & \mathbb{Z}^2 \\
	a &\mapsto & (a,a)
\end{eqnarray*}
\caption{Induction homomorphism from $H=\{ e \}$ to $G =\langle s_i \rangle \cong C_2$.}\label{InductionC2}
\end{figure}

\subsection{Rank 2} \label{Rank2inducedMaps} In this case $G$ is a dihedral group with the presentation
\[
	G = \langle s_i, s_j \; | \; s_i^2=s_j^2=(s_is_j)^m \rangle\,, 
\]
where $m=m_{ij}$ and we assume $i < j$. Consider first the case $H = \langle s_i \rangle$. The characters of $D_m$ (Table~\ref{CharTableC2xC2}) restricted to $H$ are (note that $s_i = s_j(s_is_j)^{m-1}$)
\[
	\begin{array}{c|rr}
		 D_m \downarrow & e & s_i \\ 
		 \hline 
		 \chi_1 & 1 & 1\\
		 \chi_2 & 1 & -1\\
		 \widehat{\chi_3} & 1 & -1\\
		 \widehat{\chi_4} & 1 & 1\\
		 \phi_p & 2 & 0
	\end{array}
\]
Multiplying with the rows of the character table of $H \cong C_2$ (Table~\ref{CharTableC2}) we obtain the induced representations
\[
	\begin{array}{rcl}
			\rho_1 \uparrow & = & \chi_1 + \widehat{\chi_4} + \sum \phi_p ,\\
			\rho_2 \uparrow & = & \chi_2 + \widehat{\chi_3} + \sum \phi_p .
	\end{array}
\]
The other case is when $H = \langle s_j \rangle$. This is analogous, but note that, in order to keep the notation consistent with Table~\ref{CharTableDm}, the characters $\chi_3$ and $\chi_4$ must be interchanged in the even case. Specifically, we have now $s_j = s_j(s_is_j)^0$ and hence
\[
	\begin{array}{c|rr}
		 D_m \downarrow & e & s_j \\ 
		 \hline 
		 \chi_1 & 1 & 1\\
		 \chi_2 & 1 & -1\\
		 \widehat{\chi_3} & 1 & 1\\
		 \widehat{\chi_4} & 1 & -1\\
		 \phi_p & 2 & 0
	\end{array}
\]
and therefore
\[
	\begin{array}{rcl}
			\rho_1 \uparrow & = & \chi_1 + \widehat{\chi_3} + \sum \phi_p ,\\
			\rho_2 \uparrow & = & \chi_2 + \widehat{\chi_4} + \sum \phi_p .
	\end{array}
\]
All in all, as maps of free abelian groups, we have that the induction homomorphisms $R_\mathbb{C} \left( H \right) \to R_\mathbb{C} \left( G \right)$ shown in Figure~\ref{InductionDm}.
\begin{figure}[h!]
\begin{eqnarray*}
	R_\mathbb{C}(H) &\to& R_\mathbb{C}(G)\\
	\mathbb{Z}^2 &\to & \mathbb{Z}^{c(D_{m})} \\
	(a,b) &\mapsto & (a,b,\widehat{b},\widehat{a},a+b,\ldots , a+b) \quad \text{for } H = \langle s_i \rangle,\\
	(a,b) &\mapsto & (a,b,\widehat{a},\widehat{b},a+b,\ldots , a+b) \quad \text{for } H = \langle s_j \rangle,
\end{eqnarray*}
\caption{Induction homomorphisms from $H$ to $G =\langle s_i, s_j \rangle \cong D_{m}$, $m=m_{ij}$ and $i<j$.}
\label{InductionDm}
\end{figure}

\subsection{Rank 3} 
We compute each case individually.

\subsubsection{$G = \Delta(2,2,2)$} This group is isomorphic to $C_2 \times C_2 \times C_2$, and has Coxeter generators $s_i$, $s_j$, $s_k$ where $i < j < k$. We compute the induction homomorphisms for the Coxeter subgroups $\langle s_i, s_j \rangle$, $\langle s_i, s_k \rangle$ and $\langle s_j, s_k \rangle$, all three direct factors of $G$ and isomorphic to $C_2 \times C_2$. Using the bases of $R_\mathbb{C}(C_2 \times C_2)$ and $R_\mathbb{C}(\Delta(2,2,2))$ induced from $C_2$ (Tables~\ref{CharTableC2xC2} and~\ref{CharTableC2xC2xC2}), we immediately obtain the induction homomorphisms shown, as maps of free abelian groups, shown in Figure~\ref{InductionDelta222}.
\begin{remark}
Recall how to the induction homomorphism works from a group $A$ to direct product $A \times B$: if $\rho$ is a representation of $A$ then $\text{Ind}_A^{A\times B}(\rho) = \rho \otimes r_B$, where $r_B$ is the regular representation of $B$. 
\end{remark}
\begin{figure}[h!]
\begin{eqnarray*}
	R_\mathbb{C}(H) &\to & R_\mathbb{C}(G)\\
	\mathbb{Z}^4 &\to & \mathbb{Z}^{8} \\
	(a,b,c,d) &\mapsto & (a,a,b,b,c,c,d,d) \quad \text{for } H = \langle s_i, s_j \rangle,\\
	(a,b,c,d) &\mapsto & (a,b,a,b,c,d,c,d) \quad \text{for } H = \langle s_i, s_k \rangle,\\
	(a,b,c,d) &\mapsto & (a,b,c,d,a,b,c,d) \quad \text{for } H = \langle s_j, s_k \rangle.
\end{eqnarray*}
\caption{Induction homomorphisms from $H$ to $G =\langle s_i, s_j, s_k \rangle \cong \Delta(2,2,2)=C_2\times C_2 \times C_2$, and $i<j<k$.}\label{InductionDelta222}
\end{figure}

On the other hand, with respect to the transformed bases (Tables~\ref{CharTableD2} and~\ref{BTCharTableC2xC2xC2} in Appendix~\ref{AppendixA}), the induction homomorphisms take the form shown in Tables~\ref{sjskToC2xC2xC2}, \ref{siskToC2xC2xC2}, and~\ref{D2toC2xC2xC2}. Note that, for the transformed bases, we show the restricted characters, and the induced map, in two adjacent tables separated by double vertical bars. 

\begin{table}[!htb]
\small
\[
	\begin{array}{c|rrrr||rrrr}
		\langle  s_j, s_k \rangle  \hookrightarrow C_2 \times C_2 \times C_2 
						    & e & s_j & s_k &  s_js_k  
& (\cdot | \sum \chi_i) & (\cdot | \chi_2+{\chi_3}) & (\cdot | {\chi_3}+{\chi_4}) & (\cdot | {\chi_4}) \\
		 \hline 
					 \rho_{111}\downarrow & 1 &   1 & 1   &  1  & 1 & 0 & 0 & 0\\
                             \rho_{112} -\rho_{111}\downarrow & 0 &   0 & -2  & -2 & 0 & 1 & 0  & 0\\
                             \rho_{121} -\rho_{111}\downarrow & 0 &  -2 & 0   & -2 & 0 & 0 & 1 & 1\\
                             \rho_{122} -\rho_{121}\downarrow & 0 &   0 & -2  &  2 & 0 & 1 & 0 & -1\\
                             \rho_{211} -\rho_{111}\downarrow & 0 &   0 & 0   &  0 & 0 & 0 & 0 & 0\\
                             \rho_{212} -\rho_{211}\downarrow & 0 &   0 & -2  & -2 & 0 & 1 & 0 & 0\\
                             \rho_{221} -\rho_{121}\downarrow & 0 &   0 & 0   &  0 & 0 & 0 & 0 & 0\\
                             \rho_{222} -\rho_{221}\downarrow & 0 &   0 & -2  &  2 & 0 & 1 & 0 & -1\\
	\end{array}
\] 
\medskip\caption{Restricted characters and induced map {$\langle s_j,s_k \rangle \hookrightarrow \Delta(2,2,2) = C_2 \times C_2 \times C_2$.}
}\label{sjskToC2xC2xC2}
\end{table}

\begin{table}[!htb]
\small
\[
	\begin{array}{c|rrrr||rrrr}
   \langle s_i,s_k \rangle \hookrightarrow C_2 \times C_2 \times C_2 
						    & e & s_i& s_k &s_is_k
& (\cdot | \sum \chi_i) & (\cdot | \chi_2+{\chi_3}) & (\cdot | {\chi_3}+{\chi_4}) & (\cdot | {\chi_4})
\\ 
		 \hline 
					 \rho_{111}\downarrow & 1 &  1 &  1  &   1  & 1 & 0 & 0 & 0\\
                             \rho_{112} -\rho_{111}\downarrow & 0 &  0 &  -2 &   -2 & 0 & 1 & 0 & 0\\
                             \rho_{121} -\rho_{111}\downarrow & 0 &  0 &   0 &    0 & 0 & 0 & 0 & 0\\
                             \rho_{122} -\rho_{121}\downarrow & 0 &  0 &  -2 &   -2 & 0 & 1 & 0 & 0\\
                             \rho_{211} -\rho_{111}\downarrow & 0 &  -2&  0  &   -2 & 0 & 0 & 1 & 1\\
                             \rho_{212} -\rho_{211}\downarrow & 0 &  0 &  -2 &   2  & 0 & 1 & 0 & -1\\
                             \rho_{221} -\rho_{121}\downarrow & 0 &  -2&   0 &   -2 & 0 & 0 & 1 & 1\\
                             \rho_{222} -\rho_{221}\downarrow & 0 &  0 &  -2 &    2 & 0 & 1 & 0 & -1\\
	\end{array} 
\] 
\medskip\caption{Restricted characters and induced map {
$\langle s_i,s_k \rangle \hookrightarrow \Delta(2,2,2) 
= C_2 \times C_2 \times C_2$.}
}\label{siskToC2xC2xC2}
\end{table}

\begin{table}[!htb]
\small
\[
	\begin{array}{c|rrrr||rrrr}
		 \langle s_i,s_j \rangle \hookrightarrow C_2 \times C_2 \times C_2 & e & s_i & s_j & s_is_j
		 & (\cdot | \sum \chi_i) & (\cdot | \chi_2+{\chi_3}) & (\cdot | {\chi_3}+{\chi_4}) & (\cdot | {\chi_4}) \\
		 \hline 
                          \rho_{111}\downarrow                 & 1 &  1  & 1  & 1  & 1 & 0 & 0 & 0\\
                 (\rho_{112} -\rho_{111})\downarrow            & 0 &  0  & 0  & 0  & 0 & 0 & 0 & 0\\
                 (\rho_{121} -\rho_{111})\downarrow            & 0 &  0  & -2 & -2 & 0 & 1 & 0 & 0\\
                 (\rho_{122} -\rho_{121})\downarrow            & 0 &  0  & 0  &  0 & 0 & 0 & 0 & 0\\
                 (\rho_{211} -\rho_{111})\downarrow            & 0 &  -2 & 0  & -2 & 0 & 0 & 1 & 1\\
                 (\rho_{212} -\rho_{211})\downarrow            & 0 &   0 & 0  &  0 & 0 & 0 & 0 & 0\\
                 (\rho_{221} -\rho_{121})\downarrow            & 0 &  -2 & 0  & 2  & 0 & 0 & 1 & 0\\
                 (\rho_{222} -\rho_{221})\downarrow            & 0 &   0 &  0 & 0  & 0 & 0 & 0 & 0\\
	\end{array} 
\] 
\medskip\caption{Restricted characters and induced map {$D_2 \cong \langle s_i, s_j \rangle \hookrightarrow \Delta(2,2,2) = C_2 \times C_2 \times C_2$.} 
}\label{D2toC2xC2xC2}
\end{table}

\subsubsection{$G = \Delta(2,2,m)$, $m>2$} \label{sec:Delta(2,2,m)}
This group is isomorphic to $C_2 \times D_m$ with Coxeter presentation 
\begin{equation*}
		\Delta(2,2,m) = \left\langle s_i, s_j, s_k \; | \; s_i^2,s_j^2,s_k^2,(s_is_j)^2,(s_is_k)^2,(s_js_k)^m \right\rangle\,,
\end{equation*}
We have three relevant Coxeter subgroups $H$, which we treat separately. In each case, we restrict the characters of $G$ (Table~\ref{CharTableDelta22n}) to the subgroup $H$ and then use Frobenius reciprocity to write the induced characters of $H$ into $G$ in terms of the characters of $G$.

\medskip

\noindent\underline{Case 1: $H = \langle s_i, s_j \rangle \cong D_2 = C_2 \times C_2$}.\\[1ex]
The elements $e, s_i, s_j$ and $s_i s_j$ of $H$ are obtained from the 1st, 3rd, 2nd and 4th column of Table~\ref{CharTableDelta22n} for $r$ equals $0, 0, n-1$ and $n-1$ respectively, giving the restrictions (by abuse of notation we indicate the restricted character with the same symbols):
\[
	\begin{array}{c|rr|rr}
		 C_2 \times D_m \downarrow & e & s_i & s_j & s_i s_j\\ 
		 \hline 
		 \rho_1 \otimes \chi_1 & 1 & 1 & 1 & 1\\
		 \rho_1 \otimes \chi_2 & 1 & 1 & -1 & -1\\
		 \rho_1 \otimes \widehat{\chi_3} & 1 & 1 & -1 & -1\\
		 \rho_1 \otimes \widehat{\chi_4} & 1 & 1 & 1 & 1\\
		 \rho_1 \otimes \phi_p & 2 & 2 & 0 & 0\\
		 \hline 
		 \rho_2 \otimes \chi_1 & 1 & -1 & 1 & -1\\
		 \rho_2 \otimes \chi_2 & 1 & -1 & -1 & 1\\
		 \rho_2 \otimes \widehat{\chi_3} & 1 & -1 & -1 & 1\\
		 \rho_2 \otimes \widehat{\chi_4} & 1 & -1 & 1 & -1\\
		 \rho_2 \otimes \phi_p & 2 & -2 & 0 & 0\\
	\end{array}
\]
Now we use Frobenius reciprocity (multiply the rows of the character tables of $H$ with those of $G \downarrow H$ above, and divide by $|H|=4$) to obtain the coefficients of the induced irreducible representations of $H$ in $G$ in terms of the irreducible representations of $G$. If $i < j$ then the character table of $H$ is the one given in Table~\ref{CharTableC2xC2}, but if $j <i$ we should interchange the second and third characters (rows) in that table. This gives 
\[
	\begin{array}{rcl}
			\rho_1 \otimes \rho_1 \uparrow & = & \rho_1 \otimes \chi_1 + \widehat{\rho_1 \otimes \chi_4} + \sum \rho_1 \otimes \phi_p\\
			\rho_1 \otimes \rho_2 \uparrow & = & \rho_1 \otimes \chi_2 + \widehat{\rho_1 \otimes \chi_3} + \sum \rho_1 \otimes \phi_p\\
			\rho_2 \otimes \rho_1 \uparrow & = & \rho_2 \otimes \chi_1 + \widehat{\rho_2 \otimes \chi_4} + \sum \rho_2 \otimes \phi_p\\
			\rho_2 \otimes \rho_2 \uparrow & = & \rho_2 \otimes \chi_2 + \widehat{\rho_2 \otimes \chi_3} + \sum \rho_2 \otimes \phi_p
	\end{array}
\]
if $i<j$, and
\[
	\begin{array}{rcl}
			\rho_1 \otimes \rho_1 \uparrow & = & \rho_1 \otimes \chi_1 + \widehat{\rho_1 \otimes \chi_4} + \sum \rho_1 \otimes \phi_p\\
			\rho_1 \otimes \rho_2 \uparrow & = & \rho_2 \otimes \chi_1 + \widehat{\rho_2 \otimes \chi_4} + \sum \rho_2 \otimes \phi_p\\
			\rho_2 \otimes \rho_1 \uparrow & = & \rho_1 \otimes \chi_2 + \widehat{\rho_1 \otimes \chi_3} + \sum \rho_1 \otimes \phi_p\\
			\rho_2 \otimes \rho_2 \uparrow & = & \rho_2 \otimes \chi_2 + \widehat{\rho_2 \otimes \chi_3} + \sum \rho_2 \otimes \phi_p
	\end{array}
\]
if $j < i$. Equivalently, as a map of free abelian groups, we have the homomorphisms shown in Figure~\ref{InductionDelta22mCase1}.
\begin{figure}[h!]
\begin{eqnarray*}
	R_\mathbb{C}\left(H \right) &\to& R_\mathbb{C}\left( G \right)\\
	\mathbb{Z}^4 &\to & \mathbb{Z}^{2\cdot c(D_{m})}\\
	(a,b,c,d) &\mapsto & (a,b,\widehat{b},\widehat{a},a+b,\ldots, a+b,\ c,d,\widehat{d},\widehat{c}, c+d,\ldots, c+d) \ \ \text{if } i < j,\\
	(a,b,c,d) &\mapsto & (a,c,\widehat{c},\widehat{a},a+c,\ldots, a+c,\ b,d,\widehat{d},\widehat{b}, b+d,\ldots, b+d) \ \ \text{if } j < i.
\end{eqnarray*}
\caption{Induction homomorphisms from $H = \langle s_i, s_j \rangle \cong C_2 \times C_2$ to $G =\langle s_i, s_j, s_k \rangle \cong \Delta(2,2,m)=C_2\times D_m$, $m=m_{jk}>2$, and $j<k$.}\label{InductionDelta22mCase1}
\end{figure}

\medskip

\noindent\underline{Case 2:  $H = \langle s_i, s_k \rangle \cong C_2 \times C_2$}.\\[1ex]
Restricting Table~\ref{CharTableDelta22n} to $H$ (1st, 3rd, 2nd and 4th column for $k$ equals 0) we get
\[
	\begin{array}{c|rr|rr}
		 C_2 \times D_m \downarrow & e & s_i & s_k & s_i s_k\\ 
		 \hline 
		 \rho_1 \otimes \chi_1 & 1 & 1 & 1 & 1\\
		 \rho_1 \otimes \chi_2 & 1 & 1 & -1 & -1\\
		 \rho_1 \otimes \widehat{\chi_3} & 1 & 1 & 1 & 1\\
		 \rho_1 \otimes \widehat{\chi_4} & 1 & 1 & -1 & -1\\
		 \rho_1 \otimes \phi_p & 2 & 2 & 0 & 0\\
		 \hline 
		 \rho_1 \otimes \chi_1 & 1 & -1 & 1 & -1\\
		 \rho_1 \otimes \chi_2 & 1 & -1 & -1 & 1\\
		 \rho_1 \otimes \widehat{\chi_3} & 1 & -1 & 1 & -1\\
		 \rho_1 \otimes \widehat{\chi_4} & 1 & -1 & -1 & 1\\
		 \rho_1 \otimes \phi_p & 2 & -2 & 0 & 0\\
	\end{array}
\]
Via Frobenius reciprocity we obtain, assumming first $i < k$,
\[
	\begin{array}{rcll}
			& & \rho_1 \otimes \underline{\phantom{a}} & \rho_2 \otimes \underline{\phantom{a}}\\[1ex]
			\rho_1 \otimes \rho_1 \uparrow & = & \rho_1 \otimes \chi_1 + \widehat{\rho_1 \otimes \chi_3} + \sum \rho_1 \otimes \phi_p\\
			\rho_1 \otimes \rho_2 \uparrow & = & \rho_1 \otimes \chi_2 + \widehat{\rho_1 \otimes \chi_4} + \sum \rho_1 \otimes \phi_p\\
			\rho_2 \otimes \rho_1 \uparrow & = & & \rho_2 \otimes \chi_1 + \widehat{\rho_2 \otimes \chi_3} + \sum \rho_2 \otimes \phi_p\\
			\rho_2 \otimes \rho_2 \uparrow & = & & \rho_2 \otimes \chi_2 + \widehat{\rho_2 \otimes \chi_4} + \sum \rho_2 \otimes \phi_p
	\end{array}
\]
If $k < i$, the calculation is the same but we must interchange again the 2nd and 3rd generators. All in all, we have the homomorphisms shown in Figure~\ref{InductionDelta22mCase2} as maps between free abelian groups. 

\begin{figure}[h!]
\begin{eqnarray*}
	R_\mathbb{C}\left( H \right) &\to& R_\mathbb{C}\left( G \right)\\
	\mathbb{Z}^4 &\to & \mathbb{Z}^{2\cdot c(D_{m})}\\
	(a,b,c,d) &\mapsto& (a,b,\widehat{a},\widehat{b},a+b,\ldots, a+b,\ c,d,\widehat{c},\widehat{d}, c+d,\ldots, c+d)\ \ \text{if } i < k,\\
	(a,b,c,d) &\mapsto& (a,c,\widehat{a},\widehat{c},a+c,\ldots, a+c,\ b,d,\widehat{b},\widehat{d}, b+d,\ldots, b+d)\ \ \text{if } k < i.\\
\end{eqnarray*}
\caption{Induction homomorphisms from $H = \langle s_i, s_k \rangle \cong C_2 \times C_2$ to $G =\langle s_i, s_j, s_k \rangle \cong \Delta(2,2,m)=C_2\times D_m$, $m=m_{jk}>2$, and $j<k$.}\label{InductionDelta22mCase2}
\end{figure}

\medskip

\noindent\underline{Case 3: $H = \langle s_j, s_k \rangle \cong D_m$}\\[1ex]
In this case the condition $j < k$ already holds. Restricting Table~\ref{CharTableDelta22n} to $H$ (the first two columns)
\[
	\begin{array}{c|cc}
		 C_2 \times D_m \downarrow & (s_js_k)^r & s_k(s_js_k)^r\\ 
		 \hline 
		 \rho_1 \otimes \chi_1  \\
		 \rho_1 \otimes \chi_2 \\
		 \rho_1 \otimes \widehat{\chi_3} & \multicolumn{2}{c}{T_m} \\
		 \rho_1 \otimes \widehat{\chi_4} \\
		 \rho_1 \otimes \phi_p \\
		 \hline 
		 \rho_2 \otimes \chi_1 \\
		 \rho_2 \otimes \chi_2 \\
		 \rho_2 \otimes \widehat{\chi_3} & \multicolumn{2}{c}{T_m} \\
		 \rho_2 \otimes \widehat{\chi_4} \\
		 \rho_2 \otimes \phi_p \\
	\end{array}
\]
where $T_m$ are the coefficients of the character table of $D_m$ as in Table~\ref{CharTableDm}. This immediately gives
\[
	\begin{array}{rcl}
			\chi_i \uparrow & = & \rho_1 \otimes \chi_i + \rho_2 \otimes \chi_i \quad \text{for all $i$, and}\\
			\phi_p \uparrow & = & \rho_1 \otimes \phi_i + \rho_2 \otimes \phi_p\quad \text{for all $p$.}\\
	\end{array}
\]
Equivalently, as a map of free abelian groups, this induction homomorphism is the one given in Figure~\ref{InductionDelta22mCase3}.
\begin{figure}[h!]
\begin{eqnarray*}
	R_\mathbb{C}\left( H \right) &\to& R_\mathbb{C}\left( G \right)\\
	\mathbb{Z}^{c(D_{m})} &\to & \mathbb{Z}^{2\cdot c(D_{m})} \\
	(a,b,\widehat{c},\widehat{d}, r_1, \ldots, r_N) &\mapsto & (a,b,\widehat{c},\widehat{d}, r_1, \ldots, r_N, \ a,b,\widehat{c},\widehat{d}, r_1, \ldots, r_N).
\end{eqnarray*}
\caption{Induction homomorphisms from $H = \langle s_j, s_k \rangle \cong D_m$ to $G =\langle s_i, s_j, s_k \rangle \cong \Delta(2,2,m)=C_2\times D_m$, $m=m_{jk}>2$, and $j<k$.}\label{InductionDelta22mCase3}
\end{figure}

Finally, we compute the induction homomorphisms with respect to the transformed basis (Tables~\ref{CharTableD2}, \ref{CharTableDm m odd}, \ref{CharTableDm m even}, \ref{CharTableDmxC2 m odd}, \ref{CharTableDmxC2 m even}, \ref{CharTableDmxC2 m power of 2} in Appendix~\ref{AppendixA}), summarising the results in Tables~\ref{MapDmtoDmxC2}
to~\ref{MapD2toDmxC2 m power of 2}.

\begin{table}[!htb] 
\small
\[ 	
\begin{array}{c|cc||ccccc}
D_m \hookrightarrow D_m \times C_2 	&  s_i & (s_i s_j)^r 
	& \alpha_1 
	& \alpha_2 
	& \alpha_3
	& \beta_k\\
\hline 
\rho_1 \otimes 	 \chi_1  \downarrow			& 1 & 1		& 1 & 0 & 0 &0\\
\rho_1 \otimes 	( \chi_2 -\chi_1)  \downarrow   & -2 & 0 		& 0 & 1 & 0 &0\\
\rho_1 \otimes 	(\phi_1 -\chi_2 -\chi_1)  \downarrow 
	& 0 & b_r		& 0 & 0 & 1 & 0\\
\vdots\quad\rho_1\otimes(\phi_p -\phi_{p-1})\downarrow\quad\vdots 
	& 0 &a_{p,r}  					&0&0& 0 & \delta_{p,k} \\
\hline
(\rho_2 -\rho_1) \otimes 	 \chi_1	 \downarrow		  
	& 0 &  0          
	& 0 & 0 & 0 &0 \\
(\rho_2 -\rho_1) \otimes ( \chi_2 -\chi_1) \downarrow		  
	& 0 &  0          
	& 0 & 0 & 0 &0 \\
(\rho_2 -\rho_1) \otimes (\phi_1 -\chi_2 -\chi_1) \downarrow	  
	& 0 &  0           
	& 0 & 0 & 0 &0 \\
\vdots \quad (\rho_2 -\rho_1) \otimes (\phi_p -\phi_{p-1} ) \downarrow \quad \vdots
	& 0 &  0        
	& 0 & 0 & 0 &0 \\
\end{array}
\]
\text{where }
$\alpha_1 := \sum\limits_{\ell = 1}^2 \chi_\ell +2\sum\limits_{\ell =1}^{(m-1)/2} \phi_\ell$,
$\alpha_2 := \chi_2 +\sum\limits_{\ell =1}^{(m-1)/2} \phi_\ell,$
$\alpha_3 := \sum\limits_{\ell =1}^{(m-1)/2} \phi_\ell$,
$\beta_k :=  \sum\limits_{\ell =k}^{(m-1)/2} \phi_\ell $.
\medskip	 
\caption{Restricted characters and map induced by $D_m \hookrightarrow \Delta(2,2,m) = D_m \times C_2$ for $m\geq 3$ odd. Here $a_{p,r} := 2\cos(\frac{2\pi pr}{m})-2\cos(\frac{2\pi (p-1)r}{m})$, $b_r := 2\cos(\frac{2\pi r}{m})-2$, $\delta_{p,k}$ the Kronecker delta, $2 \le p, k \le \frac{m-1}{2}$ and $0 \le r \le \frac{m-1}{2}$.}
\label{MapDmtoDmxC2}
\end{table}

\begin{table}[!htb]
\small
	\[ 	
	\begin{array}{c|cccc||cccc}
D_2 \hookrightarrow D_m \times C_2 & e & s_i & \alpha & \alpha s_i
& \sum \chi_i & \chi_2 +\widehat{\chi_3} & \widehat{\chi_3}  & \widehat{\chi_3}+\widehat{\chi_4}\\ 
		 \hline 
\rho_1 \otimes 	 \chi_1  \downarrow				  & 1 &  1 & 1  & 1                 & 1 &  0 &  0 & 0\\
\rho_1 \otimes 	( \chi_2 -\chi_1)  \downarrow                     & 0 & -2 & 0  & -2	 	    & 0 &  0 &  0 & 1\\
\rho_1 \otimes 	(\phi_1 -\chi_2 -\chi_1)  \downarrow		  & 0 &  0 & 0  & 0                 & 0 &  0 &  0 & 0\\
\vdots\quad\rho_1\otimes(\phi_p -\phi_{p-1})\downarrow\quad\vdots& 0 &  0 & 0  & 0                 & 0 &  0 &  0 & 0 \\
\hline
(\rho_2 -\rho_1) \otimes 	 \chi_1	 \downarrow		  & 0 &  0 & -2 & -2                & 0 &  1 &  0 & 0\\
(\rho_2 -\rho_1) \otimes ( \chi_2 -\chi_1) \downarrow		  & 0 &  0 & 0  & 4   		    & 0 &  0 &  1 &0\\
(\rho_2 -\rho_1) \otimes (\phi_1 -\chi_2 -\chi_1) \downarrow	  & 0 &  0 & 0  & 0                 & 0 &  0 &  0 &0\\
\vdots\quad
(\rho_2 -\rho_1)\otimes(\phi_p -\phi_{p-1})\downarrow\quad\vdots & 0 &  0 & 0  & 0                 & 0 &  0 &  0 &0
\end{array}
\]
\medskip 
\caption{Restricted characters and map induced by the two inclusions $D_2 \hookrightarrow \Delta(2,2,m) = D_m \times C_2$, for $m \geq 3$ odd. Here $2 \leq p\leq \frac{m-1}{2}$.}
\label{MapD2toDmxC2}
\end{table}

\begin{table}[!htb] 
\footnotesize
\[
\begin{array}{c|ccc||cccccc}
D_m \hookrightarrow  D_m \times C_2 & s_i & (s_i s_j)^r  &  s_j s_i s_j 
& \alpha_1 
& \alpha_2
& \alpha_3
& \alpha_4
&  \alpha_5
& \beta_k \\ 
\hline 
\rho_1 \otimes 	 \chi_1 \downarrow
	 &  0              & 0			             & 1 &  1    &  0   & 0           &0     &0       & 0     \\
\rho_1 \otimes 	( \chi_2 -\chi_1)        \downarrow
	    &  -2  & 0                  & -2&  0    &  1   & 0           &1     &0            &0\\
\rho_1 \otimes 	( \chi_3 -\chi_2)        \downarrow
	    &  0   & c_r   & -2&  0    &  0   & 1           &-1    &0            &0\\
\rho_1 \otimes 	( \chi_4 -\chi_1)        \downarrow
	    &  -2  & c_r   & 0&  0    &  0   & 0           &1     &0            &0\\
\rho_1 \otimes 	(\phi_1 -\chi_3 -\chi_1) \downarrow
	    &  -2  & b_r & 0&  0    &  0   & 0           &1     &1            &0\\ 
\vdots \quad\rho_1 \otimes 	(\phi_p -\phi_{p-1} )\downarrow   \quad \vdots 	
	    &  0    & a_{p,r}   &  0&  0    &  0   & 0           &0     &0       & \delta_{p,k}\\ 
\hline
(\rho_2 -\rho_1) \otimes 	 \chi_1			  \downarrow  &  &  & \\
(\rho_2 -\rho_1) \otimes ( \chi_2 -\chi_1)		  \downarrow  &  &  & \\
(\rho_2 -\rho_1) \otimes ( \chi_3 -\chi_2)		  \downarrow  &  & \mathbf{0} &  & \multicolumn{6}{c}{\mathbf{0}}\\
(\rho_2 -\rho_1) \otimes (\chi_4 +\chi_3 -\chi_2 -\chi_1) \downarrow  &  &  & \\
(\rho_2 -\rho_1) \otimes (\phi_1 -\chi_2 -\chi_1)	  \downarrow &  &  & \\
\vdots \quad (\rho_2 -\rho_1) \otimes (\phi_p -\phi_{p-1} )  \downarrow \quad \vdots &  &  & \\
	\end{array} 
\]
where 
$\alpha_1 := \sum\limits_{\ell = 1}^4 \chi_\ell +2\sum\limits_p^{\frac{m}{2}-1}\phi_p$,
$\alpha_2 := \chi_2 +\chi_3+\sum\limits_{\ell=1}^{\frac{m}{2}-1}\phi_\ell $,
$\alpha_3 := \chi_3+\sum\limits_{\ell=1}^{\frac{m}{2}-1}\phi_\ell $,
$\alpha_4 := \chi_2 +\chi_4+\sum\limits_{\ell=1}^{\frac{m}{2}-1}\phi_\ell  $,
$\alpha_5 := \sum\limits_{\ell = 1}^{\frac{m}{2}-1}\phi_\ell$,
$\beta_k := \sum\limits_{\ell = k}^{\frac{m}{2}-1}\phi_\ell $.
\medskip
\caption{Restricted characters and map induced by the inclusion $D_m \hookrightarrow \Delta(2,2,m) = D_m \times C_2$, for $m \geq 6$ even, not a power of $2$. Here $a_{p,r} := 2\cos(\frac{2\pi pr}{m})-2\cos(\frac{2\pi (p-1)r}{m})$, $b_r := 2\cos(\frac{2\pi r}{m})-(-1)^r-1$, $c_r := (-1)^r -1$,  $\delta_{p,k}$ the Kronecker delta, $2 \leq p, k \leq \frac{m}{2}-1$, and $0 \le r \le \frac{m}{2}$.}
\label{MapDmtoDmxC2 m even}
\end{table}

\begin{table}[!htb]
\footnotesize
\[
	\begin{array}{c|cccc||cccc}
D_2 \hookrightarrow D_m \times C_2 &  e & s_j s_i s_j & \alpha    &\alpha s_j s_i s_j  
& \sum \chi_i & \chi_2 +\widehat{\chi_3} & \widehat{\chi_3}  & \widehat{\chi_3}+\widehat{\chi_4}\\
		 \hline 
\rho_1 \otimes 	 \chi_1 		 		 \downarrow & 1 & 1 & 1 & 1 & 1 & 0 & 0 & 0 \\
\rho_1 \otimes 	( \chi_2 -\chi_1)        		  \downarrow& 0 &-2 & 0 &-2 & 0 & 0 & 0 & 1\\
\rho_1 \otimes 	( \chi_3 -\chi_2)        		  \downarrow& 0 & 0 & 0 & 0 & 0 & 0 & 0 & 0\\
\rho_1 \otimes 	( \chi_4 -\chi_1)        		 \downarrow & 0 & 0 & 0 & 0 & 0 & 0 & 0 & 0\\
\rho_1 \otimes 	(\phi_1 -\chi_3 -\chi_1) 		 \downarrow & 0 & 0 & 0 & 0 & 0 & 0 & 0 & 0\\
\vdots\quad\rho_1 \otimes 	(\phi_p -\phi_{p-1} )\downarrow\quad\vdots 		    & 0 & 0 & 0 & 0 & 0 & 0 & 0 & 0\\
\hline
(\rho_2 -\rho_1) \otimes 	 \chi_1			  \downarrow& 0 & 0 &-2 &-2 & 0 & 1 & 0 & 0\\
(\rho_2 -\rho_1) \otimes ( \chi_2 -\chi_1)		  \downarrow& 0 & 0 & 0 & 4 & 0 & 0 & 1 & 0\\
(\rho_2 -\rho_1) \otimes ( \chi_3 -\chi_2)		  \downarrow& 0 & 0 & 0 & 0 & 0 & 0 & 0 & 0\\
(\rho_2 -\rho_1) \otimes ( \chi_4 +\chi_3 -\chi_2 -\chi_1)\downarrow& 0 & 0 & 0 & 0 & 0 & 0 & 0 & 0\\
(\rho_2 -\rho_1) \otimes (\phi_1 -\chi_2 -\chi_1)	  \downarrow& 0 & 0 & 0 & 0 & 0 & 0 & 0 & 0\\
\vdots\quad(\rho_2 -\rho_1) \otimes (\phi_p -\phi_{p-1} )\downarrow\quad\vdots	    & 0 & 0 & 0 & 0 & 0 & 0 & 0 & 0\\
\end{array} 
	\]
\medskip
\caption{Restricted characters and map induced by the two inclusions $D_2 \hookrightarrow \Delta(2,2,m) = D_m \times C_2$, for $m \geq 2$ even, not a power of $2$. Here $2 \leq p\leq \frac{m}{2}-1$. Our choice on the generators is $(12)(34) \mapsto s_j s_i s_j$ and $(12) \mapsto \alpha$; any other choice yields an equivalent matrix.}\label{MapD2toDmxC2 m even}
\end{table}

\begin{table}[!htb] 
\footnotesize
\[
\begin{array}{c|ccc||cccccc}
D_m \hookrightarrow  D_m \times C_2 & s_i & (s_i s_j)^r  &  s_j s_i s_j 
& \alpha_1 
& \alpha_2
& \alpha_3
& \alpha_4
&  \alpha_5
& \beta_k \\ 
\hline 
\rho_1 \otimes 	 \chi_1 \downarrow
	 	&  1      & 1   & 1 		&  1    &  0   & 0    &0     &0       & 0     \\
\rho_1 \otimes 	( \chi_2 -\chi_1)        \downarrow
	    &  -2  & 0   & -2			&  0    &  1   & 0     &1     &0	&0\\
\rho_1 \otimes 	( \chi_3 -\chi_2)        \downarrow
	    &  0   & c_r   & -2			&  0    &  1   & 1    &0    &0	&0\\
\rho_1 \otimes 	( \chi_4 -\chi_1)        \downarrow
	    &  0  & c_r   & 2			&  0    &  -1   & 0   &0     &0	&0\\
\rho_1 \otimes 	(\phi_1 -\chi_3 -\chi_1) \downarrow
	    &  0  & b_r & 0			&  0    &  0   & 1   & 0    &1	&0\\ 
\vdots \quad\rho_1 \otimes 	(\phi_p -\phi_{p-1} )\downarrow   \quad \vdots 	
	    &  0    & a_{p,r}   &  0		&  0    &  0   & 0   &0     &0       & \delta_{p,k}\\ 
\hline
(\rho_2 -\rho_1) \otimes 	 \chi_1			  \downarrow  &  &  & \\
(\rho_2 -\rho_1) \otimes ( \chi_2 -\chi_1)		  \downarrow  &  &  & \\
(\rho_2 -\rho_1) \otimes ( \chi_3 -\chi_2)		  \downarrow  &  & \mathbf{0} &  & \multicolumn{6}{c}{\mathbf{0}}\\
(\rho_2 -\rho_1) \otimes (\chi_4 +\chi_3 -\chi_2 -\chi_1) \downarrow  &  &  & \\
(\rho_2 -\rho_1) \otimes (\phi_1 -\chi_2 -\chi_1)	  \downarrow &  &  & \\
\vdots \quad (\rho_2 -\rho_1) \otimes (\phi_p -\phi_{p-1} )  \downarrow \quad \vdots &  &  & \\
	\end{array} 
\]
where 
$\alpha_1 := \sum\limits_{\ell = 1}^4 \chi_\ell +2\sum\limits_p^{\frac{m}{2}-1}\phi_p$,
$\alpha_2 := \chi_2 +\chi_3+\sum\limits_{\ell=1}^{\frac{m}{2}-1}\phi_\ell $,
$\alpha_3 := \chi_3+\sum\limits_{\ell=1}^{\frac{m}{2}-1}\phi_\ell $,
$\alpha_4 := \chi_2 +\chi_4+\sum\limits_{\ell=1}^{\frac{m}{2}-1}\phi_\ell  $,
$\alpha_5 := \sum\limits_{\ell = 1}^{\frac{m}{2}-1}\phi_\ell$,
$\beta_k := \sum\limits_{\ell = k}^{\frac{m}{2}-1}\phi_\ell $.
\medskip
\caption{Restricted characters and map induced by the inclusion $D_m \hookrightarrow \Delta(2,2,m) = D_m \times C_2$, for $m \geq 4$ a power of $2$. Here $a_{p,r} := 2\cos(\frac{2\pi pr}{m})-2\cos(\frac{2\pi (p-1)r}{m})$, $b_r := 2\cos(\frac{2\pi r}{m})-2$, $c_r := (-1)^r -1$,  $\delta_{p,k}$ the Kronecker delta, $2 \leq p, k \leq \frac{m}{2}-1 $ where $1 < r < \frac{m}{2}$ and $1 < k < \frac{m}{2}-1$.}
\label{MapDmtoDmxC2 m power of 2}
\end{table}

\begin{table}[!htb]
\small
\[
\begin{array}{c|cccc||cccc}
D_2 \hookrightarrow D_m \times C_2 &  e &  s_i & \alpha   &\alpha s_i  
& \sum \chi_i & \chi_2 +\widehat{\chi_3} & \widehat{\chi_3}  & \widehat{\chi_3}+\widehat{\chi_4}\\
		 \hline 
\rho_1 \otimes 	 \chi_1 		 \downarrow&  1    & 1  &  1    & 1     & 1 & 0 & 0 & 0\\
\rho_1 \otimes 	( \chi_2 -\chi_1)        \downarrow&  0    & -2 &  0    & -2    & 0 & 0 & 0 & 1\\
\rho_1 \otimes 	( \chi_3 -\chi_1)        \downarrow&  0    &  0 &  0    & 0     & 0 & 0 & 0 & 0\\
\rho_1 \otimes 	( \chi_4 -\chi_2)        \downarrow&  0    &  0 &  0    & 0     & 0 & 0 & 0 & 0\\
\rho_1 \otimes 	(\phi_1 -\chi_2 -\chi_1) \downarrow&  0    &  0 &  0    & 0     & 0 & 0 & 0 & 0\\
\vdots \quad\rho_1 \otimes 	(\phi_p -\phi_{p-1} )\downarrow \quad \vdots   &  0    &  0 &  0    & 0     & 0 & 0 & 0 & 0\\
\hline
(\rho_2 -\rho_1) \otimes 	 \chi_1			  \downarrow&0 &0&-2 &-2& 0 & 1 & 0 & 0\\
(\rho_2 -\rho_1) \otimes ( \chi_2 -\chi_1)		  \downarrow&0 &0& 0 &4 & 0 & 0 & 1 & 0\\
(\rho_2 -\rho_1) \otimes ( \chi_3 -\chi_1)		  \downarrow&0 &0& 0 &0 & 0 & 0 & 0 & 0\\
(\rho_2 -\rho_1) \otimes ( \chi_4 +\chi_3 -\chi_2 -\chi_1)\downarrow&0 &0& 0 &0 & 0 & 0 & 0 & 0\\
(\rho_2 -\rho_1) \otimes (\phi_1 -\chi_2 -\chi_1)	  \downarrow&0 &0& 0 & 0& 0 & 0 & 0 & 0\\
\vdots \quad(\rho_2 -\rho_1) \otimes (\phi_p -\phi_{p-1} )\downarrow \quad \vdots	    &0 &0& 0 &0 & 0 & 0 & 0 & 0\\
\end{array} 
\]
\medskip\caption{ Restricted characters and map induced by the two inclusions $D_2 \hookrightarrow \Delta(2,2,m) = D_m \times C_2$, for $m \geq 4$ a power of $2$. Here $2 \leq p\leq \frac{m}{2}-1$.
Our choice on the generators is $(12)(34) \mapsto s_i $ and $(12) \mapsto \alpha$, and any other choice yields an equivalent matrix. 
}\label{MapD2toDmxC2 m power of 2}
\end{table}

\subsubsection{$G = \Delta(2,3,3)$} This group is isomorphic to the symmetric group $S_4$ with Coxeter presentation
\begin{equation}\label{eqn:Delta233}
	\Delta(2,3,3) = \left\langle s_i, s_j, s_k \; | \; s_i^2,s_j^2,s_k^2,(s_is_j)^3,(s_is_k)^2,(s_js_k)^3 \right\rangle
\end{equation}
and we assume $i < k$. We have again three relevant Coxeter subgroups $H$.

\medskip

\noindent\underline{Case 1: $H = \langle s_i, s_j \rangle \cong D_3$}\\[1ex]
The expanded character table for $D_3$ (from Table~\ref{CharTableDm}), assuming first $i < j$, is
\[
	\begin{array}{c|cccccc}
		 D_3 & e & s_i & s_j & s_is_j & s_js_i & s_is_js_i\\ 
		 \hline 
		 \chi_1 & 1 &  1 & 1 &  1& 1 &  1 \\
		 \chi_2 & 1 & -1 & -1 &  1 & 1 & -1 \\
		 \phi_1 & 2 &  0 &  0 &  -1 &  -1 &  0 \\
	\end{array}
\] 
There are 3 conjugacy classes, $\{e\}$, $\{s_i,s_j,s_is_js_i=s_js_is_j\}$ and $\{s_is_j,s_js_i\}$, which remain unchanged if we swap $s_i$ and $s_j$, hence if $j < i$ the table stays the same and we do not have to treat those two cases separately. The character table of $G$ (Table~\ref{CharTableDelta233}) restricted to $H$ consists on the 1st, 2nd, 2nd, 3rd, 3rd, 2nd columns (since $s_i \sim s_j$, $s_js_i \sim s_is_j$ and $s_is_js_i \sim s_i$):
\[
	\begin{array}{c|cccccc}
		 S_4 \downarrow & e & s_i & s_j & s_is_j & s_js_i & s_is_js_i\\ 
		 \hline 
		 \xi_1 & 1 &  1 & 1 & 1 & 1 & 1\\
		 \xi_2 & 1 &  -1 & -1 & 1 & 1 & -1\\
		 \xi_3 & 2 &  0 & 0 & -1 & -1 & 0\\
		 \xi_4 & 3 &  1 & 1 & 0 & 0 & 1\\
		 \xi_5 & 3 &  -1 & -1 & 0 & 0 & -1\\
	\end{array}
\]
Multiplying rows and dividing by $|H|=6$ we obtain
\[
	\begin{array}{rcl}
			\chi_1 \uparrow & = & \xi_1 + \xi_4\\
			\chi_2 \uparrow & = & \xi_2 + \xi_5\\
			\phi_1 \uparrow & = & \xi_3 + \xi_4 + \xi_5\\
	\end{array}
\]
Equivalently, we have the map of free abelian groups shown in Figure~\ref{InductionDelta233Case1}.
\begin{figure}[h!]
\begin{eqnarray*}
	R_\mathbb{C}\left( H \right) &\to& R_\mathbb{C}\left( G \right)\\
	\mathbb{Z}^3 &\to & \mathbb{Z}^5 \\
	(a,b,c) &\mapsto & (a,b,c,a+c,b+c).
\end{eqnarray*}
\caption{Induction homomorphism from $H = \langle s_i, s_j \rangle \cong D_3$ or $H = \langle s_j, s_k \rangle \cong D_3$ to $G =\langle s_i, s_j, s_k \rangle \cong \Delta(2,3,3)=S_4$, and $i<k$.}\label{InductionDelta233Case1}
\end{figure}

\medskip

\noindent\underline{Case 2: $H = \langle s_j, s_k \rangle \cong D_3$}\\[1ex]
This case is completely analogous to the previous one, so we obtain the same map, also shown in Figure~\ref{InductionDelta233Case1}. 

\medskip

\noindent\underline{Case 3: $H = \langle s_i, s_k \rangle \cong D_2 = C_2 \times C_2$}\\[1ex]
The character table of $G$ (Table~\ref{CharTableDelta233}) restricted to $H$ consists on the 1st, 2nd, 2nd, 5th columns:
\[
	\begin{array}{c|cccc}
		 S_4 \downarrow & e & s_i & s_k & s_is_k\\ 
		 \hline 
		 \xi_1 & 1 &  1 & 1 & 1\\
		 \xi_2 & 1 &  -1 & -1 & 1\\
		 \xi_3 & 2 &  0 & 0 & 2\\
		 \xi_4 & 3 &  1 & 1 & -1\\
		 \xi_5 & 3 &  -1 & -1 & -1\\
	\end{array}
\]
Multiplying the rows with the characters of $H$ (Table~\ref{CharTableC2xC2}) and dividing by $|H|=4$ we obtain (note that $i<k$ already holds)
\[
	\begin{array}{rcl}
			\rho_1 \otimes \rho_1 \uparrow & = & \xi_1 + \xi_3 + \xi_4\\
			\rho_1 \otimes \rho_2 \uparrow & = & \xi_4 + \xi_5\\
			\rho_2 \otimes \rho_1 \uparrow & = & \xi_4 + \xi_5\\
			\rho_2 \otimes \rho_2 \uparrow & = & \xi_2 + \xi_3 + \xi_5
	\end{array}
\]
Equivalently, this is the homomorphism of abelian groups shown in Figure~\ref{InductionDelta233Case3}. Note that this is the first time that of an induction homomorphism with nontrivial kernel.
\begin{figure}[h!]
\begin{eqnarray*}
	R_\mathbb{C}\left( H \right) &\to& R_\mathbb{C}\left( G \right)\\
	\mathbb{Z}^4 &\to & \mathbb{Z}^5 \\
	(a,b,c,d) &\mapsto & (a,d,a+d,a+b+c, b+c+d).
\end{eqnarray*}
\caption{Induction homomorphism from $H = \langle s_i, s_k \rangle \cong C_2 \times C_2$ to $G =\langle s_i, s_j, s_k \rangle \cong \Delta(2,3,3)=S_4$, and $i<k$.}\label{InductionDelta233Case3}
\end{figure}

Now we give the induction homomorphisms with respect to the transformed bases (Tables~\ref{CharTableD2}, \ref{CharTableDm m odd}, \ref{BTCharTableDelta233} in Appendix~\ref{AppendixA}), summarised in Table~\ref{InducedOnS4}.

\begin{table}[!htb]
\[\begin{array}{c|ccc||cccc}
		 D_2 \hookrightarrow S_4  & (1) & (1\ 2) &  (1\ 2) (3\ 4)
		 & (\widetilde{\xi_i}| \sum \chi_i) & (\widetilde{\xi_i}| \chi_2 +\widehat{\chi_3}) & (\widetilde{\xi_i}| \widehat{\chi_3}) & (\widetilde{\xi_i}| \widehat{\chi_4})\\
		 \hline 
		 \widetilde{\xi_1} \downarrow & 1 &  1 &  1            & 1 & 0 & 0 & 0\\
		 \widetilde{\xi_2}\downarrow & 0 & -2 &  0 & 0 & 1 & 0 &  0\\
		 \widetilde{\xi_3}\downarrow & 0 & 0 &  0  & 0 & 0 & 0 &  0\\
		 \widetilde{\xi_4}\downarrow & 0 & 0 & -4  & 0 & 0 & 1 &  1\\
		 \widetilde{\xi_5}\downarrow & 0 & 0 & 0   & 0 & 0 & 0 &  0\\
\end{array}\]
\[\begin{array}{c|ccc||cccc}
		 D_3 \hookrightarrow S_4  & (1) & (12) & (123) 
		 & (\widetilde{\xi_i}| 2\phi_1 +\sum \chi_i) & (\widetilde{\xi_i}| \chi_2 +\phi_1) & (\widetilde{\xi_i}| \phi_1) \\ 
		 \hline 
		 \widetilde{\xi_1} \downarrow & 1 &  1 & 1 & 1 & 0 & 0\\
		 \widetilde{\xi_2} \downarrow & 0 & -2 & 0 & 0 & 1 & 0\\
		 \widetilde{\xi_3} \downarrow & 0 & 0 & -3 & 0 & 0 & 1\\
		 \widetilde{\xi_4} \downarrow & 0 & 0 & 0 & 0 & 0 & 0\\
		 \widetilde{\xi_5} \downarrow & 0 & 0 & 0 & 0 & 0 & 0\\
\end{array}\]
where 
$\widetilde{\xi_1} := \xi_1$, 
$\widetilde{\xi_2} := \xi_2 -\xi_1$, 
$\widetilde{\xi_3} := \xi_3 -\xi_2 -\xi_1$, 
$\widetilde{\xi_4} := \xi_4 -\xi_3 -\xi_1$, 
$\widetilde{\xi_5} := \xi_5 -\xi_4 -\xi_2 +\xi_1$.

\medskip

\caption{Restricted characters and map induced by the inclusions of $D_2$ and $D_3$ into $S_4$. The second inclusion 
$\langle (23), (34) \rangle \cong D_3 \hookrightarrow S_4$ induces the same map as the first one, 
because $(23) \sim (12)$ and $(234) \sim (123)$.}\label{InducedOnS4}
\end{table} 	

\subsubsection{$G = \Delta(2,3,4)$} This group is isomorphic to $S_4 \times C_2$ with Coxeter presentation
$$
	\Delta(2,3,4) = \left\langle s_i, s_j, s_k \; | \; s_i^2,s_j^2,s_k^2,(s_is_j)^3,(s_is_k)^2,(s_js_k)^4 \right\rangle .
$$
The three relevant induction homomorphism are as follows. 

\medskip

\noindent\underline{Case 1: $H = \langle s_i, s_k \rangle \cong D_2 = C_2 \times C_2$}\\[1ex]
The character table of $G$ (Table~\ref{CharTableDelta234}) restricted to $H$  consists on the 1st, 7th, 10th and 2nd columns:
\[
	\begin{array}{c|cccc}
		 S_4\times C_2 \downarrow & e & s_i & s_k & s_is_k\\ 
		 \hline 
		 \rho_1 \otimes \xi_1 & 1 & 1 & 1 & 1 \\
		 \rho_1 \otimes\xi_2 & 1 & -1 & 1 & -1 \\
		 \rho_1 \otimes\xi_3 & 2 & 0 & 2 & 0\\
		 \rho_1 \otimes\xi_4 & 3 & 1 & -1 & 1 \\
		 \rho_1 \otimes\xi_5 & 3 & -1 & -1 & -1 \\
		 \hline
		 \rho_2 \otimes \xi_1 & 1 & -1 & -1 & 1  \\
		 \rho_2 \otimes\xi_2 & 1 & 1 & -1 & -1\\
		 \rho_2 \otimes\xi_3 & 2 & 0 & -2 & 0\\
		 \rho_2 \otimes\xi_4 & 3 & -1 & 1 & 1 \\
		 \rho_2 \otimes\xi_5 & 3 & 1 & 1 & -1
	\end{array}
\]
Suppose first that $i < k$.
Multiplying these rows with the rows of Table~\ref{CharTableC2xC2} we deduce that (note the shortcut in notation)
\[
	\begin{array}{rcll}
			& & \rho_1 \otimes \underline{\phantom{a}} & \rho_2 \otimes \underline{\phantom{a}}\\[1ex]
			\rho_1 \otimes \rho_1 \uparrow & = & 
				\xi_1 + \xi_3 + \xi_4 + & \xi_4 + \xi_5 \\
			\rho_1 \otimes \rho_2 \uparrow & = & 
				\xi_4 + \xi_5 + & \xi_2 + \xi_3 + \xi_5 \\
			\rho_2 \otimes \rho_1 \uparrow & = & 
				\xi_2 + \xi_3 + \xi_5 + & \xi_4 + \xi_5 \\
			\rho_2 \otimes \rho_2 \uparrow & = & 
				\xi_4 + \xi_5 + & \xi_1 + \xi_3 + \xi_4\\	
	\end{array}
\]
On the other hand, if $k < i$, then we should interchange the 2nd and 3rd generators. All in all, we have the homomorphisms of free abelian groups shown in Figure~\ref{InductionDelta234Case1}.
\begin{figure}[h!]
\begin{eqnarray*}
	R_\mathbb{C}\left( H \right) &\to& R_\mathbb{C}\left( G \right)\\
	\mathbb{Z}^4 &\to & \mathbb{Z}^{10} \\
	(a,b,c,d) &\mapsto & (a,c,a+c,a+b+d,b+c+d, d, b, b+d, a+c+d, a+b+c)\ \ \text{if } i < k,\\
	(a,b,c,d) &\mapsto & (a,b,a+b,a+c+d,b+c+d, d, c, c+d, a+b+d, a+b+c)\ \ \text{if } k < i.\\
\end{eqnarray*}
\caption{Induction homomorphism from $H = \langle s_i, s_k \rangle \cong C_2 \times C_2$ to $G =\langle s_i, s_j, s_k \rangle \cong \Delta(2,3,4)=S_4 \times C_2$.}\label{InductionDelta234Case1}
\end{figure}

\medskip

\noindent\underline{Case 2: $H = \langle s_i, s_j \rangle \cong D_3$}\\[1ex]
The characters of $G$ (Table~\ref{CharTableDelta234}) restricted to $H$ consists on the 1st, 7th, 7th, 3rd, 3rd, 7th columns:
\[
	\begin{array}{c|cccccc}
		 S_4\times C_2 \downarrow & e & s_i & s_j & s_is_j & s_js_i & s_is_js_i\\ 
		 \hline 
		 \rho_1 \otimes\xi_1 & 1 & 1 & 1 & 1 & 1 & 1\\
		 \rho_1 \otimes\xi_2 & 1 & -1 & -1 & 1 & 1 & -1 \\
		 \rho_1 \otimes\xi_3 & 2 & 0 & 0 & -1 & -1 & 0\\
		 \rho_1 \otimes\xi_4 & 3 & 1 & 1 & 0 & 0 & 1 \\
		 \rho_1 \otimes\xi_5 & 3 & -1 & -1 & 0 & 0 & -1  \\
		 \hline
		 \rho_2 \otimes \xi_1 & 1 & -1 & -1 & 1 & 1 & -1 \\
		 \rho_2 \otimes\xi_2 & 1 & 1 & 1 & 1 & 1 & 1\\
		 \rho_2 \otimes\xi_3 & 2 & 0 & 0 & -1 & -1 & 0\\
		 \rho_2 \otimes\xi_4 & 3 & -1 & -1 & 0 & 0 & -1\\
		 \rho_2 \otimes\xi_5 & 3 & 1 & 1 & 0 & 0 & 1 
	\end{array}
\]
Multiplying these rows with the rows of character table of $D_3$ above we deduce that (recall that this table is invariant under interchanging the Coxeter generators)
\[
	\begin{array}{rcll}
			& & \rho_1 \otimes \underline{\phantom{a}} & \rho_2 \otimes \underline{\phantom{a}}\\[1ex]
			\chi_1 \uparrow & = & \xi_1 + \xi_4 + & \xi_2 + \xi_5 \\
			\chi_1 \uparrow & = & \xi_2 + \xi_5 + & \xi_1 + \xi_4\\
			\phi_1 \uparrow & = & \xi_3 + \xi_4 + \xi_5 + & \xi_3 + \xi_4 + \xi_5\\
	\end{array}
\]
or, equivalently, the linear map shown in Figure~\ref{InductionDelta234Case2}.
\begin{figure}[h!]
\begin{eqnarray*}
	R_\mathbb{C}\left( H \right) &\to& R_\mathbb{C}\left( G \right)\\
	\mathbb{Z}^3 &\to & \mathbb{Z}^{10} \\
	(a,b,c) &\mapsto & (a,b,c,a+c,b+c, b, a, c, b+c, a+c).
\end{eqnarray*}
\caption{Induction homomorphism from $H = \langle s_i, s_j \rangle \cong D_3$ to $G =\langle s_i, s_j, s_k \rangle \cong \Delta(2,3,4)=S_4 \times C_2$.}\label{InductionDelta234Case2}
\end{figure}

\medskip

\noindent\underline{Case 3: $H = \langle s_j, s_k \rangle \cong D_4$}\\[1ex]
First we expand the character table of $D_4$, assuming $j < k$,
\[
	\begin{array}{c|cccccccc}
		 D_4 & e & s_j & s_k & s_js_k & s_ks_j & s_js_ks_j& s_ks_js_k & s_js_ks_js_k\\ 
		 \hline 
		 \chi_1 & 1 &  1 & 1 &  1& 1 &  1 & 1 &  1\\
		 \chi_2 & 1 & -1 & -1 &  1 & 1 & -1 & -1 &  1 \\
		 \chi_3 & 1 & -1 & 1 &  -1 & -1 & 1 & -1 &  1 \\
		 \chi_4 & 1 & 1 & -1 &  -1 & -1 & -1 & 1 &  1 \\
		 \phi_1 & 2 &  0 &  0 &  0 &  0 &  0 & 0 & -2  \\
	\end{array}
\] 
Note that if $k < j$ we should interchange the characters $\chi_3$ and $\chi_4$ in order to maintain the notation consistent. The characters of $G$ restricted to $H$ are the 1st, 7th, 10th, 4th, 4th, 10th, 7th, 5th columns of Table~\ref{CharTableDelta234}:
\[
	\begin{array}{c|cccccccc}
		 S_4 \times C_2 \downarrow & e & s_2 & s_3 & s_2s_3 & s_3s_2 & s_2s_3s_2& s_3s_2s_3& s_2s_3s_2s_3\\ 
		 \hline 
		 \rho_1 \otimes\xi_1 & 1 & 1 & 1 & 1 & 1 & 1 & 1 & 1\\
		 \rho_1 \otimes\xi_2 & 1 & -1 & 1 & -1 & -1 & 1 & -1 & 1 \\
		 \rho_1 \otimes\xi_3 & 2 & 0 & 2 & 0 & 0 & 2 & 0 & 2\\
		 \rho_1 \otimes\xi_4 & 3 & 1 & -1 & -1 & -1 & -1 & 1 & -1\\
		 \rho_1 \otimes\xi_5 & 3 & -1 & -1 & 1 & 1 & -1  & -1 & -1\\
		 \hline
		 \rho_2 \otimes \xi_1 & 1 & -1 & -1 & 1 & 1 & -1 & -1 & 1\\
		 \rho_2 \otimes\xi_2 & 1 & 1 & -1 & -1 & -1 & -1& 1 & 1\\
		 \rho_2 \otimes\xi_3 & 2 & 0 & -2 & 0 & 0 & -2 & 0 & 2\\
		 \rho_2 \otimes\xi_4 & 3 & -1 & 1 & -1 & -1 & 1 & -1 & -1\\
		 \rho_2 \otimes\xi_5 & 3 & 1 & 1 & 1 & 1 & 1 & 1 & -1
	\end{array}
\]
Multiplying these rows with the rows of the character table of $D_4$ above we deduce that
\[
	\begin{array}{rcll}
			& & \rho_1 \otimes \underline{\phantom{a}} & \rho_2 \otimes \underline{\phantom{a}}\\[1ex]
			\chi_1 \uparrow & = & \xi_1 + \xi_3 + 	& \xi_5 \\
			\chi_2 \uparrow & = & \xi_5 + 			& \xi_1 + \xi_3\\
			\chi_3 \uparrow & = & \xi_2 + \xi_3 + 	& \xi_4\\
			\chi_4\uparrow & = & \xi_4 + 			& \xi_2 + \xi_3\\
			\phi_1 \uparrow & = & \xi_4 + \xi_5 + 	& \xi_4 + \xi_5\\
	\end{array}
\]
The computation in the case $k<j$ is identical, but interchanging $\chi_3$ and $\chi_4$. All in all, we have the induction homomorphisms shown, as maps between abelian groups, in Figure X
\begin{figure}[h!]
\begin{eqnarray*}
	R_\mathbb{C}\left( H \right) &\to& R_\mathbb{C}\left( G \right)\\
	\mathbb{Z}^5 &\to & \mathbb{Z}^{10} \\
	(a,b,c,d,e) &\mapsto & (a,c,a+c,d+e,b+e, b, d, b+d, c+e, a+e)\ \ \text{if } j < k,\\
	(a,b,c,d,e) &\mapsto & (a,d,a+d,c+e,b+e, b, c, b+c, d+e, a+e)\ \ \text{if } k < j.\\
\end{eqnarray*}
\caption{Induction homomorphism from $H = \langle s_j, s_k \rangle \cong D_4$ to $G =\langle s_i, s_j, s_k \rangle \cong \Delta(2,3,4)=S_4 \times C_2$.}\label{InductionDelta234Case3}
\end{figure}

On the other hand, the induction homomorphisms with respect to the transformed bases (Tables~\ref{CharTableDm m odd}, \ref{CharTableDm m even}, \ref{BTCharTableDelta234} in Appendix~\ref{AppendixA}) are summarised in Tables~\ref{D2-S4xC2}, \ref{D4-S4xC2} and~\ref{D3-S4xC2}.

\begin{table}[!htb]
\small
\[
\begin{array}{c|cccc||cccc}
D_2 \hookrightarrow S_4 \times C_2 
		 & (1) & \alpha(12) & \alpha(12)(34) & (34)
		 & (\cdot | \sum \chi_i) & (\cdot | \chi_2+\widehat{\chi_3}) & (\cdot | \widehat{\chi_3}+\widehat{\chi_4}) & (\cdot | \widehat{\chi_4}) \\ 
		 \hline 
\alpha_1 \downarrow 	    	& 1 & 1 & 1 & 1 & 1 & 0 & 0 & 0 \\
\alpha_2 \downarrow	& 0 & -2& 0 & -2& 0 & 1 & 0 & 0 \\
\alpha_3 \downarrow 	& 0 & 0 & 0 & 0 & 0 & 0 & 0 & 0\\
\alpha_4 \downarrow	& 0 & 0 & -2&-2 & 0 & 0 & 1 & 1 \\
\alpha_5 \downarrow 	& 0 & 0 & 0 & 0 & 0 & 0 & 0 & 0\\
	\hline
\alpha_6 \downarrow	    	& 0 & 0 & -2&  2& 0 & 0 & 1 & 0\\
\alpha_7 \downarrow&0 & 0 & 0 & -4& 0 & 0 & 0 & 1 \\
\alpha_8 \downarrow 	    	& 0 & 0 & 0 & 0 & 0 & 0 & 0 & 0\\
\alpha_9 \downarrow     	& 0 & 0 & 0 & 0 & 0 & 0 & 0 & 0 \\
\alpha_{10} \downarrow 	    	& 0 & 0 & 0 & 0 & 0 & 0 & 0 & 0\\
\end{array}
\]
\medskip\caption{Restricted characters and map induced by the inclusion of $D_2$ into $S_4 \times C_2$. Here $\alpha_1,\ldots,\alpha_{10}$ are as in Table~\ref{BTCharTableDelta234}.}\label{D2-S4xC2}
\end{table}

\begin{table}[!htb]
\small
\noindent $
		\begin{array}{c|ccccc}
		 D_4 \hookrightarrow S_4 \times C_2& 
		 (1) & \alpha(13) & (13)(24) & \alpha(12)(34) & (1432)\\ 
		  \hline 
\alpha_1 \downarrow 	    & 1 & 1 & 1 & 1 & 1 \\
\alpha_2 \downarrow & 0 & -2& 0 & 0 & -2  \\
\alpha_3\downarrow & 0 & 0 & 0 & 0 & 0 \\
\alpha_4\downarrow  & 0 & 0 & -4&-2 & 4  \\
\alpha_5\downarrow & 0 & 0 & 0 & 0 & 4 \\
		 \hline
\alpha_6\downarrow	    & 0 & 0 & 0 &-2 & 2 \\
\alpha_7\downarrow  & 0 & 0 & 0 & 0 & 0  \\
\alpha_8\downarrow 	    & 0 & 0 & 0 & 0 & 0  \\
\alpha_9\downarrow	    & 0 & 0 & 0 & 0 & 0  \\
\alpha_{10}\downarrow 	    & 0 & 0 & 0 & 0 & 0  \\
\end{array}$	
\medskip
$\begin{array}{c|ccccc}
D_4 \hookrightarrow S_4 \times C_2 	 & (\cdot | 2\phi_1 +\sum \chi_i) & (\cdot | \chi_2+\widehat{\chi_3} +\phi_1) & (\cdot | \widehat{\chi_3}+\phi_1) & (\cdot | \widehat{\chi_4}+\chi_2) & (\cdot | \phi_1)\\ 
\hline 
\alpha_1 \downarrow  	    & 1 & 0 & 0 & 0 & 0 \\
\alpha_2 \downarrow  & 0 & 0 & 0 & 1 & 0  \\
\alpha_3 \downarrow  & 0 & 0 & 0 & 0 & 0 \\
\alpha_4 \downarrow   & 0 & 1 & 0 &-1 & 1  \\
\alpha_5 \downarrow  & 0 & 0 & -1& 0 & 0 \\
		 \hline
\alpha_6 \downarrow 	    & 0 & 1 & 0 & 0 & 0 \\
\alpha_7 \downarrow   & 0 & 0 & 0 & 0 & 0  \\
\alpha_8 \downarrow  	    & 0 & 0 & 0 & 0 & 0  \\
\alpha_9 \downarrow 	    & 0 & 0 & 0 & 0 & 0  \\
\alpha_{10} \downarrow  	    & 0 & 0 & 0 & 0 & 0  \\
	\end{array}
$
\medskip\caption{Restricted characters (top) and map induced by the inclusion of $D_4$ into $S_4 \times C_2$ (bottom). Here $\alpha_1,\ldots,\alpha_{10}$ are as in Table~\ref{BTCharTableDelta234}.}\label{D4-S4xC2}
\end{table} 

\begin{table}[!htb]
\[
\begin{array}{c|ccc||ccc}
D_3 \hookrightarrow S_4 \times C_2 &  (1) & \alpha(12) & (123) 
& (\cdot | 2\phi_1 +\sum \chi_i) & (\cdot | \chi_2+\phi_1) & (\cdot | \phi_1) \\ 
		 \hline 
\alpha_1 \downarrow            & 1 & 1 & 1 & 1 & 0 & 0 \\
\alpha_2 \downarrow  & 0 & -2& 0 & 0 & 1 & 0 \\
\alpha_3 \downarrow  & 0 & 0 & -3& 0 & 0 & 1\\
\alpha_4 \downarrow  & 0 & 0 & 0 & 0 & 0 & 0\\
\alpha_5 \downarrow  & 0 & 0 & 0 & 0 & 0 & 0 \\
		 \hline
\alpha_6 \downarrow 	    & 0 & 0 & 0 & 0 & 0 & 0 \\
\alpha_7 \downarrow  & 0 & 0 & 0 & 0 & 0 & 0 \\
\alpha_8 \downarrow  	    & 0 & 0 & 0 & 0 & 0 & 0 \\
\alpha_9 \downarrow 	    & 0 & 0 & 0 & 0 & 0 & 0 \\
\alpha_{10} \downarrow  	    & 0 & 0 & 0 & 0 & 0 & 0 \\
\end{array}
\]
\normalsize
\medskip\caption{Restricted characters and map induced by the inclusion of $D_3$ into $S_4 \times C_2$. Here $\alpha_1,\ldots,\alpha_{10}$ are as in Table~\ref{BTCharTableDelta234}.}\label{D3-S4xC2}
\end{table}

\subsubsection{$G = \Delta(2,3,5)$} This group is isomorphic to $A_5 \times C_2$ with Coxeter presentation
$$
	\Delta(2,3,5) = \left\langle s_i, s_j, s_k \; | \; s_i^2,s_j^2,s_k^2,(s_is_j)^3,(s_is_k)^2,(s_js_k)^5 \right\rangle .
$$
We have again three relevant induction homomorphisms. 

\medskip

\noindent\underline{Case 1: $H = \langle s_i, s_k \rangle \cong D_2 = C_2 \times C_2$}\\[1ex]
The characters of $G$ (Table~\ref{CharTableDelta235}) restricted to $H$ consists on the 1st, 8th, 8th, 3rd columns
\[
	\begin{array}{c|cccc}
		 A_5\times C_2 \downarrow & e & s_i & s_k & s_is_k\\ 
		 \hline 
		 \rho_1 \otimes\xi_1 & 1 & 1 & 1 & 1 \\
		 \rho_1 \otimes\xi_2 & 4 & 0 & 0 & 0 \\
		 \rho_1 \otimes\xi_3 & 5 & 1 & 1 & 1\\
		 \rho_1 \otimes\xi_4 & 3 & -1 & -1 & -1 \\
		 \rho_1 \otimes\xi_5 & 3 & -1 & -1 & -1 \\
		 \hline
		 \rho_2 \otimes \xi_1 & 1 & -1 & -1 & 1  \\
		 \rho_2 \otimes\xi_2 & 4 & 0 & 0 & 0\\
		 \rho_2 \otimes\xi_3 & 5 & -1 & -1 & 1\\
		 \rho_2 \otimes\xi_4 & 3 & 1 & 1 & -1 \\
		 \rho_2 \otimes\xi_5 & 3 & 1 & 1 & -1
	\end{array}
\]
Suppose first $i < k$.
Multiplying these rows with the rows of Table~\ref{CharTableC2xC2} we obtain
\[
	\begin{array}{rcll}
			& & \rho_1 \otimes \underline{\phantom{a}} & \rho_2 \otimes \underline{\phantom{a}}\\[1ex]
			\rho_1 \otimes \rho_1 \uparrow & = & 
				\xi_1 + \xi_2 + 2\xi_3 + & \xi_2 +\xi_3 + \xi_4 + \xi_5\\
			\rho_1 \otimes \rho_2\uparrow & = & 
				\xi_2 + \xi_3 + \xi_4 + \xi_5 + & \xi_2 + \xi_3 + \xi_4 + \xi_5\\
			\rho_2 \otimes \rho_1 \uparrow & = & 
				\xi_2 + \xi_3 + \xi_4 + \xi_5 + & \xi_2 + \xi_3 + \xi_4 + \xi_5\\
			\rho_2 \otimes \rho_2 \uparrow & = & 
				\xi_2 + \xi_3 + \xi_4 + \xi_5 + & \xi_1 + \xi_2 + 2\xi_3\\
	\end{array}
\]
If $k < i$ we must interchange the 2nd and 3rd generators, but note that we obtain the same map. All in one, we have one induction map, given as a homomorphism of free abelian groups in Figure~\ref{InductionDelta235Case1}.
\begin{figure}[h!]
\begin{eqnarray*}
	R_\mathbb{C}\left( H \right) &\to& R_\mathbb{C}\left( G \right)\\
	\mathbb{Z}^4 &\to & \mathbb{Z}^{10} \\
	(a,b,c,d) &\mapsto & (a,b+c+d,2a+b+c+d,b+c+d, b+c+d,\\ 
				&		& d, a+b+c, a+b+c+2d, a+b+c, a+b+c).
\end{eqnarray*}
\caption{Induction homomorphism from $H = \langle s_i, s_k \rangle \cong C_2 \times C_2$ to $G =\langle s_i, s_j, s_k \rangle \cong \Delta(2,3,5)=A_5 \times C_2$.}\label{InductionDelta235Case1}
\end{figure}

\medskip

\noindent\underline{Case 2: $H = \langle s_i, s_j \rangle \cong D_3$}\\[1ex]
The characters of $G$ (Table~\ref{CharTableDelta235}) restricted to $H$ consists on the 1st, 8th, 8th, 2nd, 2nd and 8th columns:
\[
	\begin{array}{c|cccccc}
		 A_5\times C_2 \downarrow & e & s_i & s_j & s_is_j& s_js_i& s_is_js_i\\ 
		 \hline 
		 \rho_1 \otimes\xi_1 & 1 & 1 & 1 & 1 & 1 & 1\\
		 \rho_1 \otimes\xi_2 & 4 & 0 & 0 & 1 & 1 & 0\\
		 \rho_1 \otimes\xi_3 & 5 & 1 & 1 & -1 & -1 & 1\\
		 \rho_1 \otimes\xi_4 & 3 & -1 & -1 & 0 & 0 & -1 \\
		 \rho_1 \otimes\xi_5 & 3 & -1 & -1 & 0 & 0 & -1 \\
		 \hline
		 \rho_2 \otimes\xi_1 & 1 & -1 & -1 & 1 & 1 & -1  \\
		 \rho_2 \otimes\xi_2 & 4 & 0 & 0 & 1 & 1 & 0\\
		 \rho_2 \otimes\xi_3 & 5 & -1 & -1 & -1 & -1 & -1\\
		 \rho_2 \otimes\xi_4 & 3 & 1 & 1 & 0 & 0 & 1\\
		 \rho_2 \otimes\xi_5 & 3 & 1 & 1 & 0 & 0 & 1
	\end{array}
\]
Multiplying these rows with the rows of the character table of $D_3$ (which is independent of whether $i<j$ or $j<i$)
\[
	\begin{array}{rcll}
			& & \rho_1 \otimes \underline{\phantom{a}} & \rho_2 \otimes \underline{\phantom{a}}\\[1ex]
			\chi_1 \uparrow & = & \xi_1 + \xi_2 + \xi_3 + 	& \xi_2 + \xi_4 + \xi_5 \\
			\chi_2 \uparrow & = & \xi_2 + \xi_4 + \xi_5 + 	& \xi_1 + \xi_2 + \xi_3 \\
			\phi_1 \uparrow & = & \xi_2 + 2\xi_3+ \xi_4 + \xi_5 + 	& \xi_2 + 2\xi_3+ \xi_4 + \xi_5 \\
	\end{array}
\]
This is then the map of free abelian groups shown in Figure~\ref{InductionDelta235Case2}.
\begin{figure}[h!]
\begin{eqnarray*}
	R_\mathbb{C}\left( H \right) &\to& R_\mathbb{C}\left( G \right)\\
	\mathbb{Z}^3 &\to & \mathbb{Z}^{10} \\
	(a,b,c) &\mapsto & (a,a+b+c,a+2c,b+c, b+c,\\ & & b, a+b+c, b+2c, a+c,a+c).
\end{eqnarray*}
\caption{Induction homomorphism from $H = \langle s_i, s_j \rangle \cong D_3$ to $G =\langle s_i, s_j, s_k \rangle \cong \Delta(2,3,5)=A_5 \times C_2$.}\label{InductionDelta235Case2}
\end{figure}

\begin{table}[!htb]
\scriptsize
\[
\begin{array}{c|cccc}
D_2 \hookrightarrow A_5 \times C_2  
& (1) 
& (12)(35)\alpha 
& (15)(23)\alpha 
& (13)(25) 
\\
\hline 
\beta_1 \downarrow & 1  & 1 & 1 & 1 \\
\beta_2 \downarrow & 0  & 0 & 0 & 0  \\
\beta_3 \downarrow & 0  & 0 & 0 & 0 \\
\beta_4 \downarrow & 0  & 0 & 0 & 0 \\
\beta_5 \downarrow & 0  & 0 & 0 &-4  \\
		 \hline
\beta_6 \downarrow & 0  &-2 &-2 & 0 \\
\beta_7 \downarrow & 0  & 0 & 0 & 0 \\
\beta_8 \downarrow & 0  & 0 & 0 & 0  \\
\beta_9 \downarrow & 0  & 0 & 0 & 0  \\
\beta_{10}\downarrow&0  & 0 & 0 & 0 \\
\end{array}
\]
\[
\begin{array}{c|cccc}
D_2 \hookrightarrow A_5 \times C_2  
& (\cdot | \sum \chi_i )
& (\cdot | \chi_2+ \chi_3 )
& (\cdot | \chi_3 +\chi_4) 
& (\cdot | \chi_4)\\
\hline 
\beta_1 \downarrow  & 1 & 0 & 0 & 0 \\
\beta_2 \downarrow  & 0 & 0 & 0 & 0 \\
\beta_3 \downarrow  & 0 & 0 & 0 & 0 \\
\beta_4 \downarrow  & 0 & 0 & 0 & 0 \\
\beta_5 \downarrow  & 0 & 0 & 0 & 1 \\
		 \hline
\beta_6 \downarrow & 0 & 1 & 1 & 0 \\
\beta_7 \downarrow & 0 & 0 & 0 & 0\\
\beta_8 \downarrow & 0 & 0 & 0 & 0 \\
\beta_9 \downarrow & 0 & 0 & 0 & 0 \\
\beta_{10}\downarrow& 0 & 0 & 0 & 0\\
\end{array}
\] \normalsize
\medskip\caption{Restricted characters (top) and map induced by the inclusion of $D_2$ into $A_5 \times C_2$ (bottom).}\label{D2toA5xC2}
\end{table} 

\medskip

\noindent\underline{Case 3: $H = \langle s_j, s_k \rangle \cong D_5$}\\[1ex]
First we expand the character table for $D_5$ (from Table~\ref{CharTableDm})	
\begin{displaymath}	
\small
\begin{array}{c|cccccccccc}
		 D_5 & e & s_j & s_k & s_js_k & s_ks_j & s_js_ks_j& s_ks_js_k& s_js_ks_js_k & s_ks_js_ks_j& s_js_ks_js_ks_j\\ 
		 \hline 
		 \chi_1 & 1 &  1 & 1 &  1& 1 &  1 & 1 &  1& 1 &  1\\
		 \chi_2 & 1 & -1 & -1 &  1 & 1 & -1 & -1 &  1 & 1 &  -1\\
		 \phi_1 & 2 &  0 &  0 &  \varphi-1 & \varphi-1 &  0 & 0 & \varphi & \varphi&  0 \\
		 \phi_2 & 2 &  0 &  0 &  \varphi &  \varphi &  0 & 0 & \varphi-1& \varphi-1 &  0 \\
	\end{array}
\end{displaymath}
where $\varphi$ is the golden ratio $\frac{1+\sqrt{5}}{2}$, and we have used $2\cos\left(\frac{2\pi}{5}\right) = \varphi-1$ and $2\cos\left(\frac{4\pi}{5}\right) = \varphi$. Note that this table is independent of interchanging $s_i$ with $s_j$. Next we restrict the characters of $G$ (Table~\ref{CharTableDelta235}) to $H$, that is, the 1st, 8th, 8th, 4th, 4th, 8th, 8th, 5th, 5th and 8th columns. 
\[
	\begin{array}{c|cccccccccc}
		 A_5\times C_2 \downarrow & e & s_j & s_k & s_js_k & s_ks_j & s_js_ks_j& s_ks_js_k& s_js_ks_js_k & s_ks_js_ks_j& s_js_ks_js_ks_j\\ 
		 \hline 
		 \rho_1 \otimes\xi_1 & 1 & 1 & 1 & 1 & 1 & 1 & 1 & 1 & 1 & 1\\
		 \rho_1 \otimes\xi_2 & 4 & 0 & 0 & -1 & -1 & 0 & 0 & -1 & - 1 & 0\\
		 \rho_1 \otimes\xi_3 & 5 & 1 & 1 & 0 & 0 & 1 & 1 & 0 & 0 & 1\\
		 \rho_1 \otimes\xi_4 & 3 & -1 & -1 & \varphi & \varphi & -1 & -1 & -\varphi+1 & -\varphi+1 & -1\\
		 \rho_1 \otimes\xi_5 & 3 & -1 & -1 & -\varphi+1 & -\varphi+1 & -1 & -1 & \varphi & \varphi & -1\\
		 \hline
		 \rho_2 \otimes\xi_1 & 1 & -1 & -1 & 1 & 1 & -1 & -1 & 1 & 1 & -1 \\
		 \rho_2 \otimes\xi_2 & 4 & 0 & 0 & -1 & -1 & 0 & 0 & -1 & -1 & 0\\
		 \rho_2 \otimes\xi_3 & 5 & -1 & -1 & 0 & 0 & -1 & -1 & 0 & 0 & -1\\
		 \rho_2 \otimes\xi_4 & 3 & 1 & 1 & \varphi & \varphi & 1 & 1 & -\varphi+1 & -\varphi+1 & 1\\
		 \rho_2 \otimes\xi_5 & 3 & 1 & 1 & -\varphi +1 & -\varphi+1 & 1 & 1 & \varphi & \varphi & 1
	\end{array}
\]
\begin{remark}
With a non-conjugated choice of Coxeter generators, it would be the 5th instead of the 4th column, or, equivalently, swapping $\rho_i \otimes \xi_4$ with $\rho_i \otimes \xi_5$ for $i=1$ and 2.
\end{remark}
Multiplying these rows with the rows of character table of $D_5$ above, and using that $\varphi^2-\varphi = 1$, we obtain
\[
	\begin{array}{rcll}
			& & \rho_1 \otimes \underline{\phantom{a}} & \rho_2 \otimes \underline{\phantom{a}}\\[1ex]
			\chi_1 \uparrow & = & 
				\xi_1 + \xi_3 + & \xi_4 + \xi_5\\
			\chi_2 \uparrow & = & 
				\xi_4 + \xi_5 + & \xi_1 + \xi_3 \\
			\phi_1 \uparrow & = & 
				\xi_2 + \xi_3 + \xi_4 + & \xi_2 + \xi_3 + \xi_4\\
			\phi_2 \uparrow & = & 
				\xi_2 + \xi_3 + \xi_5 + & \xi_2 + \xi_3 + \xi_5\\
	\end{array}
\]
This gives the homomorphism of free abelian groups in Figure~\ref{InductionDelta235Case3}.
\begin{figure}[h!]
\begin{eqnarray*}
	R_\mathbb{C}\left( H \right) &\to& R_\mathbb{C}\left( G \right)\\
	\mathbb{Z}^4 &\to & \mathbb{Z}^{10} \\
	(a,b,c,d) &\mapsto & (a,c+d,a+c+d,b+c, b+d,\\ & & b, c+d, b+c+d, a+c,a+d).
\end{eqnarray*}
\caption{Induction homomorphism from $H = \langle s_j, s_k \rangle \cong D_5$ to $G =\langle s_i, s_j, s_k \rangle \cong \Delta(2,3,5)=A_5 \times C_2$.}\label{InductionDelta235Case3}
\end{figure}

We finish by giving the same induction homomorphisms but this time with respect to the transformed bases  (Tables~\ref{CharTableD2}, \ref{CharTableDm m odd}, \ref{BTCharTableDelta235} in Appendix~\ref{AppendixA}) in Tables~\ref{D2toA5xC2}, \ref{D3toA5xC2} and~\ref{D5toA5xC2}.

\begin{table}[!htb]
\small
\[ 
	\begin{array}{c|ccc||ccc}
D_3 \hookrightarrow A_5 \times C_2  
& (1) 
& (12)(34)\alpha 
& (123) 
& (\cdot | \sum \chi_i +2 \phi_1 )
& (\cdot | \chi_2+ \phi_1 )
& (\cdot | \phi_1 )
\\
		 \hline 
\beta_1 \downarrow & 1 & 1 & 1 & 1 & 0 & 0 \\
\beta_2 \downarrow & 0 & 0 & 0 & 0 & 0 & 0\\
\beta_3 \downarrow & 0 & 0 & -3& 0 & 0 & 1 \\
\beta_4 \downarrow & 0 & 0 & 0 & 0 & 0 & 0 \\
\beta_5 \downarrow & 0 & 0 & 0 & 0 & 0 & 0\\
		 \hline
\beta_6 \downarrow & 0 &-2 & 0 & 0 & 1 & 0\\
\beta_7 \downarrow & 0 & 0 & 0 & 0 & 0 & 0\\
\beta_8 \downarrow & 0 & 0 & 0 & 0 & 0 & 0\\
\beta_9 \downarrow & 0 & 0 & 0 & 0 & 0 & 0 \\
\beta_{10}\downarrow&0 & 0 & 0 & 0 & 0 & 0\\

	\end{array}
\]
\normalsize
\medskip\caption{Restricted characters and map induced by the inclusion of $D_3$ into $A_5 \times C_2$.}\label{D3toA5xC2}
\end{table} 

\begin{table}[!htb]
\footnotesize
\[ 
	\begin{array}{c|cccc||cccc}
D_5 \hookrightarrow A_5 \times C_2  
& (1) 
& (12345) 
& (12354) 
& (12)(34)\alpha 
& (\cdot | \sum \chi_i +2\sum \phi_i )
& (\cdot | \chi_2+ \sum\phi_i )
& (\cdot | \sum\phi_i )
& (\cdot | \phi_2 )
\\
		 \hline 
\beta_1 \downarrow & 1  & 1                    & 1                     & 1 & 1 & 0 & 0 & 0\\
\beta_2 \downarrow & 0  & \sqrt{5}             & -\sqrt{5}             & 0 & 0 & 0 & 0 & 1\\
\beta_3 \downarrow & 0  & 0                    & 0                     & 0 & 0 & 0 & 0 & 0\\
\beta_4 \downarrow & 0  & \frac{5+\sqrt{5}}{2} & \frac{5-\sqrt{5}}{2}  & 0 & 0 & 0 & -1& 0\\
\beta_5 \downarrow & 0  & 0                    & 0                     & 0 & 0 & 0 & 0 & 0\\
		 \hline
\beta_6 \downarrow & 0  & 0                    & 0                     &-2 & 0 & 1 & 0 & 0\\
\beta_7 \downarrow & 0  & 0                    & 0                     & 0 & 0 & 0 & 0 & 0\\
\beta_8 \downarrow & 0  & 0                    & 0                     & 0 & 0 & 0 & 0 & 0\\
\beta_9 \downarrow & 0  & 0                    & 0                     & 0 & 0 & 0 & 0 & 0\\
\beta_{10}\downarrow&0  & 0                    & 0                     & 0 & 0 & 0 & 0 & 0\\

	\end{array}
\]
\normalsize
\medskip\caption{Restricted characters and map induced by the inclusion of $D_5$ into $A_5 \times C_2$.}\label{D5toA5xC2}
\end{table}

\end{document}

%% file: face_decomposition.tikz
\definecolor{uuuuuu}{rgb}{0.26666666666666666,0.26666666666666666,0.26666666666666666}\definecolor{ttqqqq}{rgb}{0.2,0,0}\begin{tikzpicture}[line cap=round,line join=round,>=triangle 45,x=1cm,y=1cm,scale=0.3]\clip(3.070090704349176,-3.96898471443659) rectangle (23.077967990203717,14.634366086450253);\fill[line width=1.2pt,fill=black,fill opacity=0.1] (6.8444839764998076,4.482133537227108) -- (8.540164171934249,5.207693943663967) -- (8.759650525696987,7.038976272822785) -- (7.283456684025284,8.144698195544743) -- (5.587776488590842,7.419137789107884) -- (5.368290134828104,5.5878554599490675) -- cycle;\fill[line width=1.2pt,fill=black,fill opacity=0.1] (6.63,2.6154028259781286) -- (8.12187,1.50606) -- (9.82852406880308,2.243385906154826) -- (10.04330813760616,4.090054638287781) -- (8.55143813760616,5.19939746426591) -- (6.844784068803079,4.462071558111085) -- cycle;\fill[line width=1.2pt,fill=black,fill opacity=0.05] (5.598173746379701,-0.4455943008213786) -- (7.598173746379704,-0.4455943008213786) -- (8.216207735129599,1.4565187317689308) -- (6.598173746379702,2.6320892363538784) -- (4.980139757629805,1.456518731768931) -- cycle;\fill[line width=1.2pt,fill=black,fill opacity=0.05] (10.050247911842865,4.127366853520441) -- (9.85513422827605,2.2738815469191453) -- (11.362742747550556,1.1781654870236762) -- (13.065464950391876,1.935934733729505) -- (13.26057863395869,3.7894200403308016) -- (11.752970114684185,4.885136100226271) -- cycle;\fill[line width=1.2pt,fill=black,fill opacity=0.05] (13.050946924361684,1.935934733729505) -- (11.395194493852529,1.172318879623702) -- (11.228629006986125,-0.6434127146279552) -- (12.71781595062888,-1.695528454773812) -- (14.373568381138035,-0.9319126006680107) -- (14.540133868004439,0.8838189935836472) -- cycle;\fill[line width=1.2pt,fill=black,fill opacity=0.1] (17.895818784401822,5.170206259832203) -- (16.66696151827849,3.792498395918885) -- (18.044669382191802,2.5636411297955535) -- (19.273526648315137,3.941348993708868) -- cycle;\fill[line width=1.2pt,fill=black,fill opacity=0.05] (12.70437094245782,-1.667954364772842) -- (11.241667183450124,-0.6360378181410664) -- (9.808255908233686,-1.7082720098688238) -- (10.385062779300293,-3.4028657308881254) -- (12.174960305780372,-3.3779480558724533) -- cycle;\fill[line width=1.2pt,fill=black,fill opacity=0.05] (11.777479270649964,4.876648796397846) -- (13.275415522675573,3.795298365941491) -- (14.960860591856818,4.55187399821714) -- (15.148369409012455,6.389800060949143) -- (13.650433156986848,7.471150491405497) -- (11.964988087805603,6.714574859129851) -- cycle;\fill[line width=1.2pt,fill=black,fill opacity=0.05] (15.152216036590277,6.396871128250345) -- (14.964781801168915,4.54427519204534) -- (16.668784874285734,3.7935310130737405) -- (17.90935092582756,5.182141529818166) -- (16.972059837852846,6.791094205273374) -- cycle;\fill[line width=1.2pt,fill=black,fill opacity=0.05] (11.9680222329796,6.718436663804776) -- (13.64725180938035,7.462189631895138) -- (13.458811179129027,9.289064169643352) -- (11.663118888371507,9.67438175906314) -- (10.741760649598524,8.08564658803953) -- cycle;\fill[line width=1.2pt,fill=black,fill opacity=0.05] (14.139482089391567,12.400521393314591) -- (14.704568310544015,10.653215314751126) -- (16.540976388147158,10.650696974996007) -- (17.110852776168304,12.39644663399559) -- (15.626647675748258,13.47789759886099) -- cycle;\fill[line width=1.2pt,fill=black,fill opacity=0.05] (16.96413900558322,6.803713622705737) -- (17.9,5.2) -- (19.714419242793017,5.594481933559855) -- (19.899931010264172,7.4419987991533585) -- (18.200164345081397,8.18934508331896) -- cycle;\fill[line width=1.2pt,fill=black,fill opacity=0.1] (11.664272658025446,9.680500988194913) -- (13.46271487127929,9.29430918661199) -- (14.696387888810303,10.65870992973667) -- (14.131618693087471,12.409302474444273) -- (12.333176479833627,12.795494276027195) -- (11.099503462302614,11.431093532902516) -- cycle;\draw [line width=1.2pt] (6.8444839764998076,4.482133537227108)-- (8.540164171934249,5.207693943663967);\draw [line width=1.2pt] (8.540164171934249,5.207693943663967)-- (8.759650525696987,7.038976272822785);\draw [line width=1.2pt] (8.759650525696987,7.038976272822785)-- (7.283456684025284,8.144698195544743);\draw [line width=1.2pt] (7.283456684025284,8.144698195544743)-- (5.587776488590842,7.419137789107884);\draw [line width=1.2pt] (5.587776488590842,7.419137789107884)-- (5.368290134828104,5.5878554599490675);\draw [line width=1.2pt] (5.368290134828104,5.5878554599490675)-- (6.8444839764998076,4.482133537227108);\draw [line width=1.2pt] (6.63,2.6154028259781286)-- (8.12187,1.50606);\draw [line width=1.2pt] (8.12187,1.50606)-- (9.82852406880308,2.243385906154826);\draw [line width=1.2pt] (9.82852406880308,2.243385906154826)-- (10.04330813760616,4.090054638287781);\draw [line width=1.2pt] (10.04330813760616,4.090054638287781)-- (8.55143813760616,5.19939746426591);\draw [line width=1.2pt] (8.55143813760616,5.19939746426591)-- (6.844784068803079,4.462071558111085);\draw [line width=1.2pt] (6.844784068803079,4.462071558111085)-- (6.63,2.6154028259781286);\draw [line width=1.2pt] (5.598173746379701,-0.4455943008213786)-- (7.598173746379704,-0.4455943008213786);\draw [line width=1.2pt] (7.598173746379704,-0.4455943008213786)-- (8.216207735129599,1.4565187317689308);\draw [line width=1.2pt] (8.216207735129599,1.4565187317689308)-- (6.598173746379702,2.6320892363538784);\draw [line width=1.2pt] (6.598173746379702,2.6320892363538784)-- (4.980139757629805,1.456518731768931);\draw [line width=1.2pt] (4.980139757629805,1.456518731768931)-- (5.598173746379701,-0.4455943008213786);\draw [line width=1.2pt] (10.050247911842865,4.127366853520441)-- (9.85513422827605,2.2738815469191453);\draw [line width=1.2pt] (9.85513422827605,2.2738815469191453)-- (11.362742747550556,1.1781654870236762);\draw [line width=1.2pt] (11.362742747550556,1.1781654870236762)-- (13.065464950391876,1.935934733729505);\draw [line width=1.2pt] (13.065464950391876,1.935934733729505)-- (13.26057863395869,3.7894200403308016);\draw [line width=1.2pt] (13.26057863395869,3.7894200403308016)-- (11.752970114684185,4.885136100226271);\draw [line width=1.2pt] (11.752970114684185,4.885136100226271)-- (10.050247911842865,4.127366853520441);\draw [line width=1.2pt] (13.050946924361684,1.935934733729505)-- (11.395194493852529,1.172318879623702);\draw [line width=1.2pt] (11.395194493852529,1.172318879623702)-- (11.228629006986125,-0.6434127146279552);\draw [line width=1.2pt] (11.228629006986125,-0.6434127146279552)-- (12.71781595062888,-1.695528454773812);\draw [line width=1.2pt] (12.71781595062888,-1.695528454773812)-- (14.373568381138035,-0.9319126006680107);\draw [line width=1.2pt] (14.373568381138035,-0.9319126006680107)-- (14.540133868004439,0.8838189935836472);\draw [line width=1.2pt] (14.540133868004439,0.8838189935836472)-- (13.050946924361684,1.935934733729505);\draw [line width=1.2pt] (17.895818784401822,5.170206259832203)-- (16.66696151827849,3.792498395918885);\draw [line width=1.2pt] (16.66696151827849,3.792498395918885)-- (18.044669382191802,2.5636411297955535);\draw [line width=1.2pt] (18.044669382191802,2.5636411297955535)-- (19.273526648315137,3.941348993708868);\draw [line width=1.2pt] (19.273526648315137,3.941348993708868)-- (17.895818784401822,5.170206259832203);\draw [line width=1.2pt] (12.70437094245782,-1.667954364772842)-- (11.241667183450124,-0.6360378181410664);\draw [line width=1.2pt] (11.241667183450124,-0.6360378181410664)-- (9.808255908233686,-1.7082720098688238);\draw [line width=1.2pt] (9.808255908233686,-1.7082720098688238)-- (10.385062779300293,-3.4028657308881254);\draw [line width=1.2pt] (10.385062779300293,-3.4028657308881254)-- (12.174960305780372,-3.3779480558724533);\draw [line width=1.2pt] (12.174960305780372,-3.3779480558724533)-- (12.70437094245782,-1.667954364772842);\draw [line width=1.2pt] (11.777479270649964,4.876648796397846)-- (13.275415522675573,3.795298365941491);\draw [line width=1.2pt] (13.275415522675573,3.795298365941491)-- (14.960860591856818,4.55187399821714);\draw [line width=1.2pt] (14.960860591856818,4.55187399821714)-- (15.148369409012455,6.389800060949143);\draw [line width=1.2pt] (15.148369409012455,6.389800060949143)-- (13.650433156986848,7.471150491405497);\draw [line width=1.2pt] (13.650433156986848,7.471150491405497)-- (11.964988087805603,6.714574859129851);\draw [line width=1.2pt] (11.964988087805603,6.714574859129851)-- (11.777479270649964,4.876648796397846);\draw [line width=1.2pt] (15.152216036590277,6.396871128250345)-- (14.964781801168915,4.54427519204534);\draw [line width=1.2pt] (14.964781801168915,4.54427519204534)-- (16.668784874285734,3.7935310130737405);\draw [line width=1.2pt] (16.668784874285734,3.7935310130737405)-- (17.90935092582756,5.182141529818166);\draw [line width=1.2pt] (17.90935092582756,5.182141529818166)-- (16.972059837852846,6.791094205273374);\draw [line width=1.2pt] (16.972059837852846,6.791094205273374)-- (15.152216036590277,6.396871128250345);\draw [line width=1.2pt] (11.9680222329796,6.718436663804776)-- (13.64725180938035,7.462189631895138);\draw [line width=1.2pt] (13.64725180938035,7.462189631895138)-- (13.458811179129027,9.289064169643352);\draw [line width=1.2pt] (13.458811179129027,9.289064169643352)-- (11.663118888371507,9.67438175906314);\draw [line width=1.2pt] (11.663118888371507,9.67438175906314)-- (10.741760649598524,8.08564658803953);\draw [line width=1.2pt] (10.741760649598524,8.08564658803953)-- (11.9680222329796,6.718436663804776);\draw [line width=1.2pt] (14.139482089391567,12.400521393314591)-- (14.704568310544015,10.653215314751126);\draw [line width=1.2pt] (14.704568310544015,10.653215314751126)-- (16.540976388147158,10.650696974996007);\draw [line width=1.2pt] (16.540976388147158,10.650696974996007)-- (17.110852776168304,12.39644663399559);\draw [line width=1.2pt] (17.110852776168304,12.39644663399559)-- (15.626647675748258,13.47789759886099);\draw [line width=1.2pt] (15.626647675748258,13.47789759886099)-- (14.139482089391567,12.400521393314591);\draw [line width=1.2pt] (16.96413900558322,6.803713622705737)-- (17.9,5.2);\draw [line width=1.2pt] (17.9,5.2)-- (19.714419242793017,5.594481933559855);\draw [line width=1.2pt] (19.714419242793017,5.594481933559855)-- (19.899931010264172,7.4419987991533585);\draw [line width=1.2pt] (19.899931010264172,7.4419987991533585)-- (18.200164345081397,8.18934508331896);\draw [line width=1.2pt] (18.200164345081397,8.18934508331896)-- (16.96413900558322,6.803713622705737);\draw [line width=1.2pt] (11.664272658025446,9.680500988194913)-- (13.46271487127929,9.29430918661199);\draw [line width=1.2pt] (13.46271487127929,9.29430918661199)-- (14.696387888810303,10.65870992973667);\draw [line width=1.2pt] (14.696387888810303,10.65870992973667)-- (14.131618693087471,12.409302474444273);\draw [line width=1.2pt] (14.131618693087471,12.409302474444273)-- (12.333176479833627,12.795494276027195);\draw [line width=1.2pt] (12.333176479833627,12.795494276027195)-- (11.099503462302614,11.431093532902516);\draw [line width=1.2pt] (11.099503462302614,11.431093532902516)-- (11.664272658025446,9.680500988194913);\begin{footnotesize}\draw [fill=black] (6.8444839764998076,4.482133537227108) circle (0.5pt);\draw [fill=black] (8.540164171934249,5.207693943663967) circle (0.5pt);\draw[color=black] (6.939163662938267,6.53846342310135) node {$F_1$};\draw [color=ttqqqq] (8.759650525696987,7.038976272822785)-- ++(-0.5pt,-0.5pt) -- ++(1pt,1pt) ++(-1pt,0) -- ++(1pt,-1pt);\draw [fill=uuuuuu] (7.283456684025284,8.144698195544743) circle (0.5pt);\draw [fill=uuuuuu] (5.587776488590842,7.419137789107884) circle (0.5pt);\draw [fill=uuuuuu] (5.368290134828104,5.5878554599490675) circle (0.5pt);\draw [fill=black] (6.63,2.6154028259781286) circle (0.5pt);\draw [fill=black] (8.12187,1.50606) circle (0.5pt);\draw[color=black] (8.237688149040084,3.649908954017779) node {$F_2$};\draw [color=ttqqqq] (9.82852406880308,2.243385906154826)-- ++(-0.5pt,-0.5pt) -- ++(1pt,1pt) ++(-1pt,0) -- ++(1pt,-1pt);\draw [fill=uuuuuu] (10.04330813760616,4.090054638287781) circle (0.5pt);\draw [fill=uuuuuu] (8.55143813760616,5.19939746426591) circle (0.5pt);\draw [fill=uuuuuu] (6.844784068803079,4.462071558111085) circle (0.5pt);\draw [fill=black] (5.598173746379701,-0.4455943008213786) circle (0.5pt);\draw [fill=black] (7.598173746379704,-0.4455943008213786) circle (0.5pt);\draw[color=black] (6.3296521694619035,1.211862980112381) node {$F_3$};\draw [fill=black] (8.216207735129599,1.4565187317689308) circle (0.5pt);\draw [fill=black] (6.598173746379702,2.6320892363538784) circle (0.5pt);\draw [fill=black] (4.980139757629805,1.456518731768931) circle (0.5pt);\draw [fill=black] (10.050247911842865,4.127366853520441) circle (0.5pt);\draw [fill=black] (9.85513422827605,2.2738815469191453) circle (0.5pt);\draw[color=black] (11.444248614720083,3.2789019579886967) node {$F_4$};\draw [fill=black] (11.362742747550556,1.1781654870236762) circle (0.5pt);\draw [fill=black] (13.065464950391876,1.935934733729505) circle (0.5pt);\draw [fill=black] (13.26057863395869,3.7894200403308016) circle (0.5pt);\draw [fill=black] (11.752970114684185,4.885136100226271) circle (0.5pt);\draw [fill=black] (13.050946924361684,1.935934733729505) circle (0.5pt);\draw [fill=black] (11.395194493852529,1.172318879623702) circle (0.5pt);
\draw[color=black] (12.79577410025463,0.24334848833785323) node {$F_5$};
\draw [fill=black] (11.228629006986125,-0.6434127146279552) circle (0.5pt);\draw [fill=black] (12.71781595062888,-1.695528454773812) circle (0.5pt);\draw [fill=black] (14.373568381138035,-0.9319126006680107) circle (0.5pt);\draw [fill=black] (14.540133868004439,0.8838189935836472) circle (0.5pt);\draw [fill=black] (17.895818784401822,5.170206259832203) circle (0.5pt);
\draw [fill=black] (16.66696151827849,3.792498395918885) circle (0.5pt);
\draw[color=black] (18.00887504296008,4.059420447494129) node {$F_9$};
\draw [fill=uuuuuu] (18.044669382191802,2.5636411297955535) circle (0.5pt);\draw [fill=uuuuuu] (19.273526648315137,3.941348993708868) circle (0.5pt);\draw [fill=black] (12.70437094245782,-1.667954364772842) circle (0.5pt);\draw [fill=black] (11.241667183450124,-0.6360378181410664) circle (0.5pt);\draw[color=black] (11.205744117272811,-1.941696486134819) node {$F_6$};\draw [fill=black] (9.808255908233686,-1.7082720098688238) circle (0.5pt);\draw [fill=black] (10.385062779300293,-3.4028657308881254) circle (0.5pt);\draw [fill=black] (12.174960305780372,-3.3779480558724533) circle (0.5pt);\draw [fill=black] (11.777479270649964,4.876648796397846) circle (0.5pt);\draw [fill=black] (13.275415522675573,3.795298365941491) circle (0.5pt);
\draw[color=black] (13.431786093447355,5.655452429341363) node {$F_{7}$};
\draw [fill=black] (14.960860591856818,4.55187399821714) circle (0.5pt);\draw [fill=black] (15.148369409012455,6.389800060949143) circle (0.5pt);\draw [fill=black] (13.650433156986848,7.471150491405497) circle (0.5pt);\draw [fill=black] (11.964988087805603,6.714574859129851) circle (0.5pt);\draw [fill=black] (15.152216036590277,6.396871128250345) circle (0.5pt);\draw [fill=black] (14.964781801168915,4.54427519204534) circle (0.5pt);\draw[color=black] (16.293840062814624,5.6904474321777325) node {$F_8$};\draw [fill=black] (16.668784874285734,3.7935310130737405) circle (0.5pt);\draw [fill=black] (17.90935092582756,5.182141529818166) circle (0.5pt);\draw [fill=black] (16.972059837852846,6.791094205273374) circle (0.5pt);\draw [fill=black] (11.9680222329796,6.718436663804776) circle (0.5pt);\draw [fill=black] (13.64725180938035,7.462189631895138) circle (0.5pt);\draw[color=black] (12.265764105927355,8.552501401544939) node {$F_{11}$};\draw [fill=black] (13.458811179129027,9.289064169643352) circle (0.5pt);\draw [fill=black] (11.663118888371507,9.67438175906314) circle (0.5pt);\draw [fill=black] (10.741760649598524,8.08564658803953) circle (0.5pt);\draw [fill=black] (11.670204249892139,9.672576095782206) circle (0.5pt);\draw [fill=black] (13.471789282932798,9.29588487622065) circle (0.5pt);\draw [fill=black] (14.139482089391567,12.400521393314591) circle (0.5pt);\draw [fill=black] (14.704568310544015,10.653215314751126) circle (0.5pt);
\draw[color=black] (15.472324571607354,12.154076858715757) node {$F_{13}$};
\draw [fill=black] (16.540976388147158,10.250696974996007) circle (0.5pt);\draw [fill=black] (17.110852776168304,12.39644663399559) circle (0.5pt);\draw [fill=black] (15.626647675748258,13.47789759886099) circle (0.5pt);\draw [fill=black] (16.96413900558322,6.803713622705737) circle (0.5pt);\draw [fill=black] (17.9,5.2) circle (0.5pt);\draw[color=black] (18.519882038989174,6.935970918846794) node {$F_{10}$};\draw [fill=black] (19.714419242793017,5.594481933559855) circle (0.5pt);\draw [fill=black] (19.899931010264172,7.4419987991533585) circle (0.5pt);\draw [fill=black] (18.200164345081397,8.18934508331896) circle (0.5pt);\draw [fill=black] (11.664272658025446,9.680500988194913) circle (0.5pt);\draw [fill=black] (13.46271487127929,9.29430918661199) circle (0.5pt);\draw[color=black] (12.716272601105537,11.335053871763057) node {$F_{12}$};\draw [color=ttqqqq] (14.696387888810303,10.65870992973667)-- ++(-0.5pt,-0.5pt) -- ++(1pt,1pt) ++(-1pt,0) -- ++(1pt,-1pt);\draw [fill=uuuuuu] (14.131618693087471,12.409302474444273) circle (0.5pt);\draw [fill=uuuuuu] (12.333176479833627,12.795494276027195) circle (0.5pt);\draw [fill=uuuuuu] (11.099503462302614,11.431093532902516) circle (0.5pt);\end{footnotesize}\end{tikzpicture}

%% file: tree_decomposition.tikz
\definecolor{qqqqff}{rgb}{0,0,0}\begin{tikzpicture}[line cap=round,line
join=round,>=triangle
45,x=1cm,y=1cm,scale=0.65]\clip(-1.8445174201700036,-1.451886781374909)
rectangle (7.758566169377614,4.474425630678594);\draw [line width=1.2pt]
(-1,3)-- (0,2);\draw [line width=1.2pt] (0,2)-- (-1,1);\draw [line
width=1.2pt] (0,2)-- (1,1);\draw [line width=1.2pt] (1,1)-- (2,0);\draw
[line width=1.2pt] (2,0)-- (1,-1);\draw [line width=1.2pt] (1,1)--
(3,2);\draw [line width=1.2pt] (3,2)-- (4,1);\draw [line width=1.2pt]
(3,2)-- (3.48,3.48);\draw [line width=1.2pt] (3.48,3.48)-- (5,4);\draw
[line width=1.2pt] (5,4)-- (6,3);\draw [line width=1.2pt] (4,1)--
(6,0);\draw [line width=1.2pt] (4,1)-- (4,-1);\begin{footnotesize}\draw
[fill=qqqqff] (-1,3) circle (2.5pt);\draw[color=qqqqff]
(-0.9276875892781543,3.20315127753474) node {$v_1$};\draw [fill=qqqqff]
(0,2) circle (2.5pt);\draw[color=qqqqff]
(0.07420892695417591,2.2012547613024096) node {$v_2$};\draw
[fill=qqqqff] (-1,1) circle (2.5pt);\draw[color=qqqqff]
(-1.076875892781543,1.299358245070079) node {$v_3$};\draw [fill=qqqqff]
(1,1) circle (2.5pt);\draw[color=qqqqff]
(1.076105443186506,1.299358245070079) node {$v_4$};\draw [fill=qqqqff]
(2,0) circle (2.5pt);\draw[color=qqqqff]
(2.0780019594188364,0.2069135827644685) node {$v_5$};\draw [fill=qqqqff]
(1,-1) circle (2.5pt);\draw[color=qqqqff]
(1.006105443186506,-0.6949829334678621) node {$v_6$};\draw [fill=qqqqff]
(3,2) circle (2.5pt);\draw[color=qqqqff]
(2.7798984756511665,2.2012547613024096) node {$v_7$};\draw [fill=qqqqff]
(4,1) circle (2.5pt);\draw[color=qqqqff]
(4.072343137956777,1.199358245070079) node {$v_8$};\draw [fill=qqqqff]
(3.48,3.48) circle (2.5pt);

\draw[color=qqqqff] (3.6,3.8) node {$v_{11}$};

\draw[fill=qqqqff] (5,4) circle (2.5pt);\draw[color=qqqqff]
(5.074239654189107,4.255047793767071) node {$v_{12}$};\draw
[fill=qqqqff] (6,3) circle (2.5pt);\draw[color=qqqqff]
(6.076136170421437,3.30315127753474) node {$v_{13}$};\draw [fill=qqqqff]
(6,0) circle (2.5pt);\draw[color=qqqqff]
(6.076136170421437,0.2569135827644685) node {$v_{10}$};\draw
[fill=qqqqff] (4,-1) circle (2.5pt);
\draw[color=qqqqff]
(3.7072343137956777,-0.7949829334678621) node {$v_9$};
\end{footnotesize}\end{tikzpicture}

%% file: noncontractible.tikz
\definecolor{qqwuqq}{rgb}{0,0.39215686274509803,0}\definecolor{ccqqqq}{rgb}{0.8,0,0}\definecolor{uuuuuu}{rgb}{0.26666666666666666,0.26666666666666666,0.26666666666666666}\definecolor{ttqqqq}{rgb}{0.2,0,0}\begin{tikzpicture}[line cap=round,line join=round,>=triangle 45,x=1cm,y=1cm, scale=0.32]\clip(4.004439287364117,-3.539582000176419) rectangle (38.23448883528089,13.670293681499523);

\fill[line width=1.2pt,fill=black,fill opacity=0.1] (6.8444839764998076,4.482133537227108) -- (8.540164171934249,5.207693943663967) -- (8.759650525696987,7.038976272822785) -- (7.283456684025284,8.144698195544743) -- (5.587776488590842,7.419137789107884) -- (5.368290134828104,5.5878554599490675) -- cycle;

\fill[line width=1.2pt,fill=black,fill opacity=0.1] (6.63,2.6154028259781286) -- (8.12187,1.50606) -- (9.82852406880308,2.243385906154826) -- (10.04330813760616,4.090054638287781) -- (8.55143813760616,5.19939746426591) -- (6.844784068803079,4.462071558111085) -- cycle;

\fill[line width=1.2pt,fill=black,fill opacity=0.05] (5.598173746379701,-0.4455943008213786) -- (7.598173746379704,-0.4455943008213786) -- (8.216207735129599,1.4565187317689308) -- (6.598173746379702,2.6320892363538784) -- (4.980139757629805,1.456518731768931) -- cycle;

\fill[line width=1.2pt,fill=black,fill opacity=0.05] (10.050247911842865,4.127366853520441) -- (9.85513422827605,2.2738815469191453) -- (11.362742747550556,1.1781654870236762) -- (13.065464950391876,1.935934733729505) -- (13.26057863395869,3.7894200403308016) -- (11.752970114684185,4.885136100226271) -- cycle; 

\fill[line width=1.2pt,fill=black,fill opacity=0.05] (13.050946924361684,1.935934733729505) -- (11.395194493852529,1.172318879623702) -- (11.228629006986125,-0.6434127146279552) -- (12.71781595062888,-1.695528454773812) -- (14.373568381138035,-0.9319126006680107) -- (14.540133868004439,0.8838189935836472) -- cycle;

\fill[line width=1.2pt,fill=black,fill opacity=0.1] (17.895818784401822,5.170206259832203) -- (16.66696151827849,3.792498395918885) -- (18.044669382191802,2.5636411297955535) -- (19.273526648315137,3.941348993708868) -- cycle;

\fill[line width=1.2pt,fill=black,fill opacity=0.05] (12.70437094245782,-1.667954364772842) -- (11.241667183450124,-0.6360378181410664) -- (9.808255908233686,-1.7082720098688238) -- (10.385062779300293,-3.4028657308881254) -- (12.174960305780372,-3.3779480558724533) -- cycle;

\fill[line width=1.2pt,fill=black,fill opacity=0.05] (11.777479270649964,4.876648796397846) -- (13.275415522675573,3.795298365941491) -- (14.960860591856818,4.55187399821714) -- (15.148369409012455,6.389800060949143) -- (13.650433156986848,7.471150491405497) -- (11.964988087805603,6.714574859129851) -- cycle;

\fill[line width=1.2pt,fill=black,fill opacity=0.05] (15.152216036590277,6.396871128250345) -- (14.964781801168915,4.54427519204534) -- (16.668784874285734,3.7935310130737405) -- (17.90935092582756,5.182141529818166) -- (16.972059837852846,6.791094205273374) -- cycle;\fill[line width=1.2pt,fill=black,fill opacity=0.05] (11.9680222329796,6.718436663804776) -- (13.64725180938035,7.462189631895138) -- (13.458811179129027,9.289064169643352) -- (11.663118888371507,9.67438175906314) -- (10.741760649598524,8.08564658803953) -- cycle;\fill[line width=1.2pt,fill=black,fill opacity=0.05] (14.139482089391567,12.400521393314591) -- (14.704568310544015,10.653215314751126) -- (16.540976388147158,10.650696974996007) -- (17.110852776168304,12.39644663399559) -- (15.626647675748258,13.47789759886099) -- cycle;\fill[line width=1.2pt,fill=black,fill opacity=0.05] (16.96413900558322,6.803713622705737) -- (17.9,5.2) -- (19.714419242793017,5.594481933559855) -- (19.899931010264172,7.4419987991533585) -- (18.200164345081397,8.18934508331896) -- cycle;\fill[line width=1.2pt,fill=black,fill opacity=0.1] (11.664272658025446,9.680500988194913) -- (13.46271487127929,9.29430918661199) -- (14.696387888810303,10.65870992973667) -- (14.131618693087471,12.409302474444273) -- (12.333176479833627,12.795494276027195) -- (11.099503462302614,11.431093532902516) -- cycle;\fill[line width=1.2pt,fill=black,fill opacity=0.05] (16.207619919315206,1.6436340866466264) -- (14.551867488806051,0.8800182325408226) -- (14.385302001939651,-0.9357133617108357) -- (15.874488945582407,-1.9878291018566911) -- (17.53024137609156,-1.2242132477508882) -- (17.696806862957963,0.5915183465007701) -- cycle;\fill[line width=1.2pt,fill=black,fill opacity=0.1] (16.207619919315206,1.6436340866466264) -- (16.66696151827849,3.792498395918885) -- (18.044669382191802,2.5636411297955535) -- (17.696806862957963,0.5915183465007701) -- cycle;\fill[line width=1.2pt,fill=black,fill opacity=0.1] (18.203804642153656,8.186388732099104) -- (19.899243769657847,7.429276224310104) -- (21.143220057086168,8.807774022944097) -- (20.216600556411592,10.416845023705811) -- (18.39994192292793,10.03280779385437) -- cycle;\fill[line width=1.2pt,fill=black,fill opacity=0.1] (16.540976388147158,10.650696974996007) -- (18.39994192292793,10.032807793854376) -- (20.216600556411596,10.416845023705816) -- (19.074806460940245,12.65651966607314) -- (17.110852776168304,12.39644663399559) -- cycle;


\draw [line width=1.2pt] (8.540164171934249,5.207693943663967)-- (8.759650525696987,7.038976272822785);

\draw [line width=1.2pt] (8.759650525696987,7.038976272822785)-- (7.283456684025284,8.144698195544743);

\draw [line width=1.2pt] (7.283456684025284,8.144698195544743)-- (5.587776488590842,7.419137789107884);

\draw [line width=1.2pt] (5.587776488590842,7.419137789107884)-- (5.368290134828104,5.5878554599490675);

\draw [line width=1.2pt] (5.368290134828104,5.5878554599490675)-- (6.8444839764998076,4.482133537227108);

\draw [line width=1.2pt] (6.63,2.6154028259781286)-- (8.12187,1.50606);

\draw [line width=1.2pt] (8.12187,1.50606)-- (9.82852406880308,2.243385906154826);\draw [line width=1.2pt] (9.82852406880308,2.243385906154826)-- (10.04330813760616,4.090054638287781);

\draw [line width=1.2pt] (10.04330813760616,4.090054638287781)-- (8.55143813760616,5.19939746426591);

\draw [line width=1.2pt] (8.55143813760616,5.19939746426591)-- (6.844784068803079,4.462071558111085);

\draw [line width=1.2pt] (6.844784068803079,4.462071558111085)-- (6.63,2.6154028259781286);\draw [line width=1.2pt] (5.598173746379701,-0.4455943008213786)-- (7.598173746379704,-0.4455943008213786);

\draw [line width=1.2pt] (7.598173746379704,-0.4455943008213786)-- (8.216207735129599,1.4565187317689308);

\draw [line width=1.2pt] (8.216207735129599,1.4565187317689308)-- (6.598173746379702,2.6320892363538784);\draw [line width=1.2pt] (6.598173746379702,2.6320892363538784)-- (4.980139757629805,1.456518731768931);\draw [line width=1.2pt] (4.980139757629805,1.456518731768931)-- (5.598173746379701,-0.4455943008213786);\draw [line width=1.2pt] (10.050247911842865,4.127366853520441)-- (9.85513422827605,2.2738815469191453);\draw [line width=1.2pt] (9.85513422827605,2.2738815469191453)-- (11.362742747550556,1.1781654870236762);\draw [line width=1.2pt] (11.362742747550556,1.1781654870236762)-- (13.065464950391876,1.935934733729505);\draw [line width=1.2pt] (13.065464950391876,1.935934733729505)-- (13.26057863395869,3.7894200403308016);\draw [line width=1.2pt] (13.26057863395869,3.7894200403308016)-- (11.752970114684185,4.885136100226271);\draw [line width=1.2pt] (11.752970114684185,4.885136100226271)-- (10.050247911842865,4.127366853520441);\draw [line width=1.2pt] (13.050946924361684,1.935934733729505)-- (11.395194493852529,1.172318879623702);\draw [line width=1.2pt] (11.395194493852529,1.172318879623702)-- (11.228629006986125,-0.6434127146279552);\draw [line width=1.2pt] (11.228629006986125,-0.6434127146279552)-- (12.71781595062888,-1.695528454773812);\draw [line width=1.2pt] (12.71781595062888,-1.695528454773812)-- (14.373568381138035,-0.9319126006680107);\draw [line width=1.2pt] (14.373568381138035,-0.9319126006680107)-- (14.540133868004439,0.8838189935836472);\draw [line width=1.2pt] (14.540133868004439,0.8838189935836472)-- (13.050946924361684,1.935934733729505);\draw [line width=1.2pt] (17.895818784401822,5.170206259832203)-- (16.66696151827849,3.792498395918885);\draw [line width=1.2pt] (16.66696151827849,3.792498395918885)-- (18.044669382191802,2.5636411297955535);\draw [line width=1.2pt] (18.044669382191802,2.5636411297955535)-- (19.273526648315137,3.941348993708868);\draw [line width=1.2pt] (19.273526648315137,3.941348993708868)-- (17.895818784401822,5.170206259832203);

\draw [line width=1.2pt] (12.70437094245782,-1.667954364772842)-- (11.241667183450124,-0.6360378181410664);
\draw [line width=1.2pt] (11.241667183450124,-0.6360378181410664)-- (9.808255908233686,-1.7082720098688238);
\draw [line width=1.2pt] (9.808255908233686,-1.7082720098688238)-- (10.385062779300293,-3.4028657308881254);\draw [line width=1.2pt] (10.385062779300293,-3.4028657308881254)-- (12.174960305780372,-3.3779480558724533);\draw [line width=1.2pt] (12.174960305780372,-3.3779480558724533)-- (12.70437094245782,-1.667954364772842);

\draw [line width=1.2pt] (11.777479270649964,4.876648796397846)-- (13.275415522675573,3.795298365941491);

\draw [line width=1.2pt] (13.275415522675573,3.795298365941491)-- (14.960860591856818,4.55187399821714);

\draw [line width=1.2pt] (15.148369409012455,6.389800060949143)-- (13.650433156986848,7.471150491405497);

\draw [line width=1.2pt] (13.650433156986848,7.471150491405497)-- (11.964988087805603,6.714574859129851);\draw [line width=1.2pt] (11.964988087805603,6.714574859129851)-- (11.777479270649964,4.876648796397846);\draw [line width=1.2pt] (15.152216036590277,6.396871128250345)-- (14.964781801168915,4.54427519204534);

\draw [line width=1.2pt] (14.964781801168915,4.54427519204534)-- (16.668784874285734,3.7935310130737405);\draw [line width=1.2pt] (16.668784874285734,3.7935310130737405)-- (17.90935092582756,5.182141529818166);\draw [line width=1.2pt] (17.90935092582756,5.182141529818166)-- (16.972059837852846,6.791094205273374);\draw [line width=1.2pt] (16.972059837852846,6.791094205273374)-- (15.152216036590277,6.396871128250345);\draw [line width=1.2pt] (11.9680222329796,6.718436663804776)-- (13.64725180938035,7.462189631895138);\draw [line width=1.2pt] (13.64725180938035,7.462189631895138)-- (13.458811179129027,9.289064169643352);\draw [line width=1.2pt] (13.458811179129027,9.289064169643352)-- (11.663118888371507,9.67438175906314);\draw [line width=1.2pt] (11.663118888371507,9.67438175906314)-- (10.741760649598524,8.08564658803953);\draw [line width=1.2pt] (10.741760649598524,8.08564658803953)-- (11.9680222329796,6.718436663804776);

\draw [line width=1.2pt] (14.139482089391567,12.400521393314591)-- (14.704568310544015,10.653215314751126);\draw [line width=1.2pt] (14.704568310544015,10.653215314751126)-- (16.540976388147158,10.650696974996007);\draw [line width=1.2pt] (16.540976388147158,10.650696974996007)-- (17.110852776168304,12.39644663399559);\draw [line width=1.2pt] (17.110852776168304,12.39644663399559)-- (15.626647675748258,13.47789759886099);\draw [line width=1.2pt] (15.626647675748258,13.47789759886099)-- (14.139482089391567,12.400521393314591);\draw [line width=1.2pt] (16.96413900558322,6.803713622705737)-- (17.9,5.2);\draw [line width=1.2pt] (17.9,5.2)-- (19.714419242793017,5.594481933559855);\draw [line width=1.2pt] (19.714419242793017,5.594481933559855)-- (19.899931010264172,7.4419987991533585);\draw [line width=1.2pt] (19.899931010264172,7.4419987991533585)-- (18.200164345081397,8.18934508331896);\draw [line width=1.2pt] (18.200164345081397,8.18934508331896)-- (16.96413900558322,6.803713622705737);\draw [line width=1.2pt] (11.664272658025446,9.680500988194913)-- (13.46271487127929,9.29430918661199);\draw [line width=1.2pt] (13.46271487127929,9.29430918661199)-- (14.696387888810303,10.65870992973667);\draw [line width=1.2pt] (14.696387888810303,10.65870992973667)-- (14.131618693087471,12.409302474444273);\draw [line width=1.2pt] (14.131618693087471,12.409302474444273)-- (12.333176479833627,12.795494276027195);\draw [line width=1.2pt] (12.333176479833627,12.795494276027195)-- (11.099503462302614,11.431093532902516);\draw [line width=1.2pt] (11.099503462302614,11.431093532902516)-- (11.664272658025446,9.680500988194913);\draw [line width=1.2pt] (16.207619919315206,1.6436340866466264)-- (14.551867488806051,0.8800182325408226);\draw [line width=1.2pt] (14.551867488806051,0.8800182325408226)-- (14.385302001939651,-0.9357133617108357);\draw [line width=1.2pt] (14.385302001939651,-0.9357133617108357)-- (15.874488945582407,-1.9878291018566911);\draw [line width=1.2pt] (15.874488945582407,-1.9878291018566911)-- (17.53024137609156,-1.2242132477508882);\draw [line width=1.2pt] (17.53024137609156,-1.2242132477508882)-- (17.696806862957963,0.5915183465007701);\draw [line width=1.2pt] (17.696806862957963,0.5915183465007701)-- (16.207619919315206,1.6436340866466264);\draw [line width=1.2pt] (16.207619919315206,1.6436340866466264)-- (16.66696151827849,3.792498395918885);\draw [line width=1.2pt] (16.66696151827849,3.792498395918885)-- (18.044669382191802,2.5636411297955535);\draw [line width=1.2pt] (18.044669382191802,2.5636411297955535)-- (17.696806862957963,0.5915183465007701);\draw [line width=1.2pt] (17.696806862957963,0.5915183465007701)-- (16.207619919315206,1.6436340866466264);\draw [line width=1.2pt] (18.203804642153656,8.186388732099104)-- (19.899243769657847,7.429276224310104);\draw [line width=1.2pt] (19.899243769657847,7.429276224310104)-- (21.143220057086168,8.807774022944097);

\draw [line width=1.2pt] (21.143220057086168,8.807774022944097)-- (20.216600556411592,10.416845023705811);
\draw [line width=1.2pt] (20.216600556411592,10.416845023705811)-- (18.39994192292793,10.03280779385437);
\draw [line width=1.2pt] (18.39994192292793,10.03280779385437)-- (18.203804642153656,8.186388732099104);
\draw [line width=1.2pt] (16.540976388147158,10.650696974996007)-- (18.39994192292793,10.032807793854376);
\draw [line width=1.2pt] (18.39994192292793,10.032807793854376)-- (20.216600556411596,10.416845023705816);
\draw [line width=1.2pt] (20.216600556411596,10.416845023705816)-- (19.074806460940245,12.65651966607314);
\draw [line width=1.2pt] (19.074806460940245,12.65651966607314)-- (17.110852776168304,12.39644663399559);\draw [line width=1.2pt] (17.110852776168304,12.39644663399559)-- (16.540976388147158,10.650696974996007);

\draw [line width=1.2pt,color=ccqqqq] (14.964781801168915,4.54427519204534)-- (15.152216036590277,6.396871128250345); 

\draw [line width=1.2pt,color=ccqqqq] (11.752970114684185,4.885136100226271)-- (13.26057863395869,3.7894200403308016); 




\draw [line width=1.2pt,color=qqwuqq] (27.487276811409444,2.3318362043201626)-- (28.892560717768973,-0.6141190356028785)-- (32.0153658764665,-1.1053468133755144)-- (33.539762429840906,1.0410013616603986)-- (33.94518928914475,3.175352392928884)-- (32.13528205165214,4.5642072494734744)-- (34.471504765329726,5.877105170678387)-- (35.43277763559661,8.205253277566841)-- (34.2258908764434,10.824470646817074)-- (31.297626802465054,11.51788965685271)-- (28.611859130470318,10.403418265869101)-- (28.050455955873012,7.456051599233286)-- (29.39844157549028,4.809821138359791)-- (27.489052781275703,2.3332476310329353)-- (24.260984527341186,2.9648262024548973); 


\draw [line width=1.2pt,color=ccqqqq] (29.39844157549028,4.809821138359791)-- (27.487276811409444,2.3318362043201626); 

\draw [line width=1.2pt,color=ccqqqq] (29.39844157549028,4.809821138359791)-- (32.13528205165214,4.5642072494734744); 

\draw [line width=1.2pt,color=qqwuqq] (24.260984527341186,2.9648262024548973)-- (22.773477670479807,5.511567889774639);

\draw [line width=1.2pt,color=qqwuqq] (24.260984527341186,2.9648262024548973)-- (22.235229538886344,0.791545812724327);

\draw [line width=1.2pt,color=qqwuqq] (28.892560717768973,-0.6141190356028785)-- (26.044370162470862,-2.189520762254815);


\begin{normalsize}
\draw [fill=black] (6.8444839764998076,4.482133537227108) circle (0.5pt);\draw [fill=black] (8.540164171934249,5.207693943663967) circle (0.5pt);\draw[color=black] (6.901864741223214,6.590570892952955) node {$F_1$};\draw [color=ttqqqq] (8.759650525696987,7.038976272822785)-- ++(-0.5pt,-0.5pt) -- ++(1pt,1pt) ++(-1pt,0) -- ++(1pt,-1pt);\draw [fill=uuuuuu] (7.283456684025284,8.144698195544743) circle (0.5pt);\draw [fill=uuuuuu] (5.587776488590842,7.419137789107884) circle (0.5pt);\draw [fill=uuuuuu] (5.368290134828104,5.5878554599490675) circle (0.5pt);\draw [fill=black] (6.63,2.6154028259781286) circle (0.5pt);\draw [fill=black] (8.12187,1.50606) circle (0.5pt);\draw[color=black] (8.215813493554665,3.726836432743462) node {$F_2$};\draw [color=ttqqqq] (9.82852406880308,2.243385906154826)-- ++(-0.5pt,-0.5pt) -- ++(1pt,1pt) ++(-1pt,0) -- ++(1pt,-1pt);\draw [fill=uuuuuu] (10.04330813760616,4.090054638287781) circle (0.5pt);\draw [fill=uuuuuu] (8.55143813760616,5.19939746426591) circle (0.5pt);\draw [fill=uuuuuu] (6.844784068803079,4.462071558111085) circle (0.5pt);\draw [fill=black] (5.598173746379701,-0.4455943008213786) circle (0.5pt);\draw [fill=black] (7.598173746379704,-0.4455943008213786) circle (0.5pt);\draw[color=black] (6.228044868232726,1.3010848899777738) node {$F_3$};\draw [fill=black] (8.216207735129599,1.4565187317689308) circle (0.5pt);\draw [fill=black] (6.598173746379702,2.6320892363538784) circle (0.5pt);\draw [fill=black] (4.980139757629805,1.456518731768931) circle (0.5pt);\draw [fill=black] (10.050247911842865,4.127366853520441) circle (0.5pt);\draw [fill=black] (9.85513422827605,2.2738815469191453) circle (0.5pt);\draw[color=black] (11.416457890259482,3.3225445089491807) node {$F_4$};\draw [fill=black] (11.362742747550556,1.1781654870236762) circle (0.5pt);\draw [fill=black] (13.065464950391876,1.935934733729505) circle (0.5pt);\draw [fill=black] (13.26057863395869,3.7894200403308016) circle (0.5pt);\draw [fill=black] (11.752970114684185,4.885136100226271) circle (0.5pt);\draw [fill=black] (13.050946924361684,1.935934733729505) circle (0.5pt);\draw [fill=black] (11.395194493852529,1.172318879623702) circle (0.5pt);\draw[color=black] (12.764097636240457,0.5261920360387345) node {$F_{5}$};\draw [fill=black] (11.228629006986125,-0.6434127146279552) circle (0.5pt);\draw [fill=black] (12.71781595062888,-1.695528454773812) circle (0.5pt);\draw [fill=black] (14.373568381138035,-0.9319126006680107) circle (0.5pt);\draw [fill=black] (14.540133868004439,0.8838189935836472) circle (0.5pt);\draw [fill=black] (17.895818784401822,5.170206259832203) circle (0.5pt);\draw [fill=black] (16.66696151827849,3.792498395918885) circle (0.5pt);
\draw[color=black] (18.18834761381388,4.00) node {$F_9$};
\draw [fill=uuuuuu] (18.044669382191802,2.5636411297955535) circle (0.5pt);\draw [fill=uuuuuu] (19.273526648315137,3.941348993708868) circle (0.5pt);\draw [fill=black] (12.70437094245782,-1.667954364772842) circle (0.5pt);\draw [fill=black] (11.241667183450124,-0.6360378181410664) circle (0.5pt);\draw[color=black] (11.180620934712811,-1.8995595067269535) node {$F_6$};\draw [fill=black] (9.808255908233686,-1.7082720098688238) circle (0.5pt);\draw [fill=black] (10.385062779300293,-3.4028657308881254) circle (0.5pt);\draw [fill=black] (12.174960305780372,-3.3779480558724533) circle (0.5pt);\draw [fill=black] (11.777479270649964,4.876648796397846) circle (0.5pt);\draw [fill=black] (13.275415522675573,3.795298365941491) circle (0.5pt);\draw[color=black] (13.505299496529995,6.051514994560579) node {$F_{17}$};\draw [fill=black] (14.960860591856818,4.55187399821714) circle (0.5pt);\draw [fill=black] (15.148369409012455,6.389800060949143) circle (0.5pt);\draw [fill=black] (13.650433156986848,7.471150491405497) circle (0.5pt);\draw [fill=black] (11.964988087805603,6.714574859129851) circle (0.5pt);\draw [fill=black] (15.152216036590277,6.396871128250345) circle (0.5pt);\draw [fill=black] (14.964781801168915,4.54427519204534) circle (0.5pt);\draw[color=black] (16.369033956739564,5.714605058065345) node {$F_{10}$};\draw [fill=black] (16.668784874285734,3.7935310130737405) circle (0.5pt);\draw [fill=black] (17.90935092582756,5.182141529818166) circle (0.5pt);\draw [fill=black] (16.972059837852846,6.791094205273374) circle (0.5pt);\draw [fill=black] (11.9680222329796,6.718436663804776) circle (0.5pt);\draw [fill=black] (13.64725180938035,7.462189631895138) circle (0.5pt);\draw[color=black] (12.258732731497592,8.645721505573885) node {$F_{16}$};\draw [fill=black] (13.458811179129027,9.289064169643352) circle (0.5pt);\draw [fill=black] (11.663118888371507,9.67438175906314) circle (0.5pt);\draw [fill=black] (10.741760649598524,8.08564658803953) circle (0.5pt);\draw [fill=black] (11.670204249892139,9.672576095782206) circle (0.5pt);\draw [fill=black] (13.471789282932798,9.29588487622065) circle (0.5pt);\draw [fill=black] (14.139482089391567,12.400521393314591) circle (0.5pt);\draw [fill=black] (14.704568310544015,10.653215314751126) circle (0.5pt);

\draw[color=black] (15.425686134552885,12.022331737166222) node {$F_{14}$};

\draw [fill=black] (16.540976388147158,10.650696974996007) circle (0.5pt);\draw [fill=black] (17.110852776168304,12.39644663399559) circle (0.5pt);\draw [fill=black] (15.626647675748258,13.47789759886099) circle (0.5pt);\draw [fill=black] (16.96413900558322,6.803713622705737) circle (0.5pt);\draw [fill=black] (17.9,5.2) circle (0.5pt);\draw[color=black] (18.86216748680437,6.725334867551048) node {$F_{11}$};\draw [fill=black] (19.714419242793017,5.594481933559855) circle (0.5pt);\draw [fill=black] (19.899931010264172,7.4419987991533585) circle (0.5pt);\draw [fill=black] (18.200164345081397,8.18934508331896) circle (0.5pt);\draw [fill=black] (11.664272658025446,9.680500988194913) circle (0.5pt);\draw [fill=black] (13.46271487127929,9.29430918661199) circle (0.5pt);\draw[color=black] (12.663024655291885,11.408382984834807) node {$F_{15}$};\draw [color=ttqqqq] (14.696387888810303,10.65870992973667)-- ++(-0.5pt,-0.5pt) -- ++(1pt,1pt) ++(-1pt,0) -- ++(1pt,-1pt);\draw [fill=uuuuuu] (14.131618693087471,12.409302474444273) circle (0.5pt);\draw [fill=uuuuuu] (12.333176479833627,12.795494276027195) circle (0.5pt);\draw [fill=uuuuuu] (11.099503462302614,11.431093532902516) circle (0.5pt);\draw [fill=black] (16.207619919315206,1.6436340866466264) circle (0.5pt);\draw [fill=black] (14.551867488806051,0.8800182325408226) circle (0.5pt);\draw[color=black] (15.931051039295749,0.22297309319302355) node {$F_7$};\draw [fill=black] (14.385302001939651,-0.9357133617108357) circle (0.5pt);\draw [fill=black] (15.874488945582407,-1.9878291018566911) circle (0.5pt);\draw [fill=black] (17.53024137609156,-1.2242132477508882) circle (0.5pt);\draw [fill=black] (17.696806862957963,0.5915183465007701) circle (0.5pt);

\draw[color=black] (16.8797637662253,2.0487246359587115) node {$F_8$};

\draw [fill=black] (18.203804642153656,8.186388732099104) circle (0.5pt);\draw [fill=black] (19.899243769657847,7.429276224310104) circle (0.5pt);\draw[color=black] (19.906588289939627,9.487996346811972) node {$F_{12}$};\draw [fill=black] (18.39994192292793,10.032807793854376) circle (0.5pt);\draw [fill=black] (20.216600556411596,10.416845023705816) circle (0.5pt);\draw [fill=black] (19.074806460940245,12.65651966607314) circle (0.5pt);\draw[color=black] (18.592639537608175,11.745292921330043) node {$F_{13}$};\draw [fill=black] (27.487276811409444,2.3318362043201626) circle (2.5pt);\draw [fill=black] (24.260984527341186,2.9648262024548973) circle (2.5pt);\draw[color=black] (24.62332740087304,3.7942184200425086) node {$v_2$};\draw [fill=black] (27.489052781275703,2.3332476310329353) circle (2.5pt);

\draw[color=black] (27.357662791227384,3.1540895407015634) node {$v_4$};

\draw [fill=black] (28.892560717768973,-0.6141190356028785) circle (2.5pt);\draw[color=black] (29.272684524507408,0.22297309319302355) node {$v_5$};\draw [fill=black] (32.0153658764665,-1.1053468133755144) circle (2.5pt);

\draw[color=black] (32.002255940263656,-0.2823918115498281) node {$v_7$};

\draw [fill=black] (33.539762429840906,1.0410013616603986) circle (2.5pt);

\draw[color=black] (34.2,1.5) node {$v_8$};

\draw [fill=black] (33.94518928914475,3.175352392928884) circle (2.5pt);\draw[color=black] (34.32633357193607,3.9963643819396495) node {$v_9$};\draw [fill=black] (32.13528205165214,4.5642072494734744) circle (2.5pt);

\draw[color=black] (32.30809289581032,5.3776951215701105) node {$v_{10}$};

\draw [fill=black] (34.471504765329726,5.877105170678387) circle (2.5pt);

\draw[color=black] (35.5,6.5) node {$v_{11}$};

\draw [fill=black] (35.43277763559661,8.205253277566841) circle (2.5pt);\draw[color=black] (35.90981027346371,9.016322435718644) node {$v_{12}$};\draw [fill=black] (34.2258908764434,10.824470646817074) circle (2.5pt);\draw[color=black] (34.696934502080836,11.644219940381472) node {$v_{13}$};\draw [fill=black] (31.297626802465054,11.51788965685271) circle (2.5pt);\draw[color=black] (31.76581805457221,12.351730807021465) node {$v_{14}$};\draw [fill=black] (28.611859130470318,10.403418265869101) circle (2.5pt);\draw[color=black] (29.07053856261026,11.239928016587191) node {$v_{15}$};\draw [fill=black] (28.050455955873012,7.456051599233286) circle (2.5pt);

\draw[color=black] (27.53148266421787,8.175120575429128) node {$v_{16}$};

\draw [fill=black] (29.39844157549028,4.809821138359791) circle (2.5pt);\draw[color=black] (29.879122410198846,5.647223070766298) node {$v_{17}$};\draw [fill=black] (22.773477670479807,5.511567889774639) circle (2.5pt);\draw[color=black] (23.14092368029397,6.321042943756766) node {$v_1$};\draw [fill=black] (22.235229538886344,0.791545812724327) circle (2.5pt);

\draw[color=black] (22.00186778190158,1.6043038328234847) node {$v_3$};

\draw [fill=black] (26.044370162470862,-2.189520762254815) circle (2.5pt);\draw[color=black] (26.408950064297834,-1.3605036083345783) node {$v_6$};
\end{normalsize}
\end{tikzpicture}

%% file: bad_edges.tikz
\definecolor{ccqqqq}{rgb}{0.6,0,0}\begin{tikzpicture}[line cap=round,line join=round,>=triangle 45,x=1cm,y=1cm,scale=0.8]
\clip(-1.6191988010203326,-1.183314694699264) rectangle (10.058071017791868,5.692073412577307);
\draw [rotate around={0.6182757795167048:(4.14,2.11)},line width=1.2pt] (4.14,2.11) ellipse (3.501093220102061cm and 2.127992889049359cm);\draw [->,line width=1.2pt,color=ccqqqq] (7.488097885108166,2.7472140408110732) -- (8.600120581948126,2.7393573291239957);\draw [->,line width=1.2pt,color=ccqqqq] (4.724452461293627,0.01742294867507162) -- (4.8460539703472305,-1.0798086934175564);\draw [->,line width=1.2pt,color=ccqqqq] (1.5954637544941253,3.5701970260400935) -- (0.9321506563160604,4.2003309897385);\draw [->,line width=1.2pt,color=ccqqqq] (3.867532361617296,4.242647824560858) -- (3.842887123120313,5.460839150501987);\draw [line width=1.2pt,color=ccqqqq] (8.59950278530242,2.7398424581463576)-- (9.503025670086753,2.7309622949931303);\draw [line width=1.2pt,color=ccqqqq] (3.8731259845199135,4.248416299195301)-- (3.8256763635236863,6.271788234035418);\draw [line width=1.2pt,color=ccqqqq] (1.6056223472394953,3.5692302021891558)-- (0.28852959228697656,4.851947588036795);\draw [line width=1.2pt,color=ccqqqq] (4.7202940191223055,0.026055580102842955)-- (4.903054541268602,-1.5583196751680441);\draw [line width=1.2pt,color=ccqqqq] (1.832077735466305,0.5027881590830325)-- (2.7858474845282264,2.00497551385557);\draw [->,line width=1.2pt,color=ccqqqq] (2.783463060155572,2.00497551385557) -- (2.2612741225441697,1.177580256544347);\draw [line width=1.2pt,color=ccqqqq] (4.762701387151866,2.6067121328140526)-- (4.856831791501789,4.206929006762748);\draw [->,line width=1.2pt,color=ccqqqq] (4.856831791501789,4.2016995398544195) -- (4.801891010700238,3.226934589877934);\draw [line width=1.2pt,color=ccqqqq] (6.14327629658087,0.4043007561063583)-- (5.7921914727548565,1.7058835176076848);\draw [->,line width=1.2pt,color=ccqqqq] (5.7921914727548565,1.6973204731241234) -- (6.006267584843889,0.9180834251200398);

\draw [color=ccqqqq](3.2078037253375276,2.267867790712998) node[anchor=north west] {$z_j \equiv b-a$};

\draw (2.6,4.7) node[anchor=north west] {$\gamma$};

\draw (-1.6,0.2) node[anchor=north west] {congruent to $(a,b,c,d)$};

\draw [color=ccqqqq](6.391073623995989,0.136695570424703) node[anchor=north west] {$z_j \equiv c-a$};

\draw [line width=1.2pt] (3.237222841161298,4.326835297701799)-- (3.0785187157254197,4.130675620944725)-- (3.306071899720335,3.9911961372264946);\draw [->,line width=1.2pt] (0.802525504382432,0.22467884940583605) -- (1.3666296983590542,0.6679035732446115);\end{tikzpicture}

%% file: localpicture.tikz
\begin{tikzpicture}[line cap=round,line join=round,>=triangle 45,x=1cm,y=1cm, scale=0.9]\clip(-0.5427108057048596,-1.2342048479497174) rectangle (3.360627400533817,3.417991112626764);\draw [line width=1.6pt] (1,0)-- (1,2);\draw [line width=1.6pt] (1,2)-- (2.32,3.28);\draw [line width=1.6pt] (1,2)-- (-0.24,3.3);\draw [line width=1.6pt] (1,0)-- (2.3,-1.3);\draw [line width=1.6pt] (1,0)-- (-0.24,-1.3);\draw [->,line width=1.2pt] (1,0) -- (1,0.8495988816169101);\draw [->,line width=1.2pt] (2.3,-1.3) -- (1.65,-0.65);\draw [->,line width=1.2pt] (1,0) -- (0.2964778782996653,-0.7375635146858347);\draw [->,line width=1.2pt] (2.32,3.28) -- (1.4966398485565546,2.481590156176053);\draw [->,line width=1.2pt] (1,2) -- (0.4521502044863054,2.574358656586938);\draw [shift={(1.942990610014141,0.9220136164803816)},line width=1.6pt]  plot[domain=-0.2634370941166182:2.878155559473175,variable=\t]({1*0.3689739857759035*cos(\t r)+0*0.3689739857759035*sin(\t r)},{0*0.3689739857759035*cos(\t r)+1*0.3689739857759035*sin(\t r)});\draw [line width=2.4pt] (2.2513125684202007,0.8962191502826906)-- (2.299235237374579,0.825932569149603);\draw [line width=2.4pt] (2.299235237374579,0.825932569149603)-- (2.382301196895501,0.8866346164918151);\draw [shift={(0.06552058425463358,0.9219547833697883)},line width=1.6pt]  plot[domain=-0.2634370941166182:2.8781555594731745,variable=\t]({1*0.3689739857759036*cos(\t r)+0*0.3689739857759036*sin(\t r)},{0*0.3689739857759036*cos(\t r)+1*0.3689739857759036*sin(\t r)});\draw [line width=2.4pt] (0.37384254266069294,0.8961603171720973)-- (0.42176521161507174,0.8258737360390096);\draw [line width=2.4pt] (0.42176521161507174,0.8258737360390096)-- (0.5048311711359934,0.8865757833812218);\begin{normalsize}\draw [fill=black] (1,0) circle (0.5pt);

\draw [fill=black] (1,2) circle (0.5pt);\draw[color=black] (1.127308257066187,1.0819525426791718) node {$e$};

\draw [fill=black] (-0.24,3.3) circle (0.5pt);\draw[color=black] (0.497737578640594,2.82) node {$\bar e_1$};

\draw [fill=black] (2.32,3.28) circle (0.5pt);\draw[color=black] (1.3779483216234536,2.82) node {$\bar e_2$};

\draw [fill=black,shift={(1,0.8495988816169101)},rotate=270] (0,0) ++(0 pt,0.75pt) -- ++(0.649519052838329pt,-1.125pt)--++(-1.299038105676658pt,0 pt) -- ++(0.649519052838329pt,1.125pt);\draw [fill=black] (1.65,-0.65) circle (0.5pt);

\draw [fill=black] (2.3,-1.3) circle (0.5pt);\draw [fill=black] (-0.24,-1.3) circle (0.5pt);\draw[color=black] (0.1453152038638715,-0.48) node {$\bar e_3$};

\draw[color=black] (1.75687893549178,-0.48) node {$\bar e_4$};

\draw [fill=black] (0.2964778782996653,-0.7375635146858347) circle (0.5pt);\draw [fill=black] (1.4966398485565546,2.481590156176053) circle (0.5pt);\draw [fill=black] (0.4521502044863054,2.574358656586938) circle (0.5pt);\draw [fill=black] (1.5867459826537031,1.01809466381116) circle (0.5pt);\draw [fill=black] (2.299235237374579,0.825932569149603) circle (0.5pt);

\draw[color=black] (1.9788854378839629,0.50) node {$F_{i+1}$};\draw [fill=black] (2.2513125684202007,0.8962191502826906) circle (0.5pt);

\draw [fill=black] (2.382301196895501,0.8866346164918151) circle (0.5pt);\draw [fill=black] (-0.2907240431058046,1.018035830700567) circle (0.5pt);\draw [fill=black] (0.42176521161507174,0.8258737360390096) circle (0.5pt);

\draw[color=black] (0.07360575317493133,0.50) node {$F_j$};

\draw [fill=black] (0.37384254266069294,0.8961603171720973) circle (0.5pt);\draw [fill=black] (0.5048311711359934,0.8865757833812218) circle (0.5pt);\end{normalsize}\end{tikzpicture}